\documentclass[a4paper,oneside]{amsart}
\usepackage{amssymb}
\usepackage{hyperref}
\usepackage{xcolor}
\usepackage{setspace}
\newtheorem{thm}{Theorem}[section]
\newtheorem{lem}[thm]{Lemma}
\newtheorem{cor}[thm]{Corollary}
\newtheorem{rem}[thm]{Remark}
 
\def\Xint#1{\mathchoice
    {\XXint\displaystyle\textstyle{#1}}%
     {\XXint\textstyle\scriptstyle{#1}}%
     {\XXint\scriptstyle\scriptscriptstyle{#1}}%
     {\XXint\scriptstyle\scriptscriptstyle{#1}}%
	\!\int}
\def\XXint#1#2#3{{\setbox0=\hbox{$#1{#2#3}{\int}$}
	\vcenter{\hbox{$#2#3$}}\kern-.5\wd0}}
\newcommand{ \mint }{ \,\Xint- }

\newcommand{\Lra}{\Longrightarrow}
\newcommand{\br}{\mathbb{R}}
\newcommand{\bn}{\mathbb{N}}

\newcommand{\sskip}{\smallskip}

\newcommand{\dotcup}{ \mathop{\dot{\cup}} }

\newcommand{\bb}{\mathbb{B}}
\newcommand{\bg}{\mathbb{G}}
\newcommand{\bh}{\mathbb{H}}

\theoremstyle{definition}
\newtheorem{defn}[thm]{Definition}
\newtheorem{prop}[thm]{Proposition}

\numberwithin{equation}{section}
\makeatletter
\@namedef{subjclassname@2020}{\textup{2020} Mathematics Subject Classification}
\makeatother

\hypersetup{ colorlinks = true, urlcolor = blue, linkcolor = blue, citecolor = red }

\begin{document}
	
\title[piecewise smoothness]{Piecewise smoothness for linear elliptic systems with piecewise smooth coefficients}
\everymath{\displaystyle}

\author{Youchan Kim}
\address{Department of Mathematics, University of Seoul, Seoul 02504, Republic of Korea}
\email{youchankim@uos.ac.kr}


\maketitle

\begin{abstract}
Li and Vogelius, and Li and Nirenberg obtained piecewise
$C^{1,\gamma}$-regularity for linear elliptic problems with piecewise $C^{\gamma}$-coefficients which come
from composite materials. In this paper, we obtain piecewise $C^{m+1,\gamma}$-regularity for linear
elliptic systems with piecewise $C^{m,\gamma}$-coefficients which also come from composite
materials. This answers the open problem suggested by Li and Vogelius, and Li and Nirenberg.
\end{abstract}

\section{Introduction and main result}

\subsection{Introduction}

We consider linear elliptic systems arising from composite materials such as fiber reinforced materials. Assume that $\Omega \subset \mathbb{R}^{2}$ is an open bounded $C^{m,\gamma}$-domain, and $\{ \Omega^{k} \}_{k=1}^{k_{0}}$ are the mutually disjoint $C^{m,\gamma}$-subdomains in $\Omega$ which are allowed to touch each other. We denote $\Omega^{0} = \Omega \backslash \cup_{k=1}^{k_{0}} \Omega^{k}$. Let $u$ be a weak solution of an linear elliptic equation
\begin{equation*}
D_{\alpha} \left[  A^{\alpha \beta}_{ij} D_{\beta}u^{j} \right] =   D_{\alpha} F_{\alpha}^{i} 
\quad \text{in} \quad \Omega,
\end{equation*}
where $A^{\alpha \beta}_{ij} \in C^{m,\gamma}(\Omega^{k})$ and $F_{\alpha}^{i}  \in C^{m,\gamma}(\Omega^{k})$ for any $1 \leq \alpha,\beta \leq n$, $1 \leq i,j \leq N$ and $k \in [0,k_{0}] $. With our assumption, $A^{\alpha \beta}_{ij}$ and $F_{\alpha}^{i}$ are allowed to be discontinuous across the boundary of the subdomains $\{ \partial \Omega^{k} \}_{k=1}^{k_{0}}$. In this paper, we prove that the weak solution $u$ is locally piecewise $C^{m+1,\gamma}$, meaning that $u$ is $C^{m+1,\gamma}(\tilde{\Omega} \cap \Omega^{k})$ for any compact subset $\tilde{\Omega} \subset \subset \Omega$ and $k \in [0,k_{0}]$. 

An application of our result is the following. Suppose that the fibers and resin of a fiber reinforced plastic(FRP) is assumed to be linearly elastic, and described by a linear elliptic equation. Then due to the discontinuity of the coefficients or elasticity, the gradient of the weak solution might be discontinuous on the boundary of each materials. Nevertheless, our result shows that the weak solution on each materials(fiber, resin) can be smooth, even if the fibers are allowed to touch each other. To prove our result, we generalize the difference quotient method to curves instead of traditional fixed directions.

One of long standing problem for the regularity theory of composite material is establishing a  new regularity theory when the subdomains $\{ \Omega^{k} \}_{k=1}^{k_{0}}$ are allowed to touch each other. Under this assumption, the flattening argument could not be applied at the points where two subdomains touch each other, and the difficulty arises. With this geometric assumption, local Lipschitz regularity can be  derived as in \cite{BEVM1,LYNL1,LYVM1}. But better regularity than Lipschitz regularity is not feasible, because discontinuity of the gradient might occur on the boundary of subdomains. However, from the observations made in the previous literatures, one can consider so called piecewise regularity results.  For example, if the distance between the boundary of subdomains are bigger than some positive constant, then one may use the flattening argument to prove local piecewise $C^{m,\gamma}$-regularity. But this estimate depends on the distance between the boundary of subdomains and could not be applied for composite materials. When two subdomains are allowed to touch each other, which might happen for composite materials, the best result known in the previous literatures is piecewise $C^{1,\gamma}$-regularity by Vogelius and Li \cite{LYVM1}, and Li and Nirenberg \cite{LYNL1}. In that papers, they used observation that in small scales, the boundaries of two touching subdomains are almost flat near the touching point, and such geometric structures could be viewed as linear laminates. Indeed, \cite{LYVM1} and \cite{LYNL1} proved that the solution is $C^{1,\gamma}(\tilde{\Omega} \cap \Omega^{k})$ for any compact subset $\tilde{\Omega} \subset \Omega$ and  $k \in [0,k_{0}]$ by using the regularity theory of linear laminates such as Chipot, Kinderlehrer and Vergara-Caffarelli \cite{CMKDVU1}, and in fact their estimates does not depend on the distance between the boundary of subdomains. However, the method in \cite{LYVM1} and \cite{LYNL1} could not be extended to higher derivatives $C^{m,\gamma}$ with $m \geq 2$, and in \cite{LYNL1} they suggest an open question that piecewise $C^{m,\gamma}$-regularity even holds for $m \geq 2$. On the other-hand, regularity results for higher derivatives have been obtained for special geometric structures with dimension $2$, when the domain $\Omega$ is a disk, and two subdomains are also disks. In that case, \cite{LYVM1} proved $C^{m,\gamma}$ for $m \geq 2$ by using conformal mapping theory, but unfortunately their method could not be extended to higher dimensions. We also refer to \cite{DHZH1} for a similar result by using Green functions. With the observation in this paragraph, it has been expected that the conjecture about piecewise $C^{m,\gamma}$-regularity for $m \geq 2$ seems true, and in this paper we prove that this conjecture indeed holds. 

To prove our main result, we generalize the difference quotient method. As far as we are concerned, the difference quotient method have been applied for a fixed direction, but we generalize this method by applying to curves. In fact, we use so called a `flow' to obtain higher regularity, because the computation of difference quotient method is complicated for the higher regularity and we will use an approximation argument for obtaining higher regularity. 

If a curve does not across the boundary of subdomains, then the coefficients remain regular along that curve, even if the coefficients are discontinuous between different subdomains. So we choose a suitable flow so that if we restrict the domain of the flow to some subdomain $\Omega^{k}$, then the range of that flow still remains in the same subdomain $\Omega^{k}$. To explain our new method, let's consider our model case.
Let $\Omega = [-2,2] \times [-4,4] \subset \mathbb{R}^{2}$, $\varphi_{+}(x_{2}) = \left( x^{2} \right)^{2}$ and $\varphi_{-}(x^{2}) = - \left( x^{2} \right)^{2}$ where $x = \left( x^{1},x^{2} \right)$. Then $\varphi_{+}$ and $\varphi_{-}$ touches each other at the origin $(0,0)$. We choose a function $\sigma : [-1,1] \times [-1,1] \times [-1,1] \rightarrow [-4,4]$ such that
\begin{equation*}
\sigma(x,t) = \frac{[\varphi_{+}(x^{2}+t) - \varphi_{+}(x^{2})][x_{1} - \varphi_{-}(x^{2})]}
{\varphi_{+}(x^{2}) - \varphi_{-}(x^{2})}
+ \frac{[\varphi_{-}(x^{2}+t) - \varphi_{-}(x^{2})][\varphi_{+}(x^{2}) - x_{1} ]}
{\varphi_{+}(x^{2}) - \varphi_{-}(x^{2})}.
\end{equation*}
Then $t \mapsto x + (\sigma(x,t),t)$ can be viewed as a flow in $Q_{1}^{0} : = \{ (x_{1},x^{2}) \in [-1,1] \times [-1,1]: \varphi_{-}(x^{2}) \leq x_{1} \leq \varphi_{+}(x^{2}) \}$, and we find that $\partial_{t}[f(x+(\sigma,t))]_{t=0}$ is the first difference quotient or the first derivative along the curve $t \mapsto \sigma(x,t)$. 

However, this approach only works for piecewise $C^{2,\gamma}$-regularity. The reason is that if two subdomains touches each other, then the first derivative of two boundaries coincide at that touching point, but the higher derivatives might  not coincide at that touching point. To explain, we calculate $2$-nd derivative of $h \in C^{\infty}(\mathbb{R}^{2})$ along the curves $\{ \left( \varphi_{+}(x^{2}),x^{2} \right) : x^{2} \in [-1,1] \}$ and $\{ \left( \varphi_{-}(x^{2}),x^{2} \right) : x^{2} \in [-1,1] \}$ at the origin by using the following calculation:
\begin{equation*}\begin{aligned}
\partial_{t}^{2} [h(\varphi(t),t)]
& = \varphi'(x^{2}+t)^{2} D_{11} h(\varphi(t),t) +
2 \varphi'(x^{2}+t)^{2} D_{12} h(\varphi(t),t) \\
&\quad +  D_{22} h(\varphi(t),t) + \varphi''(x^{2}+t) D_{1} h(\varphi(t),t).
\end{aligned}\end{equation*}
Since $\varphi''_{+}(0) = 2 \not = \varphi''_{-}(0) = -2$ and $\varphi'_{+}(0) = \varphi'_{-}(0)=0$, $D_{1}h(0,0) \not =0$ implies $\partial_{t}^{2} [h(\varphi_{+}(t),t)] \big|_{t=0} \not = \partial_{t}^{2} [h(\varphi_{-}(t),t)] \big|_{t=0}$. On the other-hand, it is possible to obtain $C^{2,\gamma}$ by just using the approach in the previous paragraph because $\partial_{t} [h(\varphi_{+}(t),t)] \big|_{t=0} = \partial_{t} [h(\varphi_{-}(t),t)] \big|_{t=0}$.

By our model case, we see that there are jumps of the higher order derivatives on the boundary of subdomains. In fact, this jumps appears even if $h$ is smooth which means that the jumps of the higher order derivatives are mainly due to the geometry of the composite domains. Because of this reason, we obtained geometric property of composite domains in \cite{JYKY1,KYSP2}. In \cite{JYKY1}, the author and Jang prove that the boundary of the subdomains are almost parallel if they are sufficiently close, even if the subdomains are fractal, say Reifenberg flat domains. Also \cite{KYSP2} answers an question for the typical examples of the composite domains. If $C^{1,\gamma}$-domains satisfy certain inclusion conditions, they form a composite domains.

To overcome the difficulty for handling the jumps of the higher order derivatives on the boundary of subdomains, we extend functions defined in a subdomain. We briefly explain this idea by an example. Let $u$ be a function which might have jumps on the boundary of subdomains. Let $u^{k}$ be an extension of $u$ from $\Omega^{k}$ to $\Omega$ such that $u^{k}$ is regular in $\Omega \setminus \Omega^{k}$. Then we see that $\sum_{k} u^{k}$ has no jumps on the boundary of the subdomains. So by considering $\sum_{k} u^{k}$ instead of $u$, jumps of the function is removed and the function became easier for obtaining piecewise regularity. The value of $u$ and $\sum_{k} u^{k}$ might not be the same in $\Omega$, but $ u - \sum_{k} u^{k}$ is piecewise regular because $u^{k}$ is an extension of $u$ in $\Omega^{k}$ and defined to be regular in $\Omega \setminus \Omega^{k}$. If a suitable perturbations could be done to $\sum_{k} u^{k}$ in each subdomains, then one can obtain piecewise regularity of $ u $. For partial differential equations, the functions can be differentiated. But the ideas in the previous paragraphs still work for our problem.

The results and the key idea of this paper was announced at `International Conference on
Partial Differential Equations Related to Material Science', May 6-9, 2021. 

\subsection{Main result}

We introduce the following notations in this paper.
\begin{enumerate}
\item A $(n-1)$-dimensional point is denoted as $x' = (x^{2}, \cdots, x^{n})$.

\item $Q_{r}'(y') = \left \{ x' = (x^{2}, \cdots , x^{n})  \in \mathbb{R}^{n-1} : \max_{2 \leq i \leq n} |x^{i}-y^{i}| < r \right \}$ is the open cube in $\mathbb{R}^{n-1}$ with center $y'$ and size $r$. Also we denote $Q_{r}' = Q_{r}'(0')$.
\item $Q_{r}(y)  =  \left \{ x \in \mathbb{R}^{n} : \max_{1 \leq i \leq n} |x^{\alpha}-y^{\alpha}| < r \right \} = (y^{1}-r,y^{1}+r) \times Q_{r}'(y')$ is the open cube in $\mathbb{R}^{n}$  with center $y$ and size $r$.  Also we denote $Q_{r} = Q_{r}(0)$.
\item For a function $g(x)$ in $\mathbb{R}^n$, 
$$
(g)_{U} = \mint_U g(x) \, dx = \frac{1}{|U|} \int_{U} g(x) \, dx,
$$
where $U$ is an open subset in $\mathbb{R}^n$ and $|U|$ is the $n$-dimensional Lebesgue measure of $U$.

\item $h \big|_{U}$ is the restriction of $h$ in $U$.

\item For an open set $U \subset \br^{n}$, $h \in C^{l} \left( \overline{U} \right)$ means that $Du, \cdots, D^{l}u$ exist in $U$ and are uniformly continuous in $U$.
\end{enumerate}

To handle the index in higher order derivatives, we use the following notations.
\begin{enumerate}

\item $p=(p_{1},\cdots,p_{n}) \in  \br^{n}$ denotes $n$-dimensional indices satisfying that $p_{1}, \cdots, p_{n} \geq 0$.

\item $p'=(p_{2},\cdots,p_{n}) \in  \br^{n-1}$ and $q'=(q_{2},\cdots,q_{n}) \in  \br^{n-1}$ denotes $(n-1)$-dimensional indices satisfying that $p_{2}, \cdots, p_{n} \geq 0$ and $q_{2} , \cdots, q_{n} \geq 0$.

\item For any $p'=(p_{2},\cdots,p_{n})$ and $q'=(q_{2},\cdots,q_{n}) $, we write
\begin{equation*}
p' \leq q'
\quad \text{ if } \quad
p_{i} \leq q_{i}
\quad \text{ for all } 
i \in \{ 2, \cdots, n\},
\end{equation*}
and
\begin{equation*}
p' < q'
\quad \text{ if } \quad
p' \leq q'
\quad \text{ and } \quad
p' \not = q'.
\end{equation*}
\item For any $p'=(p_{2},\cdots,p_{n})$ and $q'=(q_{2},\cdots,q_{n})$ with $p' \geq  q' \geq 0'$, we define
\begin{equation*}
\Big( \begin{array}{c} p' \\ q' \end{array} \Big)
=  \Big( \begin{array}{c} p_{2} \\ q_{2} \end{array} \Big) \cdots  \Big( \begin{array}{c} p_{2} \\ q_{2} \end{array} \Big).
\end{equation*}
\item For any $p_{1} \geq 0$ and $p'=(p_{2},\cdots,p_{n})$ with $p' \geq 0'$, we define
\begin{equation*}
D^{p_{1},p'} = D_{1}^{p_{1}} \cdots D_{n}^{p_{n}},
\quad D^{p'} = D_{2}^{p_{2}} \cdots D_{n}^{p_{n}}
\quad \text{and} \quad 
|p'| = p_{2} + \cdots + p_{n}.
\end{equation*}
\end{enumerate}

We introduce the notations for representing the subregions and the boundary of the subregions. In composite materials,  there exists a coordinate system such that the boundaries of the subregions become almost flat graphs in sufficiently small scale, see for instance \cite{KYSP1}.  So from now on, we assume that the boundaries of the regions are represented by the graphs in $Q_{7}$. We assume that the cube $Q_{7}$ is divided into the components or the subregions by using $C^{1,\gamma}$-graph functions. In $Q_{7}$, the component or the subregions will be denoted as $Q_{7}^{k}(z)$ with the set $K = \{ k_{-}, \cdots, k_{+} \}$. Also let  $ K_{+}  = \{ k_{-}+1, \cdots, k_{+} \}$ be the $C^{1,\gamma}$-graph functions. We remark that there can be only one element in $K$.  To  coincide with the previous definition in \cite{KYSP1}, we define composite domain by using $C^{1,\gamma}$-graph functions as follows.

\begin{defn}\label{composite cube}
For the sets $K = \{ k_{-}, k_{-}+1, \cdots, k_{+} \}$, $K_{+} = K \cup \left\{ k_{+} +1 \right \}$ and the graph functions $\varphi_{k} \in C^{1,\gamma} \left( Q_{r}'(z') \right)$ $(k \in K_{+})$, $\left( Q_{r}(z) , \left\{ \varphi_{k} : k \in K_{+} \right\} \right)$ is called a composite cube if
\begin{equation*}\label{}
\varphi_{k}(x')  \leq  \varphi_{k+1}(x')
\qquad \left(  x \in Q_{r}'(z), \, k \in K \right),
\end{equation*}
and
\begin{equation*}\label{}
Q_{r}(z) = \dotcup_{k \in K} Q_{r}^{k}(z),
\end{equation*}
where
\begin{equation}\label{composite cube_strip} 
Q_{r}^{k}(y) : = \big\{ (x^{1},x') \in Q_{\rho}(y) : \varphi_{k}(x') < x^{1} \leq \varphi_{k+1}(x') \big \}
\qquad ( k \in K ).
\end{equation}
Here, $\dotcup_{k} U_{k} $ denotes disjoint union meaning that $\dotcup_{k} U_{k} $ is the union of the sets $\{ U_{k} : k\in K \}$ and that $\{ U_{k} : k\in K \}$ are mutually disjoint.
\end{defn}

\begin{rem}
For the graphs $ \left\{ (\varphi_{k}(x'),x') : x' \in Q_{r}'(z') \right\}$ $(k \in K_{+})$ in Definition \ref{composite cube}, the top  $ \left\{ (\varphi_{k_{+}+1}(x'),x') : x' \in Q_{r}'(z') \right\}$  and the bottom  $ \left\{ (\varphi_{k_{-}}(x'),x') : x' \in Q_{r}'(z') \right\}$  ones should be outside of $Q_{r}(z)$, because other-wise we have that $Q_{r}(z) \not = \dotcup_{k \in K} Q_{r}^{k}(z)$.
\end{rem}

For the composite cubes inside the cube, we use the following natural definition.

\begin{defn}
For the composite cube $\left( Q_{r}(z) , \left\{ \varphi_{k} : k \in K_{+} \right\} \right)$, we denote
\begin{equation*}\label{} 
Q_{\rho}^{k}(y) : = \big\{ (x^{1},x') \in Q_{r}(y) : \varphi_{k}(x') < x^{1} \leq \varphi_{k+1}(x') \big \}
\qquad \qquad ( k \in K ),
\end{equation*}
for any $Q_{\rho}(y) \subset Q_{r}(z)$.
\end{defn}

In view of \cite{KYSP1}, there exists a coordinate system such that the graphs are almost flat. So we also assume that
\begin{equation}\label{gradient_graphs}
\| D_{x'}\varphi_{k} \|_{L^{\infty}(Q_{7}')} \leq \frac{1}{20n}
\qquad \big( k \in  K_{+}  \big).
\end{equation}

For a technical reason, we will focus on the set $Q_{5}$. Let  $ K_{-}  (\subset K)$ be the set of  graph functions which intersects $Q_{5}$. Then we see that
\begin{equation*}\label{}
 K_{-}  \subsetneq  K \subsetneq  K_{+} .
\end{equation*}
From  the definition of $K_{-}$, we discover that
\begin{equation}\label{GS180}
k \in K_{-}
\qquad \Longleftrightarrow \qquad
Q_{5} \cap \left \{  \big( \varphi_{k}(x'), x' \big) : x' \in Q_{5}' \right\} \not = \emptyset,
\end{equation}
which implies that
\begin{equation}\label{GS185}
k \in K_{-}  
\qquad \Longleftrightarrow \qquad
Q_{5}^{k} \not = \emptyset 
\quad \text{and} \quad
Q_{5}^{k-1} \not = \emptyset.
\end{equation}
It follows from \eqref{gradient_graphs} and \eqref{GS180} that
\begin{equation}\label{GS190}
\left \{  \big( \varphi_{k}(x'), x' \big) : x' \in Q_{5}' \right\} 
\subset \partial Q_{6}^{k} \cap \partial Q_{6}^{k-1} 
(\subset Q_{6})
\qquad \big( k \in  K_{-}  \big).
\end{equation}

\begin{rem}
With the assumption that  $a_{ij}, u, F_{i} \in C^{m,\gamma} \left ( Q_{7}^{k} \right )$ $(k \in K)$, we will prove that $u \in C^{m+1,\gamma} \left( Q _{3}^{k} \right )$ $(k \in K)$. The most of the calculations in this paper are performed in $Q_{5}$ but $Q_{5}^{k}$ $(k \in K)$ might not be disconnected. So with \eqref{GS190}, we consider an extended domain $Q_{7}$ to overcome this difficulty.
\end{rem}

For $A^{\alpha \beta,k}_{ij} : \br^{n} \to \br$ $(1 \leq \alpha,\beta \leq n,  \ 1 \leq i,j \leq N, \ k \in K)$, assume that
\begin{equation}\label{ell_com}
\lambda |\zeta|^{2} 
\leq \sum_{1 \leq i,j \leq N} \sum_{1 \leq \alpha, \beta \leq n} A^{\alpha \beta,k}_{ij}(x) \zeta_{\alpha}^{i} \zeta_{\beta}^{j}
\qquad \qquad 
(x \in \br^{n}, ~ \zeta \in \br^{Nn}, ~ k \in K),
\end{equation}
\begin{equation}\label{growth_com}
\left| A^{\alpha \beta,k}_{ij}(x) \right|
\leq \Lambda 
\qquad \qquad 
\left( x \in \br^{n}, ~ 1 \leq \alpha,\beta \leq n,  ~ 1 \leq i,j \leq N, ~ k \in K \right)
\end{equation}
and
\begin{equation}\label{cof_com}
A^{\alpha \beta,k}_{ij} \in C^{m,\gamma}(\br^{n})
\qquad \qquad
\left( 1 \leq \alpha,\beta \leq n,  ~ 1 \leq i,j \leq N, ~ k \in K \right),
\end{equation}
for some $\Lambda \geq \lambda >0$ and $\gamma \in (0,1/4]$. Set $A^{\alpha \beta}_{ij} : Q_{7} \to \br$ as 
\begin{equation}\label{newell}
A^{\alpha \beta}_{ij} = \sum_{k \in K} A^{\alpha \beta,k}_{ij} \chi_{ Q_{7}^{k} } 
\end{equation}
for any $1 \leq \alpha,\beta \leq n$, $1 \leq i,j \leq N$ and $k \in K$. Then by \eqref{ell_com}, \eqref{growth_com} and \eqref{cof_com},
\begin{equation}\label{ell}
\lambda |\zeta|^{2} 
\leq  \sum_{1 \leq i,j \leq N} \sum_{1 \leq \alpha, \beta \leq n} A^{\alpha \beta}_{ij}(x) \zeta_{i} \zeta_{j}
\qquad \qquad 
\left( x \in \br^{n}, ~ \zeta \in \br^{Nn} \right),
\end{equation}
\begin{equation}\label{growth}
\left| A^{\alpha \beta}_{ij}(x) \right|
\leq \Lambda 
\qquad \qquad 
\left( x \in \br^{n}, ~ 1 \leq \alpha,\beta \leq n,  \ 1 \leq i,j \leq N \right)
\end{equation}
and
\begin{equation}\label{cof}
A^{\alpha \beta}_{ij} \in C^{m,\gamma} \left( Q_{7}^{k} \right)
\qquad \qquad
\left( 1 \leq \alpha,\beta \leq n,  ~ 1 \leq i,j \leq N, ~ k \in K \right).
\end{equation}

Let $\left( Q_{7}(z) , \left\{ \varphi_{k} : k \in K_{+} \right\} \right)$ be a compostie cube with the the assumption that 
\begin{equation}\label{graph_con}
\varphi_{k} \in C^{m+1,\gamma} \left( Q_{7}' \right)
\qquad (k \in K_{+}).
\end{equation}
Then let  $u \in W^{1,2}(Q_{7})$ be a weak solution of 
\begin{equation}\label{main equation}
\sum_{1 \leq \alpha \leq n} D_{\alpha} \left[ \sum_{1 \leq j \leq N} \sum_{1 \leq \beta \leq n} A^{\alpha \beta}_{ij}D_{\beta}u^{j} \right] = \sum_{1 \leq \alpha \leq n}  D_{\alpha} F_{\alpha}^{i} 
\quad \text{in} \quad
Q_{7},
\end{equation}
with the assumption that 
\begin{equation}\label{data_con}
u , \, F \in C^{m,\gamma} \left( Q_{7}^{k} \right)
\qquad (k \in K).
\end{equation}
In this paper, we might omit the summation by using Einstein notation.

To prove piece-wise H\"{o}lder continuity of $D^{m+1}u$, we will first show in Theorem \ref{main theorem of U} that the following $U^{p'} : Q_{5} \to \br^{Nn}$ is H\"{o}lder continuous in $Q_{5}$. For any $p' \geq 0'$ with $|p'|=m$,  set $U^{p'} : Q_{5} \to \br^{Nn}$ as 
\begin{equation*}\label{} 
U_{1}^{i,p'}
=G_{1}^{i,p'} + \sum_{1 \leq \alpha \leq n} \pi_{\alpha} \left[  \sum_{0' \leq q' \leq p' } 
\Big( \begin{array}{c} p' \\ q' \end{array} \Big)
\pi^{q'} D^{(|q'|,p'-q')}  \left[ \sum_{1 \leq j \leq N } \sum_{1 \leq \beta \leq n} A^{\alpha \beta}_{ij}D_{\beta}u^{j} - F_{\alpha}^{i} \right] \right]
\end{equation*}
with
\begin{equation*}\label{} 
U_{\beta}^{i,p'} = G_{\beta}^{i,p'} +  \sum_{ 0' \leq q' \leq p' } 
\Big( \begin{array}{c} p' \\ q' \end{array} \Big)
\pi^{q'}
\left[ D^{ (|q'|,p'-q') }  D_{\beta}u^{i} + \pi_{\beta} D^{ (|q'|,p'-q') }  D_{1}u^{i} \right]
\end{equation*}
for any  $ 2 \leq \beta \leq n$ and $1 \leq i \leq N$, where $G^{p'} : Q_{5} \to \br^{Nn}$ is defined as
\begin{equation*}\label{}
G_{1}^{i,p'} =  \sum_{ k \in  K_{-}  } \sum_{1 \leq \alpha \leq n} \ \bg_{\alpha}^{p',k} \left[ U_{\alpha}^{i} \big|_{k-1} - U_{\alpha}^{i} \big|_{k} \right]  H_{k}  
\end{equation*}
and
\begin{equation*}\label{}
G_{\beta}^{i,p'} = \sum_{k \in K_{-}} \bh_{\beta}^{p',k} \left [ u \big|_{k-1} - u \big|_{k} \right ] H_{k} 
\end{equation*}
for any  $ 2 \leq \beta \leq n$ and $1 \leq i \leq N$. Here,  $H_{k} : Q_{7} \to \br$ $\big( k \in K \big)$ is Heavyside function that
\begin{equation*}\label{}
H_{k}(x)  = \left\{\begin{array}{rl}
1 &  \text{ if  } x \in Q_{7}^{l} ~ \text{ with } ~ l \geq k, \\
0 &  \text{ if  } x \in Q_{7}^{l} ~ \text{ with } ~ l < k.
\end{array}\right.\end{equation*}
Also with \eqref{GS190}, the linear operators$\ \bg_{\alpha}^{p',k} : C^{m-1,\gamma} \left( \partial Q_{6}^{k} \cap \partial Q_{6}^{k-1} \right) \to C^{\gamma} \left( Q_{5}' \right) $  $(k \in K_{-})$ and $\bh_{\alpha}^{p',k} : C^{m,\gamma} \left( \partial Q_{6}^{k} \cap \partial Q_{6}^{k-1} \right) \to C^{\gamma} \left( Q_{5}' \right) $  $(k \in K_{-})$ 
are defined as
\begin{equation*}\begin{aligned}\label{}
\ \bg_{\alpha}^{p',k}[h](x') 
& = \sum_{|\xi| \leq |p'|-1} \left[ \pi_{\alpha}  P_{\xi}^{ p' }  D_{\xi}  h \right] \big( \varphi_{k}(x'),x' \big) \\
& \quad + \sum_{ 0' < q' \leq p' } \left( \begin{array}{c} p' \\ q' \end{array} \right) D^{q'} D_{i}\varphi_{k}(x') \, D^{p'-q'}  \left[ h\big(  \varphi_{k}(x'), x' \big) \right],
\end{aligned}\end{equation*}
and 
\begin{equation*}\begin{aligned}\label{}
\bh_{\alpha}^{p',k}[h](x') &=  \sum_{0' \leq q' \leq p' } 
\Big( \begin{array}{c} p' \\ q' \end{array} \Big) 
D_{\alpha} \left[ \pi^{q'} \big( \varphi_{k}(x') , x' \big) \right]
D^{ ( |q'|, p'-q' ) } h \big( \varphi_{k}(x') , x' \big) \\
& \quad+   \sum_{|\xi| \leq |p'|-1}  D_{\alpha} \left[ P_{\xi}^{p'} \big[ \varphi_{k}(x'), x' \big] \, D_{\xi}h \big( \varphi_{k}(x') , x' \big) \right],
\end{aligned}\end{equation*}
for any $ 1 \leq \alpha \leq n$ and  $k \in K$, where the notation 
\begin{equation*}\label{}
u \big|_{k} \in C^{m,\gamma} \left( \overline{Q_{5}^{k}} \right)
\qquad \text{and} \qquad 
U_{\alpha}^{i} \big|_{k} \in C^{m-1,\gamma} \left( \overline{Q_{5}^{k}} \right)
\end{equation*}
denote the extension of 
\begin{equation*}\label{}
u \in C^{m,\gamma} \left( Q_{7}^{k} \right) 
\qquad \text{and}  \qquad
U_{\alpha}^{i} = \sum_{ 1 \leq j \leq N} \sum_{ 1 \leq \beta \leq n}  A_{\alpha \beta}^{ij} D_{\beta}u^{j} - F_{\alpha}^{i} \in C^{m-1,\gamma} \left( Q_{7}^{k} \right)
\end{equation*}
for any $1 \leq \alpha \leq n$ and $1 \leq i \leq N$, which exist by the assumption \eqref{data_con}.

\begin{rem}
By the definition, $\bg_{\alpha}^{p',k}[h](x')$ and $\bh_{\alpha}^{p',k}[h](x')$ depends only on $x'$-variables. So $G^{p'}$ depends only on $x'$-variables in $Q_{5}^{k}$ and are  H\"{o}lder continuous in each $Q_{5}^{k}$ $(k \in K)$ and might have big jumps on $\partial Q_{5}^{k} \cap \partial Q_{5}^{k+1}$ $(k \in K_{-})$. 

We will prove  that $U^{p'}$, $\pi_{\alpha}$ and $\pi^{q'}$ are H\"{o}lder continuous in $Q_{5}$. So one can only show that $D^{m+1}u$ is piece-wise H\"{o}lder continuous due to vector-valued the function $G^{p'}$ which might have big jumps on $\partial Q_{5}^{k} \cap \partial Q_{5}^{k+1}$ $(k \in K_{-})$. 
\end{rem}

In the previous \cite{KYSP1}, we discovered a term related to the gradient of the weak solution which is H\"{o}lder continuous. In this paper, we obtain a corresponding  result related to higher order derivatives. 

\begin{thm}\label{main theorem of U}
Let  $u$ be a weak solution of \eqref{main equation}. Then for any $p' \geq 0'$ with $|p
|=m$, we have that $U^{p'} \in C^{\gamma}\left( Q_{3} \right)$ with the estimate
\begin{equation*}\begin{aligned}\label{} 
\big\| U^{p'} \big\|_{L^{\infty}(Q_{r}(z))}^{2}\leq cr^{-2} E^{2},
\end{aligned}\end{equation*}
\begin{equation*}\label{}
\big\| U^{p'} \big\|_{L^{\infty}(Q_{r}(z))}^{2}
\leq c \left[ \mint_{Q_{2r}(z)} \big| U^{p'} \big|^{2} \, dx   + E^{2} \right]
\end{equation*}
and 
\begin{equation*}\label{}
\big[ U^{p'} \big]_{C^{\gamma}(Q_{r}(z))}^{2} 
\leq c \left[  \frac{1}{r^{2\gamma} }
\mint_{Q_{2r}(z)} \left| U^{p'}  - \big( U^{p'} \big)_{Q_{2r}(z)} \right|^{2} dx 
+  \sum_{|q'|=m}  \big\| U^{q'} \big \|_{L^{\infty}(Q_{2r}(z))}^{2} + E^{2} \right],
\end{equation*}
for any $Q_{2r}(z) \subset Q_{5}$ where $E  = \sum_{ k \in K }  \left[ \| u \|_{C^{m,\gamma} (Q_{7}^{k}) } +  \| F \|_{C^{m,\gamma} (Q_{7}^{k}) } \right]$.
\end{thm}

$D^{m+1}u$ and $U^{p'}$ will be compared by using the following proposition.
\begin{prop}\label{prop_compare}
For any $z \in Q_{5}$, we have that 
\begin{equation*}\label{} 
| D^{m+1}u(z) |
\leq c \left[ E + \sum_{|p'| = m}  \big| U^{p'}(z) \big|   \right]
\leq c \left[ E + | D^{m+1}u(z) | \right].
\end{equation*}
In additional, for any $w, z \in Q_{5}^{k}$ $(k \in K)$, we have that
\begin{equation*}\begin{aligned}
\big| D^{m+1}u(w) - D^{m+1}u(z)\big| 
& \leq c \sum_{|p'| = m} \left[  \big| U^{p'}(w) - U^{p'}(z)\big| 
+  |w-z|^{\gamma} \big| U^{p'}(z) \big| \right]  \\
& \quad + c |w-z|^{\gamma} \sum_{k \in K} \left[  \| u \|_{C^{m,\gamma} (Q_{7}^{k}) }  +  \| F \|_{C^{m,\gamma} (Q_{7}^{k}) }  \right].
\end{aligned}\end{equation*}
\end{prop}

\begin{thm}\label{main theorem}
For a weak solution $u$ of \eqref{main equation}, we have that $u \in C^{m+1,\gamma}\left( Q_{3}^{k} \right)$ $(k \in K)$ and 
$D^{m+1}u \in L^{\infty}(Q_{3}^{k})$ $(k \in K)$ with the estimate
\begin{equation*}\begin{aligned}\label{} 
\frac{1}{|Q_{r}|} \int_{Q_{ r }^{k}(z) }  \left| D^{m+1}u \right|^{2}  \, dx 
\leq cr^{-2} E^{2},
\end{aligned}\end{equation*}
\begin{equation*}\label{}
\left\| D^{m+1}u \right\|_{L^{\infty}(Q_{r}^{k}(z))}^{2}
\leq c \left[ \sum_{l \in K} \frac{1}{|Q_{2r}|} \int_{Q_{2r}^{l}(z)} \big| D^{m+1}u \big|^{2} \, dx   + E^{2} \right]
\end{equation*}
and 
\begin{equation*}\label{}
\big[ D^{m+1}u \big]_{C^{\gamma}(Q_{r}^{k}(z))}^{2} 
\leq \frac{c}{r^{2\gamma}}  \left[  \sum_{l \in K} \frac{1}{|Q_{2r}|}  \int_{Q_{2r}^{l}(z)} \big| D^{m+1}u \big|^{2} \, dx  
+ E^{2}  \right],
\end{equation*}
for any $Q_{2r}(z) \subset Q_{5}$ and $k \in K$, where $E  = \sum_{ k \in K }  \left[ \| u \|_{C^{m,\gamma} (Q_{7}^{k}) } +  \| F \|_{C^{m,\gamma} (Q_{7}^{k}) } \right]$.
\end{thm}

\begin{rem}
As in \cite{KYSP1}, one can in fact find a vector-valued function related to $(m+1)$-th derivatives   which $C^{\gamma}(Q_{2})$. This function has big jumps on the boundary of the subregions. 
\end{rem}

\section{Geometric settings and the related properties}

With the assumption \eqref{graph_con} for some $m \geq 0$ and $\gamma \in (0,1/4]$, let $ \left( Q_{7}, \left\{ \varphi_{k} : k \in K_{+} \right \} \right)$ be a composite cube. For the approximation argument, we assume that
\begin{equation}\label{GS130}
\varphi_{k}(x') + \delta \leq \varphi_{k+1}(x')
\qquad \big(  x \in Q_{7}', \, k \in K \big),
\end{equation} 
for  some $\delta>0$. The desired estimate does not depend on $\delta>0$ and one can use an approximation argument to handle the case that  $\delta=0$. Since $ \left( Q_{7}, \left\{ \varphi_{k} : k \in K_{+} \right \} \right)$ is a composite cube, recall from Definition \ref{composite cube} that
\begin{equation}\label{GS150}
Q_{7} = \dotcup_{k \in K} Q_{7}^{k}. 
\end{equation}

\subsection{Time derivatives of the flow}
To estimate the higher derivatives for linear elliptic systems in composite cubes, we generalize the difference quotient from lines(or the fixed directions) to general curves. To do it, we set the flow $\psi(x,t')$ as follows. For  the graph functions, $\left\{ \varphi_{k} : k \in K_{+} \right \}$, set $T_{k} : Q_{7} \to [0,1]$ ($k \in K$) as
\begin{equation}\label{GS220} 
T_{k}(x^{1},x') = \frac{x^{1} - \varphi_{k}( x')}{\varphi_{k+1}(x') - \varphi_{k}(x')} 
\quad \text{in} \quad Q_{7}^{k}.
\end{equation}

\begin{rem}
For any $l > k$, $T_{k} = 1$ in $Q_{7}^{l}$ and for any $l<k$, $T_{k}=0$ in $Q_{7}^{l}$. So with \eqref{GS130}, $T_{k}$ linearly increases from $\left( \varphi_{k}(x'),x' \right)$ to $\left( \varphi_{k+1}(x'),x' \right)$ in $Q_{7}^{k}$. 
\end{rem}

For any $k \in K$, define $\psi : Q_{6} \times Q_{1}'  \to \br$ so that
\begin{equation}\label{GS230}
\psi(x,t') 
= \big[ \varphi_{k+1} (x'+ t') -\varphi_{k+1}( x') \big] \cdot T_{k}(x)
+ \big [\varphi_{k} ( x'+ t') - \varphi_{k}( x') \big] \cdot \big[ 1-T_{k}(x) \big]
\end{equation}
in $Q_{6}^{k} \times Q_{1}'$. 
If the flow starts at $x \in Q_{6}^{k}$ then the flow $ x + \left( \psi(x,t'), t'  \right)$ at time $t' \in Q_{1}'$ still remains at the same region $Q_{7}^{k}$ as in the following lemma.

\begin{lem}\label{GSS250}
For any $(x,t')  \in Q_{6}^{k} \times Q_{1}'$, we have that 
\begin{equation*}\label{}
 x + \left( \psi(x,t'), t'  \right)  = \left( x^{1} + \psi(x,t'), x'+t' \right) \in Q_{7}^{k}.
\end{equation*}
\end{lem}

\begin{proof}
For any $(x,t') \in Q_{6} \times Q_{1}'$, we have from \eqref{gradient_graphs} that 
\begin{equation*}\begin{aligned}\label{}
& \left| \left( x^{1} + \psi(x,t')  \right) \right| \\
& \quad \leq |x^{1}| + \big| \varphi_{k+1} (x'+ t') -\varphi_{k+1}( x') \big| \cdot T_{k}(x)
+ \big| \varphi_{k} ( x'+ t') - \varphi_{k}( x') \big| \cdot \big[ 1-T_{k}(x) \big] \\
& \quad \leq |x^{1}| + |t'|  < 7,
\end{aligned}\end{equation*}
which implies that 
\begin{equation}\label{GS255}
\left \{  x + \left( \psi(x,t'), t'  \right) : (x,t')  \in Q_{6}^{k} \times Q_{1}' \right \} \subset Q_{7}.
\end{equation}
Also from \eqref{GS220} and \eqref{GS230}, one can check that for any  $(x,t') \in Q_{6}^{k} \times Q_{1}'$,
\begin{equation*}\begin{aligned}\label{}
& x^{1} + \psi(x,t') \\
& \quad = \frac{ [\varphi_{k+1}(x'+t') - \varphi_{k}(x'+t')] x^{1} + \varphi_{k}(x'+t') \varphi_{k+1}(x') - \varphi_{k+1}(x'+t') \varphi_{k}(x') }{\varphi_{k+1}(x') - \varphi_{k}(x')}.
\end{aligned}\end{equation*}
If $x \in Q_{6}^{k}$ then $\varphi_{k}(x') < x^{1} \leq \varphi_{k+1}(x')$. So for any  $(x,t') \in Q_{6}^{k} \times Q_{1}'$, we find that $\varphi_{k}(x'+t') < x^{1} + \psi(x,t') \leq \varphi_{k+1}(x'+t')$. With \eqref{GS255}, this proves the lemma.
\end{proof}

Fix $p' = (p_{2}, \cdots, p_{n}) \geq 0'$ with $|p'|=m$. Then we consider the following $m$-th order time derivative of $h$ with respect to the flow $\psi(x,t')$ :
\begin{equation}\begin{aligned}\label{GS265} 
\partial_{t'}^{p'}h \big|_{t'=0'}  : = \partial_{t^{2}}^{p_{2}} \cdots \partial_{t^{n}}^{p_{n}}  \big[ h \big( x + (\psi(x,t'),t') \big) \big] \big|_{t'=0'}
~ \text{ in  } ~ Q_{6}.
\end{aligned}\end{equation}
We calculate \eqref{GS265} as follows.

\smallskip

For any $i \in \{ 1, \cdots, n \}$, to denote the first time derivative of the flow, set
\begin{equation}\label{GS260}
\pi_{\alpha}(x)  
: = D_{i}\varphi_{k+1}( x') \cdot T_{k}(x)
+ D_{i}\varphi_{k}( x')  \cdot [1-T_{k}(x)]
\end{equation}
for any $x \in Q_{7}^{k}$ $(k \in K)$. Then we find from \eqref{GS230} that
\begin{equation}\label{}
\pi_{\alpha}(x)  = \partial_{t^{i}} \psi(x,t) \big|_{t=0} 
\qquad ( x \in Q_{6}).
\end{equation}
In addition, we set
\begin{equation}\label{GS410}
\pi'  = ( \pi_{2}, \cdots, \pi_{n} )
\quad \text{and} \quad 
\pi = (-1,\pi') = ( -1, \pi_{2}, \cdots, \pi_{n} )
~ \text{ in } ~  Q_{7}.
\end{equation}
Also we denote the power of the first time derivative of the flow as
\begin{equation}\label{GS270} 
\pi^{q'} : = \pi_{2}^{q_{2}} \cdots \pi_{n}^{q_{n}}
~ \text{ in } ~ Q_{7}.
\end{equation}

\begin{lem}\label{GSS275}
For any $p' \geq 0'$ with $|p'|=m$, there exists a polynomial $\tilde{P}_{\xi}^{ p' } : \br^{\tiny  \left( \begin{array}{c} n-1 \\ 1 \end{array} \right)  \times
\left( \begin{array}{c} n \\ 2 \end{array} \right) \times \cdots \times  \left( \begin{array}{c} n+m-2 \\ m \end{array} \right) } \to \br$ such that
\begin{equation}\begin{aligned}\label{GS280} 
\partial_{t'}^{p'}h \big|_{t'=0'}  
= \sum_{0' \leq q' \leq p' } \Big( \begin{array}{c} p' \\ q' \end{array} \Big)
\pi^{q'} D^{ (|q'|,p'-q') } h 
+ \sum_{|\xi| \leq |p'|-1} P_{\xi}^{ p' } \, D_{\xi} h 
~ \text{ in } ~ Q_{6},
\end{aligned}\end{equation}
where $P_{\xi}^{ p' } : Q_{7} \to \br$ is defined as
\begin{equation*}\label{} 
P_{\xi}^{ p' } = \tilde{P}_{\xi}^{ p' } \left [ T_{k} D_{x'}\varphi_{k+1}  + [1-T_{k}] \, D_{x'}\varphi_{k}  , \cdots, T_{k} \, D_{x'}^{m} \varphi_{k+1}  + [1-T_{k}] \, D_{x'}^{m} \varphi_{k}  \right ]  
~ \text{ in } ~ Q_{7}^{k}
\end{equation*}
for any $k \in K$.
\end{lem}

\begin{proof}
By the definition of $\partial_{t'}^{p'}h \big|_{t'=0'}$ in \eqref{GS265}, we obtain that 
\begin{equation*}\begin{aligned}\label{} 
& \partial_{t'}^{p'}h \big|_{t'=0'}  \\
&\,  = \partial_{t^{2}}^{p_{2}} \cdots \partial_{t^{n}}^{p_{n}}  \big[ h \big( x + (\psi(x,t'),t') \big) \big] \big|_{t'=0'} \\
&\, = \sum_{0' \leq q' \leq p' } \Big( \begin{array}{c} p_{2} \\ q_{2} \end{array} \Big) \cdots \Big( \begin{array}{c} p_{n} \\ q_{n} \end{array} \Big)
\big( \partial_{t^{2}} \psi \big|_{t'=0'} \big)^{q_{2}} \cdots \big( \partial_{t^{n}} \psi \big|_{t'=0'} \big)^{q_{n}} 
D_{2}^{p_{2}-q_{2}}D_{1}^{q_{2}} \cdots D_{n}^{p_{n}-q_{n}}D_{1}^{q_{n}}  h \\
&\, \quad + \sum_{|\xi| \leq p_{2}+\cdots+p_{n}-1} \tilde{P}_{\xi}^{ p_{2}, \cdots, p_{n} } \big( \partial_{t^{2}} \psi \big|_{t'=0'}, \cdots, \partial_{t^{n}} \psi \big|_{t'=0'}, \cdots, \partial_{t^{2}}^{p_{2}} \cdots \partial_{t^{n}}^{p_{n}} \psi \big|_{t'=0'} \big) D_{\xi} h 
\end{aligned}\end{equation*}
in $Q_{6}$, for some polynomial $\tilde{P}_{\xi}^{ p_{2}, \cdots, p_{n} } : \br^{\tiny  \left( \begin{array}{c} n-1 \\ 1 \end{array} \right)  \times
\left( \begin{array}{c} n \\ 2 \end{array} \right) \times \cdots \times  \left( \begin{array}{c} n+m-2 \\ m \end{array} \right) } \to \br$.  The number of the choice for the first derivative is $ \left( \begin{array}{c} (n-1) + 1 - 1\\ 1 \end{array} \right)$, the second derivative is $ \left( \begin{array}{c} (n-1) + 2  - 1 \\ 2 \end{array} \right)$, $\cdots$, the $m$-th derivative is $ \left( \begin{array}{c} (n-1) + m  - 1 \\ m \end{array} \right)$.  We simplify the above equality. Recall that
\begin{equation*}
D_{x'}^{q'} = D_{x^{2}}^{q_{2}} \cdots D_{x^{n}}^{q_{n}},
\qquad
|p'| = p_{2} + \cdots + p_{n}
\qquad \text{and} \qquad 
\Big( \begin{array}{c} p' \\ q' \end{array} \Big)
=  \Big( \begin{array}{c} p_{2} \\ q_{2} \end{array} \Big) \cdots  \Big( \begin{array}{c} p_{2} \\ q_{2} \end{array} \Big).
\end{equation*} 
From \eqref{GS260} and  \eqref{GS270}, we obtain that for any $p' \geq q' \geq 0'$,
\begin{equation*}\begin{aligned}\label{}
&\sum_{0' \leq q' \leq p' } \Big( \begin{array}{c} p_{2} \\ q_{2} \end{array} \Big) \cdots \Big( \begin{array}{c} p_{n} \\ q_{n} \end{array} \Big)
\big( \partial_{t^{2}} \psi \big|_{t'=0'} \big)^{q_{2}} \cdots \big( \partial_{t^{n}} \psi \big|_{t'=0'} \big)^{q_{n}} 
D_{2}^{p_{2}-q_{2}}D_{1}^{q_{2}} \cdots D_{n}^{p_{n}-q_{n}}D_{1}^{q_{n}}  h \\
& \quad = \sum_{0' \leq q' \leq p' } \Big( \begin{array}{c} p' \\ q' \end{array} \Big) 
\pi^{q'} D^{ (|q'|,p'-q') } h.
\end{aligned}\end{equation*}
For any  $q' \geq 0$, we  find from  \eqref{GS230} that
\begin{equation*}\label{}
\partial_{t'}^{q'}  \psi(x,t') 
= D_{x'}^{q'}  \varphi_{k+1} (x'+ t')  \, T_{k}(x) 
+ D_{x'}^{q'}  \varphi_{k} (x'+ t')  \,  [1-T_{k}(x)]
\end{equation*}
in  $Q_{6}^{k} \times Q_{1}'$, which implies that ,
\begin{equation*}\label{}
\partial_{t'}^{q'}  \psi(x,t') \big|_{t'=0'}
= D_{x'}^{q'}  \varphi_{k+1} (x')  \, T_{k}(x) 
+ D_{x'}^{q'}  \varphi_{k} (x')  \,  [1-T_{k}(x)]
\end{equation*}
for any $x \in Q_{6}^{k}$. So by letting $\tilde{P}_{\xi}^{p'} : = \tilde{P}_{\xi}^{p_{2}, \cdots, p_{n}} $, the lemma follows from \eqref{GS230}.
\end{proof}

\subsection{Decay estimate of the graph functions and \texorpdfstring{$\pi'$}{pi'} in \eqref{GS410} }

To handle the two non-crossing graph functions $\varphi_{k}$ $(k \in  K_{+} )$, we use following result in \cite{KYSP1}, which naturally holds from our geometric settings (also see \cite[Section 5]{LYVM1} or \cite[Section 4]{LYNL1}).

\begin{lem}[\cite{KYSP1}]\label{GSS300}
Suppose that $\varphi_{k}, \varphi_{l} : C^{1,\gamma}
(Q_{r+\rho}') \to \mathbb{R}$ satisfy that 
\begin{equation*}
[ D_{x'}\varphi_{k} ]_{C^{\gamma}(Q_{r+\rho}')}, [ D_{x'}\varphi_{l} ]_{C^{\gamma}(Q_{r+\rho}')} \leq c_{1},
\end{equation*}
\begin{equation*}
\| \varphi_{k} \|_{L^{\infty}(Q_{r+\rho}')}, \| \varphi_{l} \|_{L^{\infty}(Q_{r+\rho}')} \leq c_{2}, 
\end{equation*}
and
\begin{equation*}
\varphi_{k} \leq \varphi_{l}
\text{ in } Q_{r+\rho}'.
\end{equation*}
Then we have that
\begin{equation}\label{GS330}
|D_{x'}\varphi_{l}(x') - D_{x'}\varphi_{k}(x')|
\leq 3\rho^{-1} \left( \rho^{1+\gamma} c_{1} + 2c_{2} \right)^{\frac{1}{\gamma+1}} [\varphi_{l}(x') - \varphi_{k}(x')]^{\frac{\gamma}{\gamma+1}} 
\end{equation}
for any $x' \in Q_{r}'$.
\end{lem}

Let $\left( Q_{7} , \left \{ \varphi_{k} : k \in K_{+} \right \} \right )$ be a composite cube. By applying Lemma \ref{GSS300} to $Q_{7}$, we obtain from \eqref{graph_con} and \eqref{GS130} that
\begin{equation}\label{GS350}
|D_{x'}\varphi_{k}(x') - D_{x'}\varphi_{l}(x')| 
\leq c |\varphi_{k}(x') - \varphi_{l}(x')|^{\frac{1}{2}}
\qquad  \big( x' \in Q_{6}' , ~ k,l \in K_{+} \big).
\end{equation}
Also we have from  \eqref{graph_con} that
\begin{equation}\label{GS400}
|D_{x'}\varphi_{k}(x') - D_{x'}\varphi_{k}(y')|
\leq c |x'-y'|^{\frac{1}{2}}
\qquad  \left( x',y' \in Q_{6}', ~ k \in K_{+} \right).
\end{equation}
Then we have the following lemma from \cite{KYSP1}.

\begin{lem}[\cite{KYSP1}]\label{GSS600}
For a composite cube $\left( Q_{7} , \left \{ \varphi_{k} : k \in K_{+} \right \} \right )$ and the corresponding $\pi'$ in \eqref{GS410}, we have that
\begin{equation*}
\big| \pi'(y) - \pi'(z) \big| \leq c |y-z|^{\frac{1}{2}}
\qquad (y,z \in Q_{6}).
\end{equation*}
\end{lem}

\begin{proof}
The lemma follows by applying \eqref{GS350} and \eqref{GS400} to \cite[Lemma 2.3]{KYSP1}.
\end{proof}

\subsection{Estimate of the integrals related to \texorpdfstring{$\pi$}{pi} in \eqref{GS410} } 

To use the weak formation, we will apply integration by parts formula. To do it, we obtain  the estimates related to $D\pi$ in Lemma \ref{GSS800} and \eqref{GS890} where $\pi$ is defined in \eqref{GS410}.

\begin{lem}\label{GSS700}
With the assumptions \eqref{GS130} and \eqref{gradient_graphs}, we have that
\begin{equation*}\label{}
\big| D\pi(x) \big| \leq c \Big[ 1 +   |\varphi_{k+1}(x') - \varphi_{k}(x')|^{-\frac{1}{2}} \Big]
\qquad\qquad  \left ( x \in Q_{6}^{k} , ~ k \in K \right ).
\end{equation*}
\end{lem}

\begin{proof}
Since $\pi = (\pi_{1}, \cdots, \pi_{n})$ and $\pi_{1}=-1$, we estimate $D\pi_{\alpha}$ $(i  = 2, \cdots, n )$. Fix $i \in \{2, \cdots, n \}$ and $k \in K$. By Definition \ref{composite cube},
\begin{equation}\label{GS720} 
\varphi_{k}(x') \leq x^{1} \leq \varphi_{k+1}(x')
\qquad \left ( x \in Q_{6}^{k} \right ).
\end{equation}
With \eqref{GS220} and \eqref{GS260}, we find from \eqref{GS350} that
\begin{equation}\label{}
|D_{1}\pi_{\alpha}|
= \bigg| \frac{ D_{i}\varphi_{k+1} - D_{i}\varphi_{k} }{\varphi_{k+1} - \varphi_{k}} \bigg|  
\leq \frac{c}{|\varphi_{k+1} - \varphi_{k}|^{\frac{1}{2}}}
~ \text{ in } ~ Q_{6}^{k}.
\end{equation}
It only remains to estimate $D_{j}\pi_{\alpha}$ $( j = 2, \cdots, n)$ in $Q_{6}^{k}$. By \eqref{GS220} and \eqref{GS260}
\begin{equation*}\begin{aligned}\label{}
\pi_{\alpha} (x)
& = \frac{ D_{i}\varphi_{k+1}(x')[x^{1} - \varphi_{k}(x')]} { \varphi_{k+1}(x') - \varphi_{k}(x')}
+ \frac{ D_{i}\varphi_{k}(x')[\varphi_{k+1}(x') - x^{1}]} { \varphi_{k+1}(x') - \varphi_{k}(x')},
\end{aligned}\end{equation*}
for any $x  \in Q_{6}^{k}$, which implies that 
\begin{equation*}\begin{aligned}\label{}
D_{j} \pi_{\alpha} (x)
& = \frac{ D_{ij}\varphi_{k+1}(x')[x^{1} - \varphi_{k}(x')]} { \varphi_{k+1}(x') - \varphi_{k}(x')}
+ \frac{ D_{ij}\varphi_{k}(x')[\varphi_{k+1}(x') - x^{1}]} { \varphi_{k+1}(x') - \varphi_{k}(x')} \\
& \quad - \frac{ D_{i}\varphi_{k+1}(x') D_{j}\varphi_{k}(x') - D_{i}\varphi_{k}(x') D_{j}\varphi_{k+1}(x') }{ \varphi_{k+1}(x') - \varphi_{k}(x') } \\
& \quad - \frac{ D_{i}\varphi_{k+1}(x')[x^{1} - \varphi_{k}(x')]} { \varphi_{k+1}(x') - \varphi_{k}(x')}
\cdot \frac{D_{j}\varphi_{k+1}(x')-D_{j}\varphi_{k}(x')}
{ \varphi_{k+1}(x') - \varphi_{k}(x')} \\
& \quad - \frac{ D_{i}\varphi_{k}(x')[\varphi_{k+1}(x') - x^{1}]} { \varphi_{k+1}(x') - \varphi_{k}(x')} \cdot \frac{D_{j}\varphi_{k+1}(x')-D_{j}\varphi_{k}(x')}
{ \varphi_{k+1}(x') - \varphi_{k}(x')},
\end{aligned}\end{equation*}
for any $x  \in Q_{6}^{k}$ and $j = 2, \cdots, n$. Then by the following calculation,
\begin{equation*}\begin{aligned}\label{} 
& \frac{ D_{i}\varphi_{k+1}(x') D_{j}\varphi_{k}(x') - D_{i}\varphi_{k}(x') D_{j}\varphi_{k+1}(x') }{ \varphi_{k+1}(x') - \varphi_{k}(x') } \\
& \quad = \frac{ D_{i}\varphi_{k+1}(x') [ D_{j}\varphi_{k+1}(x') - D_{j}\varphi_{k}(x')] 
- D_{j}\varphi_{k+1}(x') [D_{i}\varphi_{k+1}(x') - D_{i}\varphi_{k}(x')] }{ \varphi_{k+1}(x') - \varphi_{k}(x') },
\end{aligned}\end{equation*}
the lemma follows from \eqref{GS350} and \eqref{GS720}, because $i \in \{ 2, \cdots, n \}$ and $k \in K$ were arbitrary chosen.
\end{proof}

\begin{lem}\label{GSS800}
With the assumptions \eqref{GS130} and \eqref{gradient_graphs}, we have that
\begin{equation*}\label{}
\int_{Q_{r}(z)} |D\pi|^{2} \eta^{2} \, dx
\leq \epsilon r^{\frac{1}{2}} \int_{Q_{r}(z)} |D\eta|^{2} \, dx + c  \epsilon^{-2} r^{\frac{1}{2}}  \int_{Q_{r}(z)} \eta^{2} \, dx
\quad \left( \eta \in C_{c}^{\infty}(Q_{r}(z)) \right)
\end{equation*}
for any $Q_{r}(z) \subset Q_{6}$ and $\epsilon \in (0,1]$.
\end{lem}

\begin{proof}
Fix $ \eta \in C_{c}^{\infty}(Q_{r}(z)) $ and $k \in K$. Then by Lemma \ref{GSS700} and Young's inequality, 
\begin{equation*}\begin{aligned}
\int_{Q_{r}^{k}(z)} |D\pi|^{2} \eta^{2}  dx 
& \leq c\int_{Q_{r}^{k}(z)}  \eta^{2} + \frac{ \eta^{2} \chi_{Q_{r}^{k}(z)}  } {\varphi_{k+1} - \varphi_{k}} \, dx \\
& = c_{1} \bigg[  \int_{Q_{r}(z)} \eta^{2}  dx  
+ \int_{Q_{r}'(z')} \int_{\max \{ \varphi_{k}(x'), z^{1} -r \}}^{\min \{ \varphi_{k+1}(x'), z^{1} + r \}} \frac{  \eta^{2} (x^{1},x')   }{\varphi_{k+1} - \varphi_{k}} dx^{1} dx' \bigg].
\end{aligned}\end{equation*}
Since $ \eta \in C_{c}^{\infty}(Q_{r}(z)) $, by the following calculation
\begin{equation*}\begin{aligned}
& \int_{\max \, \{ \varphi_{k}(x'), z^{1}-r \}}^{\min \, \{ \varphi_{k+1}(x'), z^{1}+r \}} 
\frac{  \eta^{2} (x^{1},x')   }{\varphi_{k+1} - \varphi_{k}} \, dx^{1}  \\
& \quad = \int_{\max \, \{ \varphi_{k}(x'), z^{1}-r \}}^{\min \, \{ \varphi_{k+1}(x'), z^{1}+r \}} \int_{z^{1}-r}^{x^{1}} \frac{ D_{y^{1}} [\eta^{2}(y^{1},x')] }{ \varphi_{k+1} - \varphi_{k} } \, dy^{1}   dx^{1}  \\
& \quad \leq 2 \int_{\max \, \{ \varphi_{k}(x'), z^{1}-r \}}^{\min \, \{ \varphi_{k+1}(x'),z^{1}+ r \}} \int_{z^{1}-r}^{z^{1}+r} \frac{  |D_{y^{1}} \eta(y^{1},x')| |\eta(y^{1},x')| }{ \varphi_{k+1} - \varphi_{k} } \, dy^{1}   dx^{1} \\
& \quad \leq 2 \int_{z^{1}-r}^{z^{1}+r} |D_{y^{1}} \eta(y^{1},x')| |\eta(y^{1},x')| \, dy^{1}.
\end{aligned}\end{equation*}
Since $Q_{r}(z) = (z^{1}-r, z^{1}+r ) \times Q_{r}'(z')$, we obtain from Young's inequality that
\begin{equation*}\begin{aligned}\label{}
& \int_{Q_{r}'(z')} \int_{\max \{ \varphi_{k}(x'), z^{1}-r \}}^{\min \{ \varphi_{k+1}(x'), z^{1}+r \}} 
\frac{  \eta^{2} (x^{1},x')   }{\varphi_{k+1} - \varphi_{k}} dx^{1} dx' \\
& \quad \leq 2 \int_{Q_{r}'(z')} \int_{z^{1}-r}^{z^{1}+r}  \left | D_{y^{1}} \eta(y^{1},x') \right | \left |\eta(y^{1},x') \right | \, dy^{1} dx' \\
& \quad \leq \epsilon r^{\frac{1}{2}}  \int_{Q_{r}(z)} |D\eta|^{2} \, dx + \epsilon^{-1} r^{-\frac{1}{2}} \int_{Q_{r}(z)} \eta^{2} \, dx.
\end{aligned}\end{equation*}
Since $k \in K$ was arbitrary chosen, by combining the above estimates, 
\begin{equation*}\label{}
\sum_{k \in K} \int_{Q_{r}^{k}(z)} |D\pi|^{2} \eta^{2}  \, dx
\leq c \epsilon r^{\frac{1}{2}} \int_{Q_{r}(z)} |D\eta|^{2} \, dx + c \left( 1+ \epsilon^{-1} r^{-\frac{1}{2}} \right) \int_{Q_{r}(z)} \eta^{2} \, dx.
\end{equation*}
By Poincar\'{e}'s inequality,
\begin{equation*}\begin{aligned}\label{}
\int_{Q_{r}(z)} \eta^{2} \, dx 
& \leq \left( \int_{Q_{r}(z)} \eta^{2} \, dx  \right)^{\frac{1}{2}} \left(  \int_{Q_{r}(z)} \eta^{2} \, dx \right)^{\frac{1}{2}}  \\
& \leq r \left( \int_{Q_{r}(z)} |D\eta|^{2} \, dx  \right)^{\frac{1}{2}} \left( \int_{Q_{r}(z)} \eta^{2} \, dx  \right)^{\frac{1}{2}} \\
& \leq \epsilon^{2} r  \int_{Q_{r}(z)} |D\eta|^{2} \, dx + \epsilon^{-2} r \int_{Q_{r}(z)} \eta^{2} \, dx.
\end{aligned}\end{equation*}
Since $\epsilon \in (0,1]$ and $r \in (0,6]$, by combining the above two estimates,  we find that 
\begin{equation*}\label{}
\int_{Q_{r}(z)} |D\pi|^{2} \eta^{2}  \, dx
\leq c \epsilon r^{\frac{1}{2}} \int_{Q_{r}(z)} |D\eta|^{2} \, dx + c  \epsilon^{-2} r^{\frac{1}{2}}  \int_{Q_{r}(z)} \eta^{2} \, dx.
\end{equation*}
Since $\eta \in C_{c}^{\infty}(Q_{r}(z))$ and $\epsilon \in (0,1]$ were arbitrary chosen, the lemma follows.
\end{proof}

With the similar reasoning for the proof of Lemma \ref{GSS800}, by applying Lemma \ref{GSS700} to \eqref{GS270}, one can also prove that
\begin{equation}\label{GS890}
 \int_{Q_{r}(z)} \left| D \left[ \pi^{q'} \right] \right|^{2} \, dx
\leq cr^{n-1}.
\end{equation}

\subsection{Heavyside type functions} 

We define $H_{k} : Q_{5} \to \br$ $\big( k \in K \big)$ so that
\begin{equation}\label{GS910}
H_{k}(x)  = \left\{\begin{array}{rl}
1 &  \text{ if  } x \in Q_{5}^{l} \text{ with } l \geq k, \\
0 &  \text{ if  } x \in Q_{5}^{l} \text{ with } l < k.
\end{array}\right.\end{equation}
Then we have that 
\begin{equation}\label{GS920}
\chi_{Q_{5}^{k}} = H_{k}  - H_{k+1} 
~ \text{ in } ~ Q_{5}.
\end{equation}

\section{Perturbation}

From now on, assume \eqref{graph_con} for some $m \geq 1$. With the assumption \eqref{cof_com}, \eqref{ell}, \eqref{graph_con} and \eqref{gradient_graphs}, let $u$ be a weak solution of \eqref{main equation} which is
\begin{equation}\label{PET120} 
D_{\alpha} \left[ A_{\alpha \beta}^{ij}D_{\beta}u^{j} \right] = D_{i} F_{\alpha}^{i}  
\quad \text{in} \quad Q_{7}.
\end{equation}
To estimate $D^{m+1}u$, assume that
\begin{equation}\label{PET130} 
u \in C^{m,\gamma} \left (Q_{7}^{k}, \br^{N} \right ) 
\qquad (k \in K).
\end{equation}
The higher order derivatives of $u$ in $Q_{5}^{k}$ will be perturbed with respect to the points on the graphs $ \left\{ (\varphi_{k}(x'),x') : x' \in Q_{5}' \right\}$ $(k \in K_{+})$. In view of \eqref{newell},
\begin{equation}\label{PET140}
A_{\alpha \beta}^{ij}, F_{\alpha}^{i} \in C^{m,\gamma} \left (Q_{6}^{k} \right )
\text{ is a restriction of }
A_{\alpha \beta}^{ij,k}, F_{\alpha}^{i,k} \in C^{m,\gamma} \left ( \br^{n} \right )
\end{equation}
for any $k \in K$, $ 1 \leq \alpha, \beta \leq n$ and $ 1 \leq i,j \leq N$. With  \eqref{GS130}, one can extend $u \in C^{m,\gamma} \left (Q_{6}^{k}, \br^{N} \right ) $ to $u_{k} \in C^{m,\gamma} \left( \partial Q_{6}^{k} \cap \partial Q_{6}^{k-1}, \br^{N} \right)$ for any $k \in K$ so that
\begin{equation}\label{PET150} 
\| u_{k} \|_{C^{m,\gamma} \left( \partial Q_{6}^{k} \cap \partial Q_{6}^{k-1}, \br^{N} \right)}
\leq c \| u \|_{C^{m,\gamma} \left( Q_{7}^{k}, \br^{N} \right)}.
\end{equation}
Here, we remark that the (higher order) derivatives of $u_{k}$ and $u_{k+1}$ might not coincide on $\partial Q_{5}^{k} \cap \partial Q_{5}^{k+1}$ $\left( k \in K \right)$. To obtain $m$-th time derivative, fix the index 
\begin{equation}\label{PET160}
p' = (p_{2}, \cdots, p_{n}) \geq 0'
\quad \text{with} \quad
|p'| = p_{2} + \cdots + p_{m} = m.
\end{equation}

With an approximation argument by using \eqref{GS130},  we may assume that
\begin{equation}\label{PET170} 
u \in C^{m+1,\gamma} \left ( Q_{7}^{k}, \br^{N} \right )
\qquad (k \in K).
\end{equation}
We remark that the term $\left\| D^{m+1}u \right\|_{L^{\infty}(Q_{r}(z))}$ (for fixed $Q_{r}(z) \subset Q_{5}$) might appear in the estimates. But because the coefficient on $\left\| D^{m+1}u \right\|_{L^{\infty}(Q_{r}(z))}$ can be sufficiently small, this term will removed in Lemma \ref{PMTS1000}.

\subsection{Perturbation on the equation}

In this subsection, we obtain a perturbation on the equation related to higher order derivatives(which will be done in Lemma \ref{PETS900}). Recall $T_{k}(x)$ from  \eqref{GS220}. For for any $i \in \{1,\cdots,n \}$, set $\psi_{i} : Q_{6} \times Q_{1}' \to \br$ as
\begin{equation}\label{PET190} 
\psi_{i}(x,t') = T_{k}(x) \, D_{i}\varphi_{k+1}(x'+t') + [1-T_{k}(x)] \, D_{i} \varphi_{k} (x'+t') 
\end{equation}
for any $(x, t') \in Q_{6}^{k} \times Q_{1}'$ $(k \in K)$. Then we find from \eqref{PET190} that 
\begin{equation}\label{PET200} 
\psi_{1}(x,t') = 0
\end{equation}
for any $(x, t') \in Q_{6} \times Q_{1}'$. Also it follows from  \eqref{GS220} and \eqref{PET190} that 
\begin{equation}\label{PET210}
\psi_{i} \big( (\varphi_{k}(x'),x') ,t' \big) 
= D_{i}\varphi_{k}(x'+t') 
\end{equation}
for any $(x', t') \in Q_{6}' \times Q_{1}'$ and $k \in K $.

To obtain piece-wise regularity, we use the weak formulation in the following way. We differentiate the equation on the boundary $\partial Q_{5}^{k} \cap \partial Q_{5}^{k+1}$  by using the flow \eqref{PET190}. Then we apply integration by parts formula with respect to each $Q_{5}^{k}$ $(k \in K)$. 

In view of \eqref{PET170},  for any $1 \leq \alpha \leq n$, $ 1 \leq i \leq N $ and $k \in K$, set
\begin{equation}\label{PET220}
U_{\alpha}^{i,k} = \sum_{ 1 \leq j \leq N} \sum_{ 1 \leq \beta \leq n}  A_{\alpha \beta}^{ij,k} D_{\beta}u^{j,k} - F_{\alpha}^{i,k} 
\in C^{m,\gamma} \left(  \overline{Q_{6}^{k}} \right),
\end{equation}
and
\begin{equation}\label{PET230}
\partial_{t'}^{p'}  U_{\alpha}^{i,k} \big|_{t'=0'} 
= \left. \partial_{t'}^{p'}  \left[ U_{\alpha}^{i,k} \big( x + (\psi(x,t'),t')  \big) \right] \right|_{t'=0'}
\quad \text{in} \quad Q_{5}^{k}.
\end{equation}
Since $D_{\alpha} \left[ A^{\alpha \beta}_{ij}D_{\beta}u^{j,k}  \right] = D_{\alpha}  F_{\alpha}^{i,k}$ in $Q_{7}$, we find that for any $ 1 \leq i \leq N $,
\begin{equation}\label{PET235} 
\sum_{ 1 \leq \alpha \leq n} D_{\alpha}U_{\alpha}^{i,k} = \sum_{1 \leq \alpha \leq n} D_{\alpha} \left[ \sum_{1 \leq j \leq N}  \sum_{ 1 \leq \beta \leq n } A_{\alpha \beta}^{ij,k} D_{\beta}u^{j,k} - F_{\alpha}^{i,k} \right]=0 
~  \text{ in } ~
Q_{6}^{k}.
\end{equation}

\begin{lem}\label{PETS240}
Let $\vec{n}_{k} = (n_{k}^{1}, n_{k}^{2}, \cdots, n_{k}^{n})$ be the outward normal vector on $\partial Q_{6}^{k}$. Then for any $ \alpha = 2, \cdots, n$ and $k \in K$, we have  that
\begin{equation}\label{PET243}
n_{k}^{\alpha} = - D_{\alpha}\varphi_{k} \, n_{k}^{1},
\qquad
n_{k+1}^{\alpha} = - D_{\alpha}\varphi_{k} \, n_{k+1}^{1}
\qquad \text{and} \qquad
n_{k}^{1} = - n_{k+1}^{1}
\end{equation}
on $\partial Q_{6}^{k} \cap \partial Q_{6}^{k+1}$. Also for any $ i = 2, \cdots, n$ and $k \in K$, we have  that
\begin{equation}\label{PET246}
n_{k}^{\alpha} = - \psi(x,0') \, n_{k}^{1}
\quad \text{and} \quad
n_{k+1}^{\alpha} = - \psi(x,0') \, n_{k+1}^{1}
\quad  \text{on} \quad  \partial Q_{6}^{k} \cap \partial Q_{6}^{k+1}.
\end{equation}
\end{lem} 

\begin{proof}
\eqref{PET243} holds from Definition \ref{composite cube}. \eqref{PET246} follows from \eqref{PET210} and \eqref{PET243}.
\end{proof}

\begin{lem}\label{PETS250}
With the assumption \eqref{PET220}, we have that
\begin{equation*}\label{}
\sum_{ 1 \leq \alpha \leq n} \left[ U_{\alpha}^{i,k}  n_{k}^{\alpha} + U_{\alpha}^{i,k+1}  n_{k+1}^{\alpha}  \right]= 0
\quad \text{in} \quad
\partial  Q_{6}^{k} \cap \partial Q_{6}^{k+1}
\qquad (k \in K).
\end{equation*}
\end{lem}

\begin{proof}
Fix $\eta \in C_{c}^{\infty}(Q_{6},\br^{N})$. By \eqref{PET220} and the weak formulation,
\begin{equation}\begin{aligned}\label{PET260}
& \sum_{k \in K} \int_{Q_{6}^{k}} \sum_{ 1 \leq \alpha \leq n} U_{\alpha}^{i,k} D_{\alpha} \eta^{i} \, dx \\
& \quad = \int_{Q_{6}}  \sum_{ 1 \leq \alpha \leq n}  \left[ \sum_{1 \leq j \leq N} \sum_{1 \leq \beta \leq n} A_{\alpha \beta}^{ij,k} D_{\beta}u^{j,k} - F_{i} \right] D_{\alpha} \eta^{i} \, dx 
=0.
\end{aligned}\end{equation}
for any $1 \leq i \leq N$. Apply the integration by parts. Then by \eqref{PET235} and \eqref{PET260},
\begin{equation*}\label{}
\sum_{k \in K} \int_{\partial Q_{6}^{k}}  \sum_{ 1 \leq \alpha \leq n}  \eta^{i} \, U_{\alpha}^{i,k} \, n_{k}^{\alpha} \, dS
= \sum_{k \in K} \int_{Q_{6}^{k}}  \sum_{ 1 \leq \alpha \leq n} \left[ D_{\alpha} U_{\alpha}^{i,k} \,  \eta^{i} + U_{\alpha}^{i,k} D_{\alpha}\eta^{i} \right] \, dx 
= 0,
\end{equation*}
for any $1 \leq i \leq N$. Since $\eta \in C_{c}^{\infty} \left ( Q_{6}, \br^{N} \right )$, we find that
\begin{equation}\label{PET266}
0 = \sum_{k \in K} \int_{ Q_{6}' }  \sum_{ 1 \leq \alpha \leq n}  \left[ \eta^{i} \, U_{\alpha}^{i,k} \, n_{k}^{\alpha} + \eta^{i} \, U_{\alpha}^{i,k+1} \, n_{k+1}^{\alpha}  \right] \left( \varphi_{k}(x'), x' \right) \, dx',
\end{equation}
for any $1 \leq i \leq N$. In view of  \eqref{GS130}, fix $z \in \partial Q_{6}^{k} \cap \partial Q_{6}^{k+1}$ satisfying that $Q_{r}(z) \subset Q_{6}^{k} \cup Q_{6}^{k+1}$. 
Then we obtain from \eqref{graph_con} and \eqref{PET220} that 
\begin{equation}\label{PET270}
\left| U_{\alpha}^{i,k} \, n_{k}^{\alpha}+ U_{\alpha}^{i,k+1} \, n_{k+1}^{\alpha}  \right|^{2}  \left( \varphi_{k}(x'), x' \right) 
\text{ is bounded and continuous in } Q_{r}'(z') 
\end{equation}
for any $1 \leq \alpha \leq n$, $1 \leq i \leq N$ and $k \in K$. 

With \eqref{PET270}, for $m = 1,2,\cdots$ and $ 1 \leq i \leq N$,  choose   $\eta^{i,1}_{m} : \left ( z^{1}-r,z^{1}+r \right ) \to [0,1] $ with $\eta^{i,1}_{m} \in C_{c}^{\infty} \left (z^{1}-r,z^{1}+r \right )$ and  $\eta^{i}_{m} : Q_{r}'(z') \to \br$ with $\eta^{i}_{m} \in C_{c}^{\infty} \left( Q_{r}'(z') \right)$  so that 
\begin{equation*}\label{}
\eta^{i,1}_{m} = 1
\quad \text{ in } \quad 
\left[ z^{1}- (1-2^{-m})r, z^{1}+ (1-2^{m})r \right ]
\qquad (m =1,2,\cdots),
\end{equation*}
and
\begin{equation*}\label{}
\eta^{i} _{m} (x') \xrightarrow[]{m \to \infty}  \sum_{ 1 \leq \alpha \leq n}  \left[ U_{\alpha}^{i,k} \, n_{k}^{\alpha} +  U_{\alpha}^{i,k+1} \, n_{k+1}^{i}  \right] \left( \varphi_{k}(x'), x' \right) 
\quad \text{ in } \quad L^{2} \left( Q_{r}'(z') \right).
\end{equation*}
Take $\eta^{i}_{m} = \eta^{i,1}_{m} \cdot \eta^{i}_{m} \in C_{c}^{\infty}(Q_{r}(z))$ instead of $\eta^{i}$ in \eqref{PET266}. Then by letting $m \to \infty$,
\begin{equation*}\label{}
0 = \int_{ Q_{r}'(z') }  \sum_{1 \leq \alpha \leq n}  \left|  \sum_{ 1 \leq \alpha \leq n}  U_{\alpha}^{i,k} \, n_{k}^{\alpha} + U_{\alpha}^{i,k+1} \, n_{k+1}^{i}  \right|^{2}  \left( \varphi_{k}(x'), x' \right) \, dx'.
\end{equation*}
So we find from \eqref{PET270} that that 
\begin{equation*}\label{}
\sum_{ 1 \leq \alpha \leq n} \left[ U_{\alpha}^{i,k}(z) \, n_{k}^{\alpha}(z) + U_{\alpha}^{i,k+1}(z) \, n_{k+1}^{i}(z)  \right]= 0.
\end{equation*}
Since $ z \in \partial Q_{6}^{k} \cap \partial Q_{6}^{k+1}$ was arbitrary chosen, the lemma follows.
\end{proof}

\begin{lem}\label{PETS400}
For any $1 \leq i \leq N$ and $\eta \in C_{c}^{\infty}(Q_{5}, \br^{N})$, we have that
\begin{equation*}\begin{aligned}
0 & =  \sum_{ k \in K } \int_{\partial Q_{5}^{k}} \partial_{t'}^{p'} \Big( U_{1}^{i,k} \big( x+(\psi,t' ) \big) \Big) \Big|_{t'=0'}  \eta^{i} \, n_{k}^{1}  \, dS \\
&\quad - \sum_{ k \in K } \sum_{ 2 \leq \alpha \leq n}  \int_{\partial Q_{5}^{k}} \partial_{t'}^{p'}  \Big(  \psi_{\alpha}(x,t') \,  U_{\alpha}^{i,k} \big( x+(\psi,t') \big) \Big) \Big|_{t'=0'}  \eta^{i}  \,  n_{k}^{1}  \,  dS,
\end{aligned}\end{equation*}
where $\vec{n}_{k} = (n_{k}^{1}, n_{k}^{2}, \cdots, n_{k}^{n})$ is the outward normal vector on $\partial Q_{6}^{k}$.
\end{lem}

\begin{proof}
It follows from Lemma \ref{PETS250} that for any $1 \leq i \leq N$ and $ k \in K$, 
\begin{equation*}\begin{aligned}\label{}
U_{1}^{i,k} - \sum_{ 2 \leq \alpha \leq n}  D_{\alpha}\varphi_{k}  \,U_{\alpha}^{i,k} 
= U_{1}^{i,k+1} - \sum_{ 2 \leq \alpha \leq n}  D_{\alpha}\varphi_{k}  \,U_{\alpha}^{i,k+1} 
\quad \text{on} \quad
\partial Q_{6}^{k} \cap \partial Q_{6}^{k+1}.
\end{aligned}\end{equation*}
Fix $x = \left( \varphi_{k}(x'), x' \right) \in \partial Q_{5}^{k} \cap \partial Q_{5}^{k+1}$ $( k \in K)$. Then for any $t' \in Q_{1}'$, we have that
\begin{equation*}\begin{aligned}\label{}
& U_{1}^{i,k} \left( \varphi_{k}(x'+t'), x'+t' \right) - \sum_{ 2 \leq \alpha \leq n}  D_{\alpha}\varphi_{k} \left( x'+t' \right)  \,U_{\alpha}^{i,k} \left( \varphi_{k}(x'+t'), x'+t' \right) \\
& \  = U_{1}^{i,k+1} \left( \varphi_{k}(x'+t'), x'+t' \right)  - \sum_{ 2 \leq \alpha \leq n}  D_{\alpha}\varphi_{k} \left( x'+t' \right) \,U_{\alpha}^{i,k+1} \left( \varphi_{k}(x'+t'), x'+t' \right),
\end{aligned}\end{equation*}
for any $1 \leq  i \leq N$. Since  $x =  \left( \varphi_{k}(x'), x' \right) \in \partial Q_{5}^{k} \cap \partial Q_{5}^{k+1} ( \subset Q_{6})$, we find from \eqref{PET210} that
\begin{equation*}\label{}
\psi_{\alpha}(x,t)
=  \psi_{\alpha} \big( (\varphi_{k}(x'),x') ,t' \big)  
= D_{\alpha}\varphi_{k} \left( x'+t' \right)  
\qquad \left( t' \in Q_{1}' , ~ 1 \leq \alpha \leq n \right ).
\end{equation*}
Also from \eqref{GS230} and that $x = \left( \varphi_{k}(x'), x' \right) \in \partial Q_{5}^{k} \cap \partial Q_{5}^{k+1} ( \subset Q_{6})$, we find  that 
\begin{equation*}\label{}
x + \left( \psi(x,t'),t' \right)
= \left( \varphi_{k}(x'+t'), x'+t' \right)  
\qquad \left( t' \in Q_{1}' \right ).
\end{equation*}
So by combining the above three equality,
\begin{equation*}\begin{aligned}\label{}
& U_{1}^{i,k} \left( x + (\psi,t') \right) - \sum_{ 2 \leq \alpha \leq n}  \psi_{\alpha}(x,t')  U_{\alpha}^{i,k} \left( x+(\psi,t') \right) \\
& \quad = U_{1}^{i,k+1} \left( x+(\psi,t') \right) - \sum_{ 2 \leq \alpha \leq n}  \psi_{\alpha}(x,t')   U_{\alpha}^{i,k+1} \left( x+(\psi,t') \right)
\end{aligned}\end{equation*}
for any $t' \in Q_{1}'$ and $1 \leq i \leq N$. Then by using that $\partial_{t'}^{p'} = \partial_{t^{2}}^{p_{2}} \cdots \partial_{t^{n}}^{p_{n}}$, we find that
\begin{equation*}\begin{aligned}\label{}
& \left. \partial_{t'}^{p'} \left[ U_{1}^{i,k} \left( x + (\psi,t') \right) - \sum_{ 2 \leq \alpha \leq n}  \psi_{\alpha}(x,t')  U_{\alpha}^{i,k} \left( x+(\psi,t') \right)   \right] \right|_{t'=0'} \\
& \quad = \left. \partial_{t'}^{p'} \left[ U_{1}^{i,k+1} \left( x+(\psi,t') \right) - \sum_{ 2 \leq \alpha \leq n}  \psi_{\alpha}(x,t')   U_{\alpha}^{i,k+1} \left( x+(\psi,t') \right) \right] \right|_{t'=0'}.
\end{aligned}\end{equation*}
Since $x = \left( \varphi_{k}(x'), x' \right) \in \partial Q_{5}^{k} \cap \partial Q_{5}^{k+1}$ $(k \in K)$ was arbitrary chosen and $\eta \in C_{c}^{\infty} \left( Q_{5} ,\br^{N} \right)$, the lemma holds from that $n_{k}^{1} = -n_{k+1}^{1}$ for any $k \in K$.
\end{proof}

Let $\vec{n}_{k} = \left( n_{k}^{1}, \cdots, n_{k}^{n} \right)$ be the outward normal vector on $\partial Q_{5}^{k}$. For any $\eta \in C_{c}^{\infty} \left( Q_{5}, \br^{N} \right)$, we have from Lemma \ref{PETS400} that
\begin{equation}\label{PET450}
I^{i} + II^{i} + III^{i}=0
\qquad (1 \leq i \leq N),
\end{equation}
where
\begin{equation*}\begin{aligned}
I^{i} & =  \sum_{ k \in K } \int_{\partial Q_{5}^{k}} \partial_{t'}^{p'} \Big( U_{1}^{i,k} \big( x+(\psi,t' ) \big) \Big) \Big|_{t'=0'}  \eta^{i} \, n_{k}^{1}  \, dS, \\
II^{i} &= - \sum_{ k \in K } \sum_{2 \leq \alpha \leq n}  \int_{\partial Q_{5}^{k}} \partial_{t'}^{p'}  \Big( \big[ \psi_{\alpha}(x,t') - \psi_{\alpha}(x,0')  \big] U_{\alpha}^{i,k} \big( x+(\psi,t') \big) \Big) \Big|_{t'=0'}  \eta^{i}  \,  n_{k}^{1}  \,  dS, \\
III^{i} &=  - \sum_{ k \in K } \sum_{2 \leq \alpha \leq n}  \int_{\partial Q_{5}^{k}} \partial_{t'}^{p'}  \Big( \psi_{\alpha}(x,0') \, U_{\alpha}^{i,k} \big( x+(\psi,t') \big) \Big) \Big|_{t'=0'}  \eta^{i} \,  n_{k}^{1}  \,   dS ,
\end{aligned}\end{equation*}
Then by Lemma \ref{PETS240}, for any $\eta^{i} \in C_{c}^{\infty}(Q_{5})$.
\begin{equation*}\begin{aligned}\label{} 
III ^{i}
& = - \sum_{ k \in K } \sum_{2 \leq \alpha \leq n}  \int_{\partial Q_{5}^{k}} \partial_{t'}^{p'}  \big[ \psi_{\alpha}(x,0') \, U_{\alpha}^{i,k} \big( x+(\psi,t') \big) \big] \Big|_{t'=0'}  \eta^{i} \,  n_{k}^{1}  \,   dS \\
& =  \sum_{ k \in K } \sum_{2  \leq \alpha \leq n } \int_{\partial Q_{5}^{k}} \partial_{t'}^{p'}  \big[ U_{\alpha}^{i,k} \big( x+(\psi,t') \big) \big]  \Big|_{t'=0'} \eta^{i} \,  n_{k}^{\alpha}  \,   dS.
\end{aligned}\end{equation*}
Thus
\begin{equation}\label{PET460} 
I^{i} + III^{i} 
=  \sum_{ k \in K } \sum_{ 1 \leq \alpha \leq n } \int_{\partial Q_{5}^{k}} \partial_{t'}^{p'}  \big[ U_{\alpha}^{i,k} \big( x+(\psi,t') \big) \big]  \Big|_{t'=0'} \eta^{i} \,  n_{k}^{\alpha}  \,   dS.
\end{equation}
We first estimate $I^{i}$ and $III^{i}$ $(1 \leq i \leq N)$. Recall from \eqref{GS280} in Lemma \ref{GSS275} that
\begin{equation}\label{PET470} 
\partial_{t'}^{p'}  h \big|_{t'=0'} 
 = \sum_{0' \leq q' \leq p' } \Big( \begin{array}{c} p' \\ q' \end{array} \Big)
\pi^{q'} 
D^{ (|q'|,p'-q') } h + \sum_{|\xi| \leq |p'|-1} P_{\xi}^{ p' } (x) D_{\xi} h 
~ \text{ in } ~  Q_{5}.
\end{equation}
With \eqref{PET220}, we find from \eqref{PET460} and \eqref{PET470}, for any $\eta^{i} \in C_{c}^{\infty} \left( Q_{5}, \br^{N} \right)$, 
\begin{equation}\begin{aligned}\label{PET480} 
I^{i}+III^{i} 
& =  \sum_{1 \leq \alpha \leq n} \sum_{ k \in K } \int_{ \partial Q_{5}^{k} }  \sum_{0' \leq q' \leq p' } \Big( \begin{array}{c} p' \\ q' \end{array} \Big)
 \pi^{q'} D^{ (|q'|,p'-q') } U_{\alpha}^{i,k} \,  \eta^{i} \  n_{k}^{\alpha} \, dS \\
& \quad + \sum_{1 \leq \alpha \leq n} \sum_{ k \in K } \int_{ \partial Q_{5}^{k} }  
\sum_{|\xi| \leq |p'|-1} P_{\xi}^{ p' }(x) 
D_{\xi}U_{\alpha}^{i,k} \,  \eta^{i}  \ n_{k}^{\alpha} \, dS \\
& : = IV_{1}^{i} + IV_{2}^{i}.
\end{aligned}\end{equation}
for any $1 \leq i \leq N$.

We estimate $IV_{1}^{i}$ by applying the integration by parts with respect to $x^{1}$-variable. In view of \eqref{cof_com}, \eqref{data_con} and \eqref{PET170},  for the simplicity of the calculation, we set 
\begin{equation}\label{PET490}
E_{z,r} = \sum_{k \in K} \left[ \| D^{m+1}u \|_{L^{\infty}(Q_{r}^{k}(z))}  +  \| u \|_{C^{m,\gamma}( Q_{7}^{k})} +  \| F \|_{C^{m,\gamma}( Q_{7}^{k})} \right],
\end{equation}
for any $Q_{r}(z) \subset Q_{5}$. To handle $IV_{2}^{i}$, we obtain the following Lemma \ref{PETS500}. In Lemma \ref{PETS500}, $f_{k}$ only need to be defined on the set 
\begin{equation*}
\partial Q_{5}^{k} \cap \left( \bigcup_{l \in \left( \{ k,k+1 \} \cap  K_{-} \right)} \left\{ \left( \varphi_{l}(x'),x' \right) : x \in Q_{5}' \right\} \right) 
~ \left(\subset \overline{Q_{6}^{k}} \right)
\end{equation*}
(where the inclusion holds from \eqref{GS180}) but we define on the set $\overline{Q_{6}^{k}}$  for the simplicity.

\begin{lem}\label{PETS500}
Under the assumption \eqref{GS180}, suppose that $g_{k} : \overline{ Q_{6}^{k} }  \to \br$  and $h :  \overline{Q_{6} } \to \br$ satisfies that $g_{k}, h  \in C^{\gamma} \left( \overline{Q_{6}^{k}} \right)$ for any $k \in K$. Then for any  $\eta \in C_{c}^{\infty}(Q_{5})$,
\begin{equation*}\label{}
\sum_{k \in K} \int_{ \partial Q_{5}^{k} } g_{k}h \eta \, n_{k}^{1} \,  dS 
= \sum_{ k \in  K_{-}  } \int_{ Q_{5} } \bg_{k} \left[ g_{k-1}-g_{k},h \right] H_{k}  D_{1} \eta \,  dx,
\end{equation*}
where the linear operator $\bg_{k} : C^{\gamma} \left( \partial Q_{6}^{k} \cap \partial Q_{6}^{k-1} \right) \times C^{\gamma} \left( \partial Q_{6}^{k} \cap \partial Q_{6}^{k-1} \right)  \to C^{\gamma} \left( Q_{5}' \right) $ is  defined as
\begin{equation*}
\bg_{k}[g,h](x) =  \left[ g h \right]  \big( \varphi_{k}(x'),x' \big).
\end{equation*}
\end{lem}

\begin{proof}
By the definition of $K_{-}$ in \eqref{GS180},
\begin{equation*}\begin{aligned}
\sum_{k \in K} \int_{\partial Q_{5}^{k}} g_{k} h \eta \, n_{k}^{1} \,  dS 
& =  - \sum_{k+1 \in K_{-}} \int_{ Q_{5} \cap \left\{ \left( \varphi_{k+1}(x'),x' \right) : x' \in Q_{5}' \right\} } g_{k} h \eta \, n_{k}^{1} \,  dS  \\
& \qquad  - \sum_{k \in K_{-}} \int_{ Q_{5} \cap  \left\{ \left( \varphi_{k}(x'),x' \right) : x' \in Q_{5}' \right\} } g_{k} h \eta \, n_{k}^{1} \,  dS. 
\end{aligned}\end{equation*}
By using \eqref{GS910} and integration by parts with respect to $x^{1}$-variable,
\begin{equation*}\begin{aligned}
& \sum_{k+1 \in K_{-}} \int_{ Q_{5} \cap \left\{ \left( \varphi_{k+1}(x'),x' \right) : x' \in Q_{5}' \right\} } g_{k} h \eta \, n_{k}^{1} \,  dS \\
& \quad = \sum_{k+1 \in K_{-}}  \int_{ Q_{5} } \left[ g_{k} h \right] \big( \varphi_{k+1}(x'),x' \big)  D_{1}\eta \, H_{k+1} \, dx \\
& \quad = \sum_{k \in K_{-} }  \int_{ Q_{5} } \left[ g_{k-1} h \right] \big( \varphi_{k}(x'),x' \big)  D_{1}\eta \, H_{k} \, dx,
\end{aligned}\end{equation*}
and
\begin{equation*}\begin{aligned}
& \sum_{k \in K_{-}} \int_{ Q_{5} \cap \left\{ \left( \varphi_{k}(x'),x' \right) : x' \in Q_{5}' \right\} } g_{k} h \eta \, n_{k}^{1} \,  dS \\
& \quad = \sum_{k \in K_{-}}  \int_{ Q_{5} } \left[ g_{k} h \right] \big( \varphi_{k}(x'),x' \big)  D_{1}\eta \, (1-H_{k}) \, dx,
\end{aligned}\end{equation*}
where we used that $\eta \in C_{c}^{\infty}(Q_{5})$. So the lemma holds by combining the above three equality and that $\sum_{k \in K_{-}}  \int_{ Q_{5} } \left[ g_{k} h \right] \big( \varphi_{k}(x'),x' \big)  D_{1}\eta \, dx= 0$.
\end{proof}

We now estimate $I^{i}+III^{i}$ in Lemma \ref{PETS600} and $II^{i}$ in Lemma \ref{PETS700}.  We remark that $G_{1}$ in Lemma \ref{PETS600} and Lemma \ref{PETS700} are H\"{o}lder continuous in each subregion $Q_{5}^{k}$ $(k \in K)$. Moreover, $\bg_{1,k}[h]$ only depends on $x'$-variables and $G_{1}$ might have big jumps on the graphs $\left \{  \big( \varphi_{k}(x'), x' \big) : x' \in Q_{5}' \right\}$ $(k \in K_{-})$.

\begin{lem}\label{PETS600}
For any $\eta \in C_{c}^{\infty} \left( Q_{r}(z), \br^{N} \right)$ with $Q_{r}(z) \subset Q_{5}$, we have that
\begin{equation*}\label{}\begin{aligned}
& \left| I^{i} + III^{i} - \sum_{ k \in K } \int_{ Q_{r}^{k}(z) }  \sum_{0' \leq q' \leq p' } 
\Big( \begin{array}{c} p' \\ q' \end{array} \Big)
\pi^{q'} 
D_{1} \left [ D^{ (|q'|,p'-q') } U_{1}^{i,k}  \eta^{i}  \right] \, + \, G_{1}^{i} D_{1}\eta^{i}  \, dx \right| \\
& \quad  \leq c r^{\frac{1}{4}} \int_{Q_{r}(z)} \epsilon |D\eta|^{2} +   \epsilon^{-2}  |\eta|^{2} + E_{z,r}^{2}   \, dx,
\end{aligned}\end{equation*}
for any $1 \leq i \leq N$ and $\epsilon \in (0,1]$, where
\begin{equation}\label{PET620}
G_{1} ^{i}=  \sum_{ k \in  K_{-}  } \sum_{1 \leq \alpha \leq n} 
\ \bg_{\alpha}^{k} \left[ U_{\alpha}^{i,k-1} -  U_{\alpha}^{i,k} \right]H_{k} 
\end{equation}
with the linear operator $\ \bg_{\alpha}^{k} : C^{m-1,\gamma} \left( \partial Q_{6}^{k} \cap \partial Q_{6}^{k-1} \right) \to C^{\gamma} \left( Q_{5}' \right) $  $(k \in K_{-})$  defined as
\begin{equation}\label{PET625} 
\ \bg_{\alpha}^{k}[h](x') = \sum_{|\xi| \leq |p'|-1} \left[ \pi_{\alpha} \,  P_{\xi}^{ p' }   D_{\xi}  h \right ] \big( \varphi_{k}(x'),x' \big).
\end{equation}
\end{lem}

\begin{proof}
With respect to $x^{1}$-variable, use integration by parts to $IV_{1}$ in \eqref{PET480}. Then
\begin{equation*}\begin{aligned}\label{}
IV_{1}^{i}  & =  \sum_{k \in K} \sum_{ 1 \leq \alpha \leq n } \int_{ Q_{5}^{k} }  \sum_{0' \leq q' \leq p' } 
\Big( \begin{array}{c} p' \\ q' \end{array} \Big)
D_{\alpha} \left[ \pi^{q'} \right]
D^{ (|q'|,p'-q') } U_{\alpha}^{i,k} \, \eta^{i}  \, dx \\
& \quad  + \sum_{k \in K}  \sum_{ 1 \leq \alpha \leq n }  \int_{Q_{5}^{k} } \sum_{0' \leq q' \leq p' } 
\Big( \begin{array}{c} p' \\ q' \end{array} \Big)
\pi^{q'} 
D_{\alpha} \Big[ D^{ (|q'|,p'-q') } U_{\alpha}^{i,k}  \eta^{i} \Big] \, dx.
\end{aligned}\end{equation*}
By that $Q_{r}(z) \subset Q_{5}$,
\begin{equation}\begin{aligned}\label{PET640}
IV_{1}^{i}  & =  \sum_{k \in K}  \sum_{ 1 \leq \alpha \leq n }  \int_{ Q_{r}^{k}(z) }  \sum_{0' \leq q' \leq p' } 
\Big( \begin{array}{c} p' \\ q' \end{array} \Big)
D_{\alpha} \left[ \pi^{q'} \right]
D^{ (|q'|,p'-q') } U_{\alpha}^{i,k} \, \eta^{i}  \, dx \\
& \quad  + \sum_{k \in K}  \sum_{ 1 \leq \alpha \leq n }  \int_{Q_{r}^{k}(z) } \sum_{0' \leq q' \leq p' } 
\Big( \begin{array}{c} p' \\ q' \end{array} \Big)
\pi^{q'} 
D_{\alpha} \Big[ D^{ (|q'|,p'-q') } U_{\alpha}^{i,k}  \eta^{i} \Big] \, dx.
\end{aligned}\end{equation}
From \eqref{gradient_graphs} and \eqref{GS260}, we have that $|\pi_{2}|, \cdots, |\pi_{n}| \leq c$. So by Young's inequality,
\begin{equation*}\begin{aligned}\label{}
& \bigg| \sum_{ k \in K }   \int_{ Q_{r}^{k}(z) }  \sum_{ 1 \leq \alpha \leq n } \sum_{0' \leq q' \leq p' } 
\Big( \begin{array}{c} p' \\ q' \end{array} \Big)
D_{\alpha} \left[ \pi^{q'} \right]
D^{ (|q'|,p'-q') } U_{\alpha}^{i,k}  \eta^{i} \, dx \bigg|  \\
& \quad \leq c  \int_{ Q_{r}(z) }  \big| D\pi \big| E_{z,r}   \left| \eta \right| \, dx \\
& \quad \leq c  \int_{ Q_{r}(z) } r^{\frac{1}{4}} E_{z,r}^{2} + r^{-\frac{1}{4}} |D\pi|^{2} \left| \eta \right| ^{2} \, dx,
\end{aligned}\end{equation*}
where the definition of $E_{z,r}$ in \eqref{PET490} is used. 
So by \eqref{PET640} and Lemma \ref{GSS800},
\begin{equation*}\begin{aligned}\label{}
& \left| IV_{1}^{i} - \sum_{k \in K}   \int_{Q_{r}^{k}(z) } \sum_{ 1 \leq \alpha \leq n }  \sum_{0' \leq q' \leq p' } 
\Big( \begin{array}{c} p' \\ q' \end{array} \Big)
\pi^{q'} 
D_{1} \left[ D^{ (|q'|,p'-q') } U_{1}^{i,k}  \eta^{i} \right] \, dx \right|  \\
& \quad \leq c r^{\frac{1}{4}} \int_{Q_{r}(z)} \epsilon |D\eta|^{2} +   \epsilon^{-2}  |\eta|^{2} + E_{z,r}^{2}   \, dx.
\end{aligned}\end{equation*}

We next estimate $IV_{2}$. By \eqref{GS410} and Lemma \ref{PETS240}, we find that 
\begin{equation*}\label{}
n_{k}^{\alpha}  = - \pi_{\alpha} \, n_{k}^{1}
\quad \text{ on } \quad
\partial Q_{5}^{k} \cap \partial Q_{5}^{k+1},
\end{equation*}
for any $ i = 1,2, \cdots, n$ and $k \in K$. So by \eqref{PET480},
\begin{equation*}\begin{aligned}\label{}
IV_{2}^{i} & = \sum_{ k \in K } \int_{ \partial Q_{5}^{k} }   \sum_{1 \leq \alpha \leq n} 
\sum_{|\xi| \leq |p'|-1} P_{\xi}^{ p' }(x) 
D_{\xi}U_{\alpha}^{i,k}  \, \eta^{i}  \ n_{k}^{\alpha} \, dS \\
& =  -  \sum_{ k \in K }  \int_{ \partial Q_{5}^{k} }  \sum_{1 \leq \alpha \leq n} 
\sum_{|\xi| \leq |p'|-1} P_{\xi}^{ p' }(x) 
D_{\xi}U_{\alpha}^{i,k}  \, \eta^{i}  \, \pi_{\alpha} \, n_{k}^{1} \, dS.
\end{aligned}\end{equation*}
With \eqref{GS130}, we have from \eqref{PET220}, \eqref{GS260} and the definition of $P_{\xi}^{p'}$ in Lemma \ref{GSS275}  that 
\begin{equation*}\label{}
D_{\xi} U_{\alpha}^{i,k}, \pi_{\alpha} P_{\xi}^{p'} \in C^{\gamma} \left( \partial Q_{6}^{k} \cap \partial Q_{6}^{k-1} \right)
\qquad (k \in K, \, i \in \{ 1, \cdots, n \} ),
\end{equation*}
for any $|\xi| \leq |p'| - 1$. So take $g_{k} = D_{\xi}U_{\alpha}^{i,k}$  and $h = \pi_{\alpha} P_{\xi}^{p'} $ in Lemma \ref{PETS500} to find that
\begin{equation*}\label{}
IV_{2}^{i} = - \sum_{ k \in K } \int_{ Q_{5}^{k} }   \, G_{1} D_{1}\eta^{i}  \, dx
\end{equation*}
with
\begin{equation*}\label{}
G_{1}^{i} =  \sum_{ k \in  K_{-}  } \sum_{1 \leq \alpha \leq n} \sum_{|\xi| \leq |p'|-1} 
\bg_{k} \left[ D_{\xi} U_{\alpha}^{i,k-1} -  D_{\xi} U_{\alpha}^{i,k} , \pi_{\alpha} P_{\xi}^{p'}  \right] H_{k}
\end{equation*}
where $\bg_{k}$ is defined in Lemma \ref{PETS500}. So by comparing the definition of $\bg_{k}$ in Lemma \ref{PETS500} and $\ \bg_{\alpha}^{k}$ in \eqref{PET625}, the lemma follows form that $IV_{1}^{i}+IV_{2}^{i} = I^{i}+III^{i}$ in \eqref{PET480} $( 1 \leq i \leq N)$ and that $Q_{r}(z) \subset Q_{5}$.
\end{proof}

\begin{lem}\label{PETS700}
For any $\eta \in C_{c}^{\infty} \left( Q_{r}(z) , \br^{N} \right)$ with $Q_{r}(z) \subset Q_{5}$, we have that
\begin{equation*}\begin{aligned}
II^{i} = & - \int_{Q_{r}(z)} G_{1}^{i} D_{1}\eta^{i} \, dx,
\end{aligned}\end{equation*}
for any $1 \leq i \leq N$, where 
\begin{equation*}\label{}
G_{1}^{i} = \sum_{k \in K_{-}} \sum_{ 1 \leq \alpha \leq n }  \ \bg_{\alpha}^{k} \left[ U_{\alpha}^{i,k-1} -  U_{\alpha}^{i,k} \right]  H_{k},
\end{equation*}
with the linear operator $\ \bg_{\alpha}^{k} : C^{m-1,\gamma} \left( \partial Q_{6}^{k} \cap \partial Q_{6}^{k-1} \right) \to C^{\gamma} \left( Q_{5}' \right) $  $(k \in K_{-})$  defined as
\begin{equation*}
\ \bg_{\alpha}^{k}[h](x') =  \sum_{1 \leq \alpha \leq n} \sum_{ 0' < q' \leq p' } \Big( \begin{array}{c} p' \\ q' \end{array} \Big) D^{q'} D_{\alpha}\varphi_{k}(x') \, D^{p'-q'} \left[  h\big(  \varphi_{k}(x'), x' \big) \right].
\end{equation*}
\end{lem}

\begin{proof}
Recall from \eqref{PET450} that
\begin{equation*}\begin{aligned}
II^{i} = - \sum_{ k \in K } \sum_{2 \leq \alpha \leq n}  \int_{\partial Q_{5}^{k}} \partial_{t'}^{p'}  \Big( \big[ \psi_{\alpha}(x,t') - \psi_{\alpha}(x,0')  \big] U_{\alpha}^{i,k} \big( x+(\psi,t') \big) \Big) \Big|_{t'=0'}  \eta^{i}  \,  n_{k}^{1}  \,  dS.
\end{aligned}\end{equation*}
By a direct calculation,
\begin{equation*}\begin{aligned}\label{} 
II^{i} = -  \sum_{ k \in K } \int_{\partial Q_{5}^{k}}  \sum_{2 \leq \alpha \leq n} \sum_{ 0' < q' \leq p' } 
\Big( \begin{array}{c} p' \\ q' \end{array} \Big) \partial_{t'}^{q'} \psi_{\alpha}(x,t')  \Big|_{t'=0'}
 \partial_{t'}^{p'-q'} \big[ U_{\alpha}^{i,k} \big( x+(\psi,t') \big) \big] \Big|_{t'=0'} \eta^{i}   \,  n_{k}^{1}  \,  dS.
\end{aligned}\end{equation*}
We now compute the terms in the above integral. From \eqref{PET210}, we have that
\begin{equation}\begin{aligned}\label{PET770}
\left. \partial_{t'}^{q'} \psi_{\alpha} \big( (\varphi_{k}(x'),x') ,t' \big)  \right|_{t'=0'} 
=  \left. \partial_{t'}^{q'} D_{\alpha}\varphi_{k}(x'+t')   \right|_{t'=0'} 
= D^{q'} D_{\alpha}\varphi_{k}(x'),
\end{aligned}\end{equation}
or any $i = 2, \cdots, n $ and $k \in K_{-}$. By \eqref{GS230}, for any $k \in K_{-}$, we have that
\begin{equation*}\begin{aligned}\label{}
(\varphi_{k}(x'),x') + \big( \psi \big( (\varphi_{k}(x'),x'),t' \big) ,t' \big)
& = \big( \varphi_{k}(x'),x' \big) + \big( \varphi_{k}(x'+t') - \varphi_{k}(x') ,t' \big) \\
& = \big( \varphi_{k}(x'+t') , x ' +  t' \big),
\end{aligned}\end{equation*}
which implies that
\begin{equation}\begin{aligned}\label{PET780}
& \left. \partial_{t'}^{p'-q'} \left[ h \Big( (\varphi_{k}(x'),x') + \big( \psi \big( (\varphi_{k}(x'),x'),t' \big) ,t' \big) \Big) \right] \right|_{t'=0'} \\
& \quad =  \left. \partial_{t'}^{p'-q'} \left[ h \big( \varphi_{k}(x'+t') , x ' +  t' \big) \right] \right|_{t'=0'} \\
& \quad = D^{p'-q'} \big[ h \big(  \varphi_{k}(x'), x' \big) \big].
\end{aligned}\end{equation}
With \eqref{PET770} and \eqref{PET780}, for any $0' < q' \leq p'$, take $g_{k} =  \partial_{t'}^{p'-q'} \big[ U_{\alpha}^{i,k} \big( x+(\psi,t') \big) \big] \Big|_{t'=0'} $ and $h =  \partial_{t'}^{q'} \psi_{\alpha}(x,t')  \Big|_{t'=0'}  $ in Lemma \ref{PETS500} to find that
\begin{equation}\label{PET750}
II^{i} = -  \sum_{ k \in  K_{-}  } \int_{Q_{5}^{k}} \sum_{2 \leq \alpha \leq n} \ \bg_{\alpha}^{k} \left [U_{\alpha}^{i,k-1} - U_{\alpha}^{i,k} \right ] H_{k}  D_{1} \eta^{i} \,  dx,
\end{equation}
for the operator $\ \bg_{\alpha}^{k} : C^{m-1,\gamma} \left(  \overline{ Q_{6}^{k} }  \right)  \to C^{\gamma} \left( Q_{5}' \right) $  defined as
\begin{equation*}
\ \bg_{\alpha}^{k}[h](x) =  \sum_{ 0' < q' \leq p' } \Big( \begin{array}{c} p' \\ q' \end{array} \Big) D^{q'} D_{\alpha}\varphi_{k}(x') \, D^{p'-q'} \left[  h\big(  \varphi_{k}(x'), x' \big) \right].
\end{equation*}
The lemma holds by that $\eta \in C_{c}^{\infty} \left( Q_{r}(z), \br^{N} \right)$ and  $D_{1}\varphi_{k}(x')=0$ $(k \in K_{-})$.
\end{proof}

By combining Lemma \ref{PETS600} and Lemma \ref{PETS700}, we obtain the following lemma.

\begin{lem}\label{PETS800}
For any $\eta \in C_{c}^{\infty} \left( Q_{r}(z), \br^{N} \right)$ with $Q_{r}(z) \subset Q_{5}$, we have that
\begin{equation*}\begin{aligned}
& \left| \sum_{k \in K } \int_{Q_{r}^{k}(z) }  \sum_{0' \leq q' \leq p' } \Big( \begin{array}{c} p' \\ q' \end{array} \Big)
\pi^{q'} D^{(|q'|,p'-q')} U_{\alpha}^{i,k} D_{\alpha} \eta^{i} \, dx - \int_{Q_{r}(z) } G_{1}^{i} D_{1} \eta^{i} \, dx \right| \\
& \quad \leq c r^{\frac{1}{4}} \int_{Q_{r}(z)} \epsilon \left| D\eta \right|^{2} +   \epsilon^{-2}  \left| \eta \right |^{2} + E_{z,r}^{2}   \, dx,
\end{aligned}\end{equation*}
for any $ 1 \leq i \leq N$ and $\epsilon \in (0,1]$, where
\begin{equation}\label{PET820}
G_{1}^{i} =  \sum_{ k \in  K_{-}  } \sum_{1 \leq \alpha \leq n} \ \bg_{\alpha}^{k} \left[ U_{\alpha}^{i,k-1} - U_{\alpha}^{i,k} \right]  H_{k} 
\end{equation}
with the linear operator $\ \bg_{\alpha}^{k} : C^{m-1,\gamma} \left( \partial Q_{6}^{k} \cap \partial Q_{6}^{k-1} \right) \to C^{\gamma} \left( Q_{5}' \right) $  $(k \in K_{-})$  defined as
\begin{equation*}\begin{aligned}\label{}
\ \bg_{\alpha}^{k}[h](x') 
& = \sum_{|\xi| \leq |p'|-1} \left[ \pi_{\alpha}  P_{\xi}^{ p' }  D_{\xi}  h \right] \big( \varphi_{k}(x'),x' \big) \\
& \quad + \sum_{ 0' < q' \leq p' } \left( \begin{array}{c} p' \\ q' \end{array} \right) D^{q'} D_{\alpha}\varphi_{k}(x') \, D^{p'-q'}  \left[ h\big(  \varphi_{k}(x'), x' \big) \right].
\end{aligned}\end{equation*}
\end{lem}

\begin{proof}
Since $I + II + III=0$, by  \eqref{PET480}, Lemma \ref{PETS600} and Lemma \ref{PETS700},
\begin{equation*}\label{}\begin{aligned}
& \bigg| \sum_{ k \in K } \int_{Q_{r}^{k}(z) }  \sum_{0' \leq q' \leq p' } 
\Big( \begin{array}{c} p' \\ q' \end{array} \Big)
\pi^{q'} D_{\alpha} \Big[ D^{ (|q'|,p'-q') } U_{\alpha}^{i,k}  \eta^{i}  \Big]  \, dx  -  \int_{ Q_{r}(z) } G_{1}^{i} D_{1} \eta^{i}  \, dx \bigg| \\ 
& \quad \leq c r^{\frac{1}{4}} \int_{Q_{r}(z)} \epsilon \left| D\eta \right|^{2} +   \epsilon^{-2}  \left| \eta \right|^{2} + E_{z,r}^{2}   \, dx.
\end{aligned}\end{equation*}
By a direct calculation,
\begin{equation*}
D_{\alpha} \Big[ D^{ (|q'|,p'-q') } U_{\alpha}^{i,k}  \eta^{i}  \Big]  = D^{ (|q'|,p'-q') } D_{\alpha}U_{\alpha}^{i,k}   \eta^{i}  +  D^{ (|q'|,p'-q') } U_{\alpha}^{i,k}  D_{\alpha}\eta^{i}.
\end{equation*}
From \eqref{PET220} and \eqref{PET235}, we have that $D_{\alpha}U_{\alpha}^{i,k} = 0$ in $Q_{r}^{k}(z)$. So the lemma holds.
\end{proof}

So we obtain the desired estimate of this subsection.

\begin{lem}\label{PETS900}
For any $\eta \in C_{c}^{\infty} \left( Q_{r}(z) , \br^{N} \right)$ with $Q_{r}(z) \subset Q_{5}$, we have that
\begin{equation*}\begin{aligned}\label{}
& \left| \int_{Q_{r}(z) } \sum_{1 \leq \alpha \leq n} \left[ \sum_{1 \leq j \leq N} \sum_{1 \leq \beta \leq n} A_{\alpha \beta}^{ij,k}   \sum_{0' \leq q' \leq p' } \Big( \begin{array}{c} p' \\ q' \end{array} \Big)
\pi^{q'} D^{(|q'|,p'-q')} D_{\beta}u^{j,k} + \nu_{\alpha}^{i} \right] D_{\alpha} \eta^{i} - G_{1}^{i} D_{1} \eta^{i}  \, dx \right| \\
& \quad \leq c  r^{n+\frac{1}{4}} \left[   \| D^{m+1}u \|_{L^{\infty}(Q_{r}(z))}^{2}   +  \sum_{k \in K}  \left[  \| u \|_{C^{m,\gamma}( Q_{7}^{k})}^{2} +  \| F \|_{C^{m,\gamma}( Q_{7}^{k})}^{2} \right]  \right] \\
& \qquad + c  r^{\frac{1}{4}} \int_{ Q_{r}(z) } \epsilon |D\eta|^{2}  + \epsilon^{-2} |\eta|^{2}  \, dx,
\end{aligned}\end{equation*}
for any $1 \leq i \leq N$, with $G_{1}^{i} : Q_{5} \to \br$ in Lemma \ref{PETS800} and
\begin{equation*}
\nu_{\alpha}^{i} = \sum_{k \in K} \sum_{0' \leq q' \leq p' } \Big( \begin{array}{c} p' \\ q' \end{array} \Big)
\pi^{q'} \Big[ D^{ (|q'|,p'-q') } U_{\alpha}^{i,k} - \sum_{1 \leq j \leq N} \sum_{1 \leq \beta \leq n} A_{\alpha \beta}^{ij,k}D^{ (|q'|,p'-q') } D_{\beta}u^{j,k} \bigg]  \chi_{Q_{7}^{k}}
\end{equation*}
for any $ 1 \leq \alpha \leq n$ and  $1 \leq i \leq N$. Moreover, we have that
\begin{equation}\label{PET930}
\left \| G_{1}^{i} \right \|_{C^{\gamma}(Q_{5}^{k})} + \sum_{k \in K} \left \| \nu_{\alpha}^{i} \right \|_{C^{\gamma}(Q_{5}^{k})} 
\leq c \sum_{k \in K} \left[   \| u \|_{C^{m,\gamma}( Q_{7}^{k})} +  \| F \|_{C^{m,\gamma}( Q_{7}^{k})} \right],
\end{equation}
for any $ 1 \leq \alpha \leq n$ and  $1 \leq i \leq N$.
\end{lem}

\begin{proof}
With \eqref{PET490} and \eqref{GS920}, the estimate holds from Lemma \ref{PETS800}. With \eqref{PET220}, we find from \eqref{PET130} and \eqref{PET820} in Lemma \ref{PETS800} that \eqref{PET930} holds.
\end{proof}

\subsection{Choosing test function}

We will choose a suitable test function in $Q_{r}(z)$ by using $\eta$ in \eqref{CTF650} (which will be used in Lemma \ref{PMTS200} and Lemma \ref{PMTS300}) With \eqref{GS130}, set $Z_{k} : Q_{7} \to \br$ $\big( k \in K \big)$ as
\begin{equation}\label{CTF110} 
Z_{k}(x) : 
= \left\{ \begin{array}{ccc}
\varphi_{k+1}(x') & \text{ if } & x^{1} > \varphi_{k+1}(x'), \\
x^{1} & \text{ if } & \varphi_{k}(x') < x^{1} \leq \varphi_{k+1}(x'), \\
\varphi_{k}(x') & \text{ if } & x^{1} \leq \varphi_{k}(x'),
\end{array}\right.
\end{equation}
which implies that
\begin{equation}\label{CTF120} 
(x^{1},x') \in Q_{7}^{l} 
\qquad \Longrightarrow \qquad
Z_{k}(x^{1},x')  = \left\{ \begin{array}{cl}
\varphi_{k+1}(x') & \text{ if } l > k , \\
x^{1} & \text{ if } l = k , \\
\varphi_{k}(x') & \text{ if } k > l.
\end{array}\right.
\end{equation}
We use the perturbation on $Q_{5}^{k}$ with respect to the point $z_{k} \left ( \in \br^{n} \right )$ defined as follows. For any $k \in K$ and fixed $z \in Q_{5}$, set
\begin{equation}\label{CTF140} 
z_{k} = \big( Z_{k}(z), z' \big).
\end{equation}
To use the perturbation, with \eqref{gradient_graphs}, we will check that $z_{k} \in \overline{ Q_{6}^{k} }$ in Corollary \ref{CTFS150}. However, $z_{k}$ depends on $z$. So later in  Lemma \ref{CTFS350}, we will handle the points $z_{k}$ $(k \in K)$ by using the graph functions $\varphi_{k}$ $(k \in K)$ which does not depends on $z$. In fact, we use perturbation on $Q_{5r}^{k}(z)$ with respect to the point $z_{k}$ in \eqref{CTF140}. So with \eqref{gradient_graphs}, we will check that $z_{k} \in \overline{ Q_{6r}^{k}(z) }$ in Lemma \ref{CTFS180}.

\begin{lem}\label{CTFS180}
For any $Q_{6r}(z) \subset Q_{6}$ and the corresponding points $z_{k}$ $(k \in K)$ in \eqref{CTF120}, we have that
\begin{equation}\label{CTF182}
Q_{5r}^{k}(z) \not = \emptyset  
\qquad \Lra \qquad
z_{k} \in \overline{ Q_{6r}^{k}(z) }
\qquad 
(k \in K).
\end{equation}
Moreover, for any $x \in Q_{5r}^{k}(z)$, there exists a path belongs to $ Q_{6r}^{k}(z) \cup \{ z_{k }\}$ with length less than $cr$ which connects between $x$ and $z_{k}$.
\end{lem}

\begin{proof}
Fix $k \in K$. Assume that $Q_{6r}(z) \subset Q_{6}$ and $Q_{5r}^{k}(z) \not = \emptyset  $. Then we have that $Q_{5}^{k} \supset Q_{5r}^{k}(z) \not = \emptyset$, and it follows from Corollary \ref{CTFS150} that $z_{k} \in \overline{ Q_{6}^{k} }$. So we only need to prove that $z_{k} \in Q_{6r}(z)  $.

Since $Q_{5r}^{k}(z) \not = \emptyset$, we have that $\inf_{Q_{5r}'(z')} \varphi_{k} \leq z^{1} + 5r$ and $\sup_{Q_{5r}'(z')} \varphi_{k+1} \geq z^{1} - 5r$, 
other-wise we have that $Q_{5r}^{k}(z) = \emptyset$. So by \eqref{gradient_graphs},
\begin{equation}\label{CTF190}
\varphi_{k}(x') < z^{1} + \frac{11r}{2}
\qquad \text{and} \qquad
\varphi_{k+1}(x') > z^{1} - \frac{11r}{2}
\qquad \left(x' \in Q_{5r}'(z') \right).
\end{equation}

To prove that $z_{k} \in Q_{6r}(z)  $, we consider three cases. (1) Assume that $z^{1} > \varphi_{k+1}(z')$. Then from \eqref{gradient_graphs} and \eqref{CTF190}, we have that $ z^{1} - 6r < \varphi_{k+1}(z') < z^{1} $, which implies that $z_{k} = (Z_{k}(z),z') = \left( \varphi_{k+1}(z'),z' \right) \in \overline{ Q_{6r}^{k}(z) }$ (2) Assume that $\varphi_{k}(z') \leq z^{1} \leq \varphi_{k+1}(z')$. Then we have from \eqref{CTF110} and \eqref{CTF140} that $z_{k} = z \in \overline{ Q_{6r}^{k}(z) } $. (3) Assume that $z_{1} < \varphi_{k}(z')$. Then from \eqref{gradient_graphs} and  \eqref{CTF140}, we have that $  z^{1} < \varphi_{k}(z') < z^{1} + 6r$, which implies that $z_{k} = (Z_{k}(z),z') = \left( \varphi_{k}(z'),z' \right) \in \overline{ Q_{6r}^{k}(z) }$. Since $k \in K$ was arbitrary chosen, \eqref{CTF182} holds.

\sskip

We choose a path connecting $x \in Q_{5r}^{k}(z)$ and $z_{k} \in \overline{ Q_{6r}^{k}(z) }$ as follows. Set 
\begin{equation*}\label{}
{\it l}(s) = \left( \min \left\{ z^{1} + \frac{11r}{2}, \varphi_{k+1}(sz' + (1-s)x') \right\}, sz' + (1-s)x' \right) 
\qquad (s \in [0,1])
\end{equation*}
The path is a connected path of (1) the line connecting $x$ and ${\it l}(0)$, (2) the path ${\it l}(s)$ $(s \in [0,1])$ connecting  ${\it l}(0)$ and ${\it l}(1)$, (3) the line connecting ${\it l}(1)$ and $z_{k}$.  Since $x \in Q_{5r}^{k}(z)$ and $z_{k} \in \overline{ Q_{6r}^{k}(z) }$, one can show that this new path belongs to $Q_{6r}^{k}(z) \cup \{ z_{k }\}$ and has length less than $50nr$.
\end{proof}

The following corollary is a special case of Lemma \ref{CTFS180}.

\begin{cor}\label{CTFS150}
For any $z \in Q_{5}$ and the corresponding points $z_{k}$ $(k \in K)$ in \eqref{CTF120}, we have that
\begin{equation*}\label{}
Q_{5}^{k} \not  = \emptyset
\qquad \Lra \qquad
z_{k} \in \overline{ Q_{6}^{k} }
\qquad 
(k \in K).
\end{equation*}
\end{cor}

\begin{proof}
By taking $z=0$ and $r=1$ in Corollary \ref{CTFS150}, the lemma follows.
\end{proof}

Recall from \eqref{GS280} that
\begin{equation}\label{CTF640} 
\partial_{t'}^{p'} h \big|_{t'=0'} 
= \sum_{0' \leq q' \leq p' } \Big( \begin{array}{c} p' \\ q' \end{array} \Big)
\pi^{q'} 
D^{ (|q'|,p'-q') } h
+ \sum_{|\xi| \leq |p'|-1} P_{\xi}^{ p' } (x) D_{\xi} h 
~ \text{ in } ~ Q_{5}.
\end{equation}
For any $z \in Q_{5}$ and $p' \geq 0'$ with $|p'|=m$, the test function will be chosen by using the following $\eta : Q_{5} \to \br^{N}$ defined as 
\begin{equation}\begin{aligned}\label{CTF650}
\eta 
& = \partial_{t}^{p'}u \big|_{t=0} - \sum_{k \in K , \, Q_{5}^{k} \not = \emptyset} \sum_{0' \leq q' \leq p' } 
\Big( \begin{array}{c} p' \\ q' \end{array} \Big)
\pi^{q'} \big( Z_{k},x' \big)
D^{ ( |q'|, p'-q' ) } u_{k}(z_{k}) \\
& \quad - \sum_{k \in K, \,  Q_{5}^{k} \not = \emptyset }  \sum_{|\xi| \leq |p'|-1} 
P_{\xi}^{ p' } (Z_{k},x')  D_{\xi}u_{k} \big(Z_{k}  , x' \big) \\
& \quad + \sum_{ k \in K_{-}, \,  Q_{5}^{k} \not = \emptyset  }  \sum_{0' \leq q' \leq p' } 
\Big( \begin{array}{c} p' \\ q' \end{array} \Big) 
\pi^{q'} \big( \varphi_{k}(x') , x' \big)
D^{ ( |q'|, p'-q' ) }  u_{k} \big( \varphi_{k}(z') , z' \big)   \\
& \quad +  \sum_{ k \in K_{-}, \,  Q_{5}^{k} \not = \emptyset  } \sum_{|\xi| \leq |p'|-1}  P_{\xi}^{p'} \big( \varphi_{k}(x') , x' \big)   D_{\xi}u_{k} \big( \varphi_{k}(x') , x' \big).
\end{aligned}\end{equation}
It follows from \eqref{CTF640} that
\begin{equation}\label{CTF220}
\tilde{\eta} = I + II + III
\quad \text{in} \quad
Q_{5},
\end{equation}
where
\begin{equation*}\label{}
I = \sum_{0' \leq q' \leq p' } 
\Big( \begin{array}{c} p' \\ q' \end{array} \Big)
\left[ \pi^{q'} D^{(|q'|, p'-q')} u -  \sum_{ k \in K, \,  Q_{5}^{k} \not = \emptyset } 
\pi^{q'} \big( Z_{k},x' \big)
D^{ ( |q'|, p'-q' ) }  u_{k}(z_{k}) \chi_{Q_{5}^{k}} \right],
\end{equation*}
\begin{equation*}\begin{aligned}\label{}
II  & =  - \sum_{k \in K, \,  Q_{5}^{k} \not = \emptyset  }  \sum_{0' \leq q' \leq p' } 
\Big( \begin{array}{c} p' \\ q' \end{array} \Big) 
\pi^{q'} \big( Z_{k},x' \big) 
D^{ ( |q'|, p'-q' ) }  u_{k}(z_{k}) \chi_{Q_{5} \setminus Q_{5}^{k}} \\
& \quad + \sum_{ k \in K_{-}, \,  Q_{5}^{k} \not = \emptyset  }  \sum_{0' \leq q' \leq p' } 
\Big( \begin{array}{c} p' \\ q' \end{array} \Big) 
\pi^{q'} \big( \varphi_{k}(x') , x' \big)
D^{ ( |q'|, p'-q' ) }  u_{k} \big( \varphi_{k}(z') , z' \big)
\end{aligned}\end{equation*}
and
\begin{equation*}\begin{aligned}\label{}
III & =  - \sum_{k \in K, \,  Q_{5}^{k} \not = \emptyset } \sum_{|\xi| \leq |p'|-1} P_{\xi}^{ p'  } \big( Z_{k},x' \big)  D_{\xi}u_{k} \big( Z_{k}, x' \big) \chi_{Q_{5} \setminus Q_{5}^{k}} \\
& \quad +  \sum_{ k \in K_{-}, \,  Q_{5}^{k} \not = \emptyset  } \sum_{|\xi| \leq |p'|-1}  P_{\xi}^{p'} \big( \varphi_{k}(x') , x' \big)   D_{\xi}u_{k} \big( \varphi_{k}(x') , x' \big).
\end{aligned}\end{equation*}
Due to the approximation argument, we have that $\delta>0$ in \eqref{GS130} and \eqref{PET170} hold. So we find that $\eta$ in \eqref{CTF650} is weakly differentiable in $Q_{5}$ and that $D\eta \in L^{\infty}(Q_{5})$.

\begin{lem}\label{CTFS250}
$\eta$ in \eqref{CTF650} is weakly differentiable in $Q_{5}$ and that $D\eta \in L^{\infty}(Q_{5})$.
\end{lem}

\begin{proof}
We only need to check that $\partial_{t}^{p'}u \big|_{t=0}$ is weakly differentiable in $Q_{5}$ and the derivative is bounded in $Q_{5}$. Fix $\alpha \in \{ 1, \cdots, n \}$.

In view of \eqref{PET170}, we have that $\partial_{t}^{p'}u_{k} \big|_{t=0} \in C^{1} \left( \overline{ Q_{5}^{k} }, \br^{N} \right)$. Thus for any $\varphi \in C_{c}^{\infty}(Q_{5},\br^{N})$,
\begin{equation*}\begin{aligned}\label{}
\int_{Q_{5}} \partial_{t}^{p'}u \big|_{t=0} D_{i}\varphi \, dx  
& = \sum_{k \in K} \int_{Q_{5}^{k}} \partial_{t}^{p'}u \big|_{t=0} D_{i}\varphi \, dx \\
& = \sum_{k \in K} \int_{Q_{5}^{k}} \partial_{t}^{p'}u_{k} \big|_{t=0} \, dx  \\
& = - \sum_{k \in K} \int_{Q_{5}^{k}} D_{i} \left( \partial_{t}^{p'}u_{k} \right) \big|_{t=0} \varphi \, dx 
+ \sum_{k \in K} \int_{\partial Q_{5}^{k}} \partial_{t}^{p'}u_{k} \big|_{t=0} \varphi \, n_{k}^{\alpha} \,  dS.
\end{aligned}\end{equation*}
where $\vec{n}_{k} = (n_{k}^{1}, n_{k}^{2}, \cdots, n_{k}^{n})$ is the outward normal vector on $\partial Q_{5}^{k}$.  

From that $u_{k}=u_{k-1}$ on $\partial Q_{5}^{k} \cap \partial Q_{5}^{k-1}$, we obtain that $\partial_{t}^{p'}u_{k} \big|_{t=0} =  \partial_{t}^{p'}u_{k-1} \big|_{t=0} $ on $\partial Q_{5}^{k} \cap \partial Q_{5}^{k-1}$. So from that  $n_{k}^{\alpha} = - n_{k-1}^{i} $ $(k \in K)$, we obtain that
\begin{equation*}\label{}
\sum_{k \in K} \int_{\partial Q_{5}^{k}} \partial_{t}^{p'}u_{k} \big|_{t=0} \varphi \, n_{k}^{\alpha} \,  dS = 0.
\end{equation*}
Thus
\begin{equation*}\begin{aligned}\label{}
\int_{Q_{5}} \partial_{t}^{p'}u \big|_{t=0} D_{i}\varphi \, dx  
& = - \sum_{k \in K} \int_{Q_{5}^{k}} D_{\alpha}\left( \partial_{t}^{p'}u_{k} \right) \big|_{t=0} \varphi \, dx \\
& = - \int_{Q_{5}} \left[ \sum_{k \in K} D_{\alpha}\left( \partial_{t}^{p'}u_{k} \right) \big|_{t=0} \chi_{Q_{5}^{k}} \right] \varphi \, dx.
\end{aligned}\end{equation*}
With  \eqref{PET170}, we discover that $\partial_{t}^{p'}u \big|_{t=0} $ is weakly differentiable in $Q_{5}$ and that $D_{\alpha}\left( \partial_{t}^{p'}u \big|_{t=0} \right) = \sum_{k \in K} D_{\alpha}\left( \partial_{t}^{p'}u_{k} \right) \big|_{t=0} \chi_{Q_{5}^{k}} \in L^{\infty}(Q_{5})$. Since $i \in \{ 1, \cdots, n\}$ was arbitrary chosen, the lemma holds.
\end{proof}

Recall from \eqref{GS910} that
\begin{equation}\label{CTF290} 
x \in Q_{5}^{l} 
\qquad \Longrightarrow \qquad
H_{k}(x)  = \left\{ \begin{array}{cl}
1 & ~ \text{ if } ~ l \geq k , \\
0 & ~ \text{ if } ~ l < k.
\end{array}\right.
\end{equation}
To handle $I$, $II$ and $III$, we obtain three following lemmas related to \eqref{CTF290}.

\begin{lem}\label{CTFS300}
For any $h : Q_{5} \to \br$ and $k \in K$, we have that
\begin{equation*}
 h \big(Z_{k} , x' \big) \chi_{Q_{5} \setminus Q_{5}^{k}}
= (1-H_{k}) \, h \big( \varphi_{k}(x') , x' \big) + H_{k+1} \, h \big( \varphi_{k+1}(x') , x' \big) .
\end{equation*}
\end{lem}

\begin{proof}
Fix $k \in K$. From \eqref{CTF290}, we find that
\begin{equation*}
\sum_{l <k, \, l \in K} h \big(Z_{k}  , x' \big) \chi_{Q_{5}^{l}}
=  \sum_{ l<k, \, l \in K } h \big( \varphi_{k}(x') , x' \big) \chi_{Q_{5}^{l}}
=  (1-H_{k}) \, h  \big( \varphi_{k}(x') , x' \big) 
~ \text{ in } ~ Q_{5}.
\end{equation*}
Also from \eqref{CTF290}, we find that
\begin{equation*}
\sum_{ l>k , \, l \in K} h \big(Z_{k}, x' \big) \chi_{Q_{5}^{l}}
= \sum_{ l>k, \, l \in K} h \big( \varphi_{k+1}(x')  , x' \big) \chi_{Q_{5}^{l}}
= H_{k+1} \, h \big( \varphi_{k+1}(x') , x' \big)
~ \text{ in } ~ Q_{5}.
\end{equation*}
Since $k \in K$ was arbitrary chosen, the lemma follows from the above equalities.
\end{proof}

With Lemma \ref{GSS950}, we only need to handle the indices in $K_{-}$ in  \eqref{GS180} not all the indices in $K$ . Here, $ K_{-}$ represents the index of graph functions  intersecting $Q_{5}$.

\begin{lem}\label{GSS950}
For any $k \in K$ with $Q_{5}^{k} \not = \emptyset$, we have that
\begin{equation}\label{GS960}
k+1 \not \in K_{-} 
\qquad \Lra \qquad
H_{k+1}=0 ~ \text{ in } ~ Q_{5},
\end{equation}
and
\begin{equation}\label{GS970}
k \not \in K_{-} 
\qquad \Lra \qquad 
H_{k}=1 ~ \text{ in } ~ Q_{5}.
\end{equation}
\end{lem}

\begin{proof}
Let $K_{-} = \left\{ l_{-} , l_{-}+1, \cdots, l_{+} \right\} $. From \eqref{GS180} and that $Q_{5}^{k} \not = \emptyset$, we find that
\begin{equation}\label{GS980}
k \in  \left\{ l_{-}-1, l_{-} , l_{-}+1, \cdots, l_{+} \right\}.
\end{equation}

To prove \eqref{GS960}, assume that $k+1 \not \in K_{-} $. Then from \eqref{GS980} and that $k+1 \not \in K_{-} $, we obtain that $k+1 = l_{+}+1$. Since $K_{-} = \left\{ l_{-} , l_{-}+1, \cdots, l_{+} \right\} $, we have from Definition \ref{composite cube} that $H_{k+1} = 0$ in $Q_{5}$.

To prove \eqref{GS970}, assume that $k \not \in K_{-} $. Then by \eqref{GS980} and that  $k \not \in K_{-} $, we find that $k = l_{-}-1$. Since $K_{-} = \left\{ l_{-} , l_{-}+1, \cdots, l_{+} \right\} $, we have that $H_{k} = 1$ in $Q_{5}$.
\end{proof}

We use perturbations with respect to the points $z_{k}$ $(k \in K)$. In view of Lemma \ref{CTFS300} and Lemma \ref{CTFS350}, $z_{k}$ can be handled by the graph functions $\varphi_{k}$ for any $k \in K$.

\begin{lem}\label{CTFS350}
For any $k \in K$, we have that
\begin{equation}\label{CTF360}
H_{k}(x) \not = 1
\text{ for some } x \in Q_{r}(z)
\qquad \Lra \qquad
\left| z_{k} - \left(   \varphi_{k}(z'), z' \right)  \right| <2r,
\end{equation}
and
\begin{equation}\label{CTF370}
H_{k+1}(x) \not = 0
\text{ for some } x \in Q_{r}(z)
\qquad \Lra \qquad
\left| z_{k} -  \left(  \varphi_{k+1}(z') , z' \right)  \right| < 2r.
\end{equation}
\end{lem}

\begin{proof}
We first prove \eqref{CTF360}. Suppose that $H_{k}(x) \not = 1$  for some $x \in Q_{r}(z)$. Then we have from \eqref{composite cube_strip} and \eqref{CTF290} that $x^{1} \leq \varphi_{k}(x')$. 

(1) Assume that $\varphi_{k}(z') \geq z^{1}$. Then from \eqref{CTF120} and \eqref{CTF140}, we have that $z_{k} = \left( \varphi_{k}(z'), z' \right)$, and so \eqref{CTF360} holds. 

(2) Assume that $\varphi_{k}(z') < z^{1} \leq \varphi_{k+1}(z') $. Then we have from \eqref{CTF120} and \eqref{CTF140} that $z_{k}=z$. Since $x^{1} \leq \varphi_{k}(x')$, choose a point $y = (y^{1},y') \in Q_{r}(z)$ on the line connecting $z$ and $x$ such that $y^{1}= \varphi_{k}(y')$.  Then by \eqref{gradient_graphs} and that $z_{k}=z$,
\begin{equation*}\label{}
\left | \varphi_{k}(z') - z_{k}^{1} \right | \leq \left |\varphi_{k}(z') - \varphi_{k}(y') \right | + \left| \varphi_{k}(y') - z^{1} \right| \leq \frac{ \left| z' - y' \right| }{10n}  + \left| y^{1} - z^{1} \right|.
\end{equation*}
Since $y = (y^{1},y') \in Q_{r}(z)$ is on the line connecting $z$ and $x$, we have that $|z'-y'| \leq |z'-x'| < 2nr$ and $|z^{1}-y^{1}| \leq |z^{1}-x^{1}| < r$, and so \eqref{CTF360} holds.

(3) Assume that $z^{1} > \varphi_{k+1}(z') \left( > \varphi_{k}(z') \right)$. Then we have that $z_{k} = \left( \varphi_{k+1}(z'), z' \right)$. Since $x^{1} \leq \varphi_{k}(x') \left( < \varphi_{k+1}(x') \right)$, choose $y, t \in Q_{r}(z)$ on the line connecting $z$ and $x$ such that $y^{1}= \varphi_{k}(y')$ and $t^{1} = \varphi_{k+1}(t')$. Then by \eqref{gradient_graphs} and that $z_{k} = \left( \varphi_{k+1}(z'), z' \right)$, 
\begin{equation*}\begin{aligned}\label{}
\left| \varphi_{k}(z') - z_{k}^{1} \right| 
& \leq \left| \varphi_{k}(z') - \varphi_{k}(y') \right| + \left| \varphi_{k}(y') - \varphi_{k+1}(t') \right| + \left| \varphi_{k+1}(t') - \varphi_{k+1}(z') \right| \\
& \leq \frac{ \left| z' - y' \right| }{10n}  + |y^{1}-t^{1}| + \frac{ \left| t' - z' \right| }{10n}.
\end{aligned}\end{equation*}
Since $y,t \in Q_{r}(z)$ are on the line connecting $z$ and $x$, we have that $|z'-y'|, |t'-z'| \leq |z'-x'| < 2nr$ and $|y^{1}-t^{1}| \leq |z^{1}-x^{1}| < r$, and so \eqref{CTF360} holds.

\sskip

We now prove \eqref{CTF370}. Suppose that $H_{k+1}(x) \not = 0$  for some $x \in Q_{r}(z)$. Then we have from \eqref{composite cube_strip}  and \eqref{CTF290} that $x^{1} > \varphi_{k+1}(x')$. 

(1) Assume that $\varphi_{k+1}(z') < z^{1}$. Then from \eqref{CTF120} and \eqref{CTF140}, we have that $z_{k} = \left( \varphi_{k+1}(z'), z' \right)$ and so \eqref{CTF370} holds. 

(2) Assume that $\varphi_{k}(z') < z^{1} \leq \varphi_{k+1}(z') $. Then we have from \eqref{CTF120} and \eqref{CTF140} that $z_{k}=z$. Since $x^{1} > \varphi_{k+1}(x')$, choose a point $y = (y^{1},y')\in Q_{r}(z)$ on the line connecting $z$ and $x$ such that $y^{1}= \varphi_{k+1}(y')$.  Then by \eqref{gradient_graphs} and that $z_{k}=z$,
\begin{equation*}\label{}
\left | \varphi_{k+1}(z') - z_{k}^{1} \right | \leq \left |\varphi_{k+1}(z') - \varphi_{k+1}(y') \right | + \left| \varphi_{k+1}(y') - z^{1} \right| \leq \frac{ \left| z' - y' \right| }{10n}  + \left| y^{1} - z^{1} \right|.
\end{equation*}
Since $y = (y^{1},y') \in Q_{r}(z)$ is on the line connecting $z$ and $x$,  we have that $|z'-y'| \leq |z'-x'| < 2nr$ and $|z^{1}-y^{1}| \leq |z^{1}-x^{1}| < r$, and so \eqref{CTF370} holds.

(3) Assume that $z^{1} \leq \varphi_{k}(z') \left( < \varphi_{k+1}(z') \right)$. Then we have that $z_{k} = \left( \varphi_{k}(z'), z' \right)$. Since $x^{1} > \varphi_{k+1}(x') \left( > \varphi_{k}(x') \right)$, choose $y, t \in Q_{r}(z)$ on the line connecting $z$ and $x$ such that $y^{1}= \varphi_{k+1}(y')$ and $t^{1} = \varphi_{k}(t')$. Then by \eqref{gradient_graphs} and that  $z_{k} = \left( \varphi_{k}(z'), z' \right)$,
\begin{equation*}\begin{aligned}\label{}
\left| \varphi_{k+1}(z') - z_{k}^{1} \right| 
& \leq \left| \varphi_{k+1}(z') - \varphi_{k+1}(y') \right| + \left| \varphi_{k+1}(y') - \varphi_{k}(t') \right| + \left| \varphi_{k}(t') - \varphi_{k}(z') \right| \\
& \leq \frac{ \left| z' - y' \right| }{10n}  + |y^{1}-t^{1}| + \frac{ \left| t' - z' \right| }{10n}.
\end{aligned}\end{equation*}
Since $y,t \in Q_{r}(z)$ are on the line connecting $z$ and $x$,  we have that $|z'-y'|, |t'-z'| \leq |z'-x'| < 2nr$ and $|y^{1}-t^{1}| \leq |z^{1}-x^{1}| < r$, and so \eqref{CTF370} holds.
\end{proof}

Set $S = \cup_{k \in  K_{+} } \big\{ \big( \varphi_{k}(x'),x' \big) \in x' \in Q_{5}' \big\}$. Then $|S| = 0$. By using Corollary \ref{CTFS150}, Lemma \ref{CTFS300} and Lemma \ref{GSS950} and Lemma \ref{CTFS350}, we estimate $DI$, $D(II)$ and $D(III)$ in the set $Q_{r}(z) \setminus S$ as in the following three lemmas.

\begin{lem}\label{CTFS400}
We have the following estimate for $I$ in \eqref{CTF220} :
\begin{equation*}\label{}
\int_{Q_{r}(z) \setminus S} \bigg| DI - \sum_{0' \leq q' \leq p' } \Big( \begin{array}{c} p' \\ q' \end{array} \Big) \pi^{q'} D^{(|q'|, p'-q')} Du  \bigg| ^{2} \, dx
\leq cr^{n+1} \| D^{m+1}u \|_{L^{\infty}( Q_{2r}(z) )}^{2}.
\end{equation*}
\end{lem}

\begin{proof}
From \eqref{composite cube_strip} and \eqref{CTF110}, we have that $ \big( Z_{k}(x),x' \big) = x$ for any $x \in Q_{5}^{k} $ $(k \in K)$. It follows that $\pi^{q'} \big( Z_{k}(x),x' \big) = \pi^{q'}(x)$ for any  $x \in  Q_{5}^{k} $ $(k \in K)$. So by \eqref{CTF220},
\begin{equation*}\label{CTF480}
I = \sum_{0' \leq q' \leq p' } 
\Big( \begin{array}{c} p' \\ q' \end{array} \Big)
\bigg[ \pi^{q'} D^{(|q'|, p'-q')} u -  
\pi^{q'} 
D^{ ( |q'|, p'-q' ) }  u_{k}(z_{k})  \bigg]
~ \text{ in } ~ Q_{r}^{k}(z)
\quad (k \in K).
\end{equation*}
It follows that
\begin{equation*}\begin{aligned}\label{}
&  \sum_{k \in K } \int_{Q_{r}^{k}(z) \setminus S}  \left| DI - \sum_{0' \leq q' \leq p' } \Big( \begin{array}{c} p' \\ q' \end{array} \Big)
\pi^{q'} D^{(|q'|, p'-q')} Du  \right|^{2} \, dx \\
& \quad =  \sum_{k \in K } \int_{Q_{r}^{k}(z) \setminus S}   \left| \sum_{0' \leq q' \leq p' } D \left[ \pi^{q'} \right]  \left[ D^{ ( |q'|, p'-q' ) }  u - D^{ ( |q'|, p'-q' ) }  u_{k}(z_{k}) \right] \right|^{2} \, dx .
\end{aligned}\end{equation*}
Since $Q_{2r}(z) \subset Q_{5}$ and $D^{m}u_{k} \in C^{\gamma} \left( \overline{ Q_{2r}^{k}(z) } \right)$, we have from Lemma \ref{CTFS180} that
\begin{equation*}\label{} 
|D^{m}u - D^{m}u_{k}(z_{k})| \leq cr \| D^{m+1}u \|_{L^{\infty}(Q_{2r}^{k}(z))}
~ \text{ in } ~ Q_{r}^{k}(z)
\qquad ( k \in K),
\end{equation*}
by separating the cases that $Q_{r}^{k}(z) = \emptyset$ and $Q_{r}^{k}(z) \not = \emptyset$. It follows that
\begin{equation*}\begin{aligned}\label{}
& \int_{Q_{r}(z) \setminus S}  \left| DI - \sum_{0' \leq q' \leq p' } \Big( \begin{array}{c} p' \\ q' \end{array} \Big)
\pi^{q'} D^{(|q'|, p'-q')} Du  \right|^{2} \, dx \\
& \quad \leq c r^{2}  \| D^{m+1}u \|_{L^{\infty}( Q_{2r}(z) )}^{2}  \int_{Q_{r}(z) }  \left| D \left[ \pi^{q'} \right] \right|^{2} \, dx,
\end{aligned}\end{equation*}
and so the lemma follows from \eqref{GS890}.
\end{proof}

\begin{lem}\label{CTFS500}
For any $\alpha \in \{ 1, \cdots, n \}$, the next estimate holds for $II$ in \eqref{CTF220}:
\begin{equation*}\label{}
\left| D_{\alpha}(II) +  \sum_{ k \in K_{-}  }  \bh_{\alpha}^{k} \left[ u_{k-1}- u_{k} \right]   H_{k}  \right| 
\leq c r^{\gamma} \sum_{ k \in K_{-}   } \| u_{k} \|_{C^{m,\gamma} \left( \overline{Q_{6}^{k}} \right)},
\end{equation*}
in $Q_{r}(z) \setminus S$ where the linear operator $\bh_{\alpha}^{k} : C^{m,\gamma} \left( \partial Q_{6}^{k} \cap \partial Q_{6}^{k-1} \right) \to C^{\gamma} \left( Q_{5}' \right) $ $(k \in K_{-})$ defined as
\begin{equation*}\label{}
\bh_{\alpha}^{k}[h](x') =  \sum_{0' \leq q' \leq p' } 
\Big( \begin{array}{c} p' \\ q' \end{array} \Big) 
D_{j} \left[ \pi^{q'} \big( \varphi_{k}(x') , x' \big) \right]
D^{ ( |q'|, p'-q' ) } h \big( \varphi_{k}(x') , x' \big).
\end{equation*}
\end{lem}

\begin{proof}
In Lemma \ref{CTFS300}, we take $h =  \sum_{0' \leq q' \leq p' } 
\Big( \begin{array}{c} p' \\ q' \end{array} \Big)  \pi^{q'} D^{ ( |q'|, p'-q' ) }  u_{k}(z_{k})$. Then 
\begin{equation*}\begin{aligned}\label{}
II & = - \sum_{ k \in K, \,  Q_{5}^{k} \not = \emptyset  }  \sum_{0' \leq q' \leq p' } 
\Big( \begin{array}{c} p' \\ q' \end{array} \Big) 
(1-H_{k})  \pi^{q'} \big( \varphi_{k}(x') , x' \big) 
D^{ ( |q'|, p'-q' ) }  u_{k}(z_{k}) \\
& \quad - \sum_{ k \in K, \,  Q_{5}^{k} \not = \emptyset   }  \sum_{0' \leq q' \leq p' } 
\Big( \begin{array}{c} p' \\ q' \end{array} \Big) 
H_{k+1}  \pi^{q'} \big( \varphi_{k+1}(x') , x' \big) 
D^{ ( |q'|, p'-q' ) }  u_{k}(z_{k}) \\
& \quad + \sum_{ k \in K_{-}, \,  Q_{5}^{k} \not = \emptyset  }  \sum_{0' \leq q' \leq p' } 
\Big( \begin{array}{c} p' \\ q' \end{array} \Big) 
\pi^{q'} \big( \varphi_{k}(x') , x' \big)
D^{ ( |q'|, p'-q' ) }  u_{k} \big( \varphi_{k}(z') , z' \big),
\end{aligned}\end{equation*}
in $Q_{r}(z)$. Fix $\alpha  \in \{ 1,  \cdots, n \}$.  We find from Lemma \ref{GSS950} that
\begin{equation*}\label{}
D_{\alpha}(II) = II_{1} + II_{2} + II_{3} \quad \text{in} \quad Q_{r}(z) \setminus S,
\end{equation*}
where
\begin{equation*}\begin{aligned}\label{}
II_{1} & = - \sum_{ k \in K_{-}, \,  Q_{5}^{k} \not = \emptyset  }  \sum_{0' \leq q' \leq p' } 
\Big( \begin{array}{c} p' \\ q' \end{array} \Big) 
(1-H_{k}) D_{\alpha} \left[ \pi^{q'} \big( \varphi_{k}(x') , x' \big) \right]
D^{ ( |q'|, p'-q' ) }  u_{k}(z_{k}) \\
& \quad + \sum_{ k \in K_{-}, \,  Q_{5}^{k} \not = \emptyset  }  \sum_{0' \leq q' \leq p' } 
\Big( \begin{array}{c} p' \\ q' \end{array} \Big) 
(1-H_{k})  D_{\alpha} \left[\pi^{q'} \big( \varphi_{k}(x') , x' \big) \right]
D^{ ( |q'|, p'-q' ) }  u_{k}  \big( \varphi_{k}(z') , z' \big),
\end{aligned}\end{equation*}
\begin{equation*}\begin{aligned}\label{}
II_{2} & =  \sum_{ k \in K_{-}, \,  Q_{5}^{k} \not = \emptyset  }  \sum_{0' \leq q' \leq p' } 
\Big( \begin{array}{c} p' \\ q' \end{array} \Big) 
H_{k} D_{\alpha} \left[ \pi^{q'} \big( \varphi_{k}(x') , x' \big) \right]
D^{ ( |q'|, p'-q' ) }  u_{k} \big( \varphi_{k}(z') , z' \big),
\end{aligned}\end{equation*}
and
\begin{equation*}\begin{aligned}\label{}
II_{3} & =  - \sum_{ k+1 \in K_{-}, \,  Q_{5}^{k} \not = \emptyset   }  \sum_{0' \leq q' \leq p' } 
\Big( \begin{array}{c} p' \\ q' \end{array} \Big) 
H_{k+1} D_{\alpha} \left[ \pi^{q'} \big( \varphi_{k+1}(x') , x' \big) \right]
D^{ ( |q'|, p'-q' ) }  u_{k}(z_{k}) \\
& =  - \sum_{ k \in K_{-}, \,  Q_{5}^{k-1} \not = \emptyset   }  \sum_{0' \leq q' \leq p' } 
\Big( \begin{array}{c} p' \\ q' \end{array} \Big) 
H_{k} D_{\alpha} \left[ \pi^{q'} \big( \varphi_{k}(x') , x' \big) \right]
D^{ ( |q'|, p'-q' ) }  u_{k-1}(z_{k-1}),
\end{aligned}\end{equation*}
in $Q_{r}(z) \setminus S$. In view of \eqref{GS260}, we discover that
\begin{equation}\label{CTF560}
\left| D \left[ \pi^{q'} \big( \varphi_{k}(x') , x' \big)  \right]  \right| \leq c 
\quad \text{in} \quad
Q_{5}'.
\end{equation}
By Corollary \ref{CTFS150}, if $Q_{5}^{k} \not = \emptyset$ for some $k \in K$ then $z_{k}  \in \overline{ Q_{6}^{k} }$. Also we have from \eqref{GS190}  that $\big( \varphi_{k}(z') , z' \big) \in \overline{ Q_{6}^{k} }$ for any $k \in K_{-}$. So by  Lemma \ref{CTFS350} and \eqref{CTF560},
\begin{equation*}\begin{aligned}\label{}
|II_{1}| 
& \leq c \sum_{ k \in K_{-}, \,  Q_{5}^{k} \not = \emptyset  }  \sum_{0' \leq q' \leq p' } 
(1-H_{k})  \left| D^{ ( |q'|, p'-q' ) }  u_{k}(z_{k}) - 
D^{ ( |q'|, p'-q' ) }  u_{k} \big( \varphi_{k}(z') , z' \big)  \right|  \\
& \leq c r^{\gamma} \sum_{ k \in K_{-}, \,  Q_{5}^{k} \not = \emptyset  } \| u_{k} \|_{C^{m,\gamma} \left( \overline{Q_{6}^{k}} \right)}
\end{aligned}\end{equation*}
in $Q_{r}(z) \setminus S$. 

From \eqref{GS190}, we have that $\big( \varphi_{k}(x') , x' \big) \in \overline{ Q_{6}^{k} }$ for any $x' \in Q_{r}'(z')$ and  $k \in K_{-}$. So it follows from \eqref{CTF560} that
\begin{equation*}\begin{aligned}\label{}
& \left| II_{2}  - \sum_{ k \in K_{-}, \,  Q_{5}^{k} \not = \emptyset  }  \sum_{0' \leq q' \leq p' }  
\Big( \begin{array}{c} p' \\ q' \end{array} \Big) 
H_{k} D_{\alpha} \left[ \pi^{q'} \big( \varphi_{k}(x') , x' \big) \right]
D^{ ( |q'|, p'-q' ) }  u_{k} \big( \varphi_{k}(x') , x' \big)   \right| \\
& \quad \leq c r^{\gamma} \sum_{ k \in K_{-}, \,  Q_{5}^{k} \not = \emptyset  } \| u_{k} \|_{C^{m,\gamma} \left( \overline{Q_{6}^{k}} \right)},
\end{aligned}\end{equation*}
in $Q_{r}(z) \setminus S$. 

By Corollary \ref{CTFS150}, if $Q_{5}^{k-1} \not = \emptyset$ for some $k-1\in K$ then  $z_{k-1}  \in \overline{ Q_{6}^{k-1} }$. Also we have from \eqref{GS190}  that $\big( \varphi_{k}(x') , x' \big) \in \overline{ Q_{6}^{k-1} }$ for any $x' \in Q_{r}'(z')$ and $k \in K_{-}$. So by Lemma \ref{CTFS350} and \eqref{CTF560} that
\begin{equation*}\begin{aligned}\label{}
& \left| II_{3}  + \sum_{ k \in K_{-}, \,  Q_{5}^{k-1} \not = \emptyset   }  \sum_{0' \leq q' \leq p' } 
\Big( \begin{array}{c} p' \\ q' \end{array} \Big)  H_{k}   
D_{\alpha}\left[ \pi^{q'} \big( \varphi_{k}(x') , x' \big)  \right]
D^{ ( |q'|, p'-q' ) }  u_{k-1} \big( \varphi_{k}(x') , x' \big) \right| \\
& \quad \leq c r^{\gamma} \sum_{ k \in K_{-}, \,  Q_{5}^{k-1} \not = \emptyset  } \| u_{k} \|_{C^{m,\gamma} \left( \overline{Q_{6}^{k}} \right)},
\end{aligned}\end{equation*}
in $Q_{r}(z) \setminus S$. 

Since $\alpha \in \{1,  \cdots, n \}$ was arbitrary, the lemma holds from  \eqref{GS185}.
\end{proof}

\begin{lem}\label{CTFS600}
For any $\alpha \in \{ 1,\cdots, n \}$, the following  holds for $III$ in \eqref{CTF220}:
\begin{equation*}\label{}
D_{\alpha}(III)  = -  \sum_{ k \in K_{-}  }  \bh_{\alpha}^{k} \left[ u_{k-1}  - u_{k} \right] H_{k}  
\end{equation*}
in $Q_{r}(z) \setminus S$ with the linear operator $\bh_{\alpha}^{k} : C^{m,\gamma} \left( \partial Q_{6}^{k} \cap \partial Q_{6}^{k-1} \right) \to C^{\gamma} \left( Q_{5}' \right) $  $(k \in K_{-})$ defined as
\begin{equation*}\label{}
\bh_{\alpha}^{k}[h](x) =  \sum_{|\xi| \leq |p'|-1}  D_{\alpha} \left[ P_{\xi}^{p'} \big[ \varphi_{k}(x'), x' \big] \, D_{\xi}h \big( \varphi_{k}(x') , x' \big) \right] \, H_{k} .
\end{equation*}
\end{lem}

\begin{proof}
In Lemma \ref{CTFS300}, we take $h = \sum_{|\xi| \leq |p'|-1} P_{\xi}^{ p'  } \big( Z_{k},x' \big)  D_{\xi}u_{k} \big( Z_{k}, x' \big)   $. Then 
\begin{equation*}\begin{aligned}\label{} 
III & = - \sum_{ k \in K, \,  Q_{5}^{k} \not = \emptyset  } \sum_{|\xi| \leq |p'|-1} (1-H_{k}) P_{\xi}^{p'} \big[ \varphi_{k}(x'),x' \big] \, D_{\xi}u_{k} \big( \varphi_{k}(x') , x' \big) \\
& \quad - \sum_{ k \in K, \,  Q_{5}^{k} \not = \emptyset  } \sum_{|\xi| \leq |p'|-1}  H_{k+1} P_{\xi}^{p'} \big[ \varphi_{k+1}(x'), x' \big] \, D_{\xi}u_{k} \big( \varphi_{k+1}(x') , x' \big) \\
& \quad +  \sum_{ k \in K_{-}, \,  Q_{5}^{k} \not = \emptyset  } \sum_{|\xi| \leq |p'|-1}  P_{\xi}^{p'} \big( \varphi_{k}(x') , x' \big)   D_{\xi}u_{k} \big( \varphi_{k}(x') , x' \big)
\end{aligned}\end{equation*}
in $Q_{r}(z) \setminus S$. Fix $ \alpha \in  \{ 1, \cdots, n \}$. Then from Lemma \ref{GSS950}, we find that
\begin{equation*}\label{}
D_{\alpha}(III) = III_{1} + III_{2} \quad \text{in} \quad Q_{5} \setminus S,
\end{equation*}
where
\begin{equation*}\begin{aligned}\label{}
III_{1} & = \sum_{ k \in K_{-}, \,  Q_{5}^{k} \not = \emptyset  } \sum_{|\xi| \leq |p'|-1}  H_{k} D_{\alpha} \left[ P_{\xi}^{p'} \big[ \varphi_{k}(x'),x' \big] \, D_{\xi}u_{k} \big( \varphi_{k}(x') , x' \big) \right] ,
\end{aligned}\end{equation*}
and
\begin{equation*}\begin{aligned}\label{}
III_{2} & = -  \sum_{ k+1 \in K_{-}, \,  Q_{5}^{k} \not = \emptyset  } \sum_{|\xi| \leq |p'|-1}  H_{k+1} D_{\alpha} \left[ P_{\xi}^{p'} \big[ \varphi_{k+1}(x'), x' \big] \, D_{\xi}u_{k} \big( \varphi_{k+1}(x') , x' \big) \right] \\
&  =   - \sum_{ k \in K_{-}, \,  Q_{5}^{k-1} \not = \emptyset  } \sum_{|\xi| \leq |p'|-1}  H_{k} D_{\alpha} \left[ P_{\xi}^{p'} \big[ \varphi_{k}(x'), x' \big] \, D_{\xi}u_{k-1} \big( \varphi_{k}(x') , x' \big) \right],
\end{aligned}\end{equation*}
in $x \in Q_{r}(z) \setminus S$.  Since $\alpha \in \{1, \cdots, n \}$ was arbitrary chosen, the lemma follows from \eqref{GS185}.
\end{proof}

In the following lemma, $G_{\alpha}$ $(\alpha \in \{1, \cdots, N \} )$ does not depend on $Q_{r}(z)$.

\begin{lem}\label{CTFS700}
For any $Q_{2r}(z) \subset Q_{5}$ and $\alpha  \in \{ 1, \cdots, n \}$, the following
\begin{equation*}\begin{aligned}\label{} 
& \int_{ Q_{r}(z) } \bigg| D_{\alpha}\eta - \sum_{0' \leq q' \leq p' } 
\Big( \begin{array}{c} p' \\ q' \end{array} \Big) \pi^{q'} D^{(|q'|, p'-q')} D_{\alpha}u  +  G_{j} \bigg|^{2} \, dx \\
& \quad \leq cr^{n+2\gamma} \left[  \| D^{m+1}u \|_{L^{\infty}( Q_{2r}(z) )}^{2} + \sum_{k \in K_{-}}  \| u \|_{C^{m,\gamma}( Q_{7}^{k} )} \right],
\end{aligned}\end{equation*}
holds for $\eta$ in \eqref{CTF650} where 
\begin{equation*}
G_{\alpha} = \sum_{k \in K_{-}} \bh_{\alpha}^{k} \left [ u_{k-1} - u_{k} \right ] H_{k} 
\end{equation*}
with the linear operator $\bh_{\alpha}^{k} : C^{m,\gamma} \left( \partial Q_{6}^{k} \cap \partial Q_{6}^{k-1} \right) \to C^{\gamma} \left( Q_{5}' \right) $  $(k \in K_{-})$ defined as
\begin{equation*}\begin{aligned}\label{}
\bh_{\alpha}^{k}[h](x) &=  \sum_{0' \leq q' \leq p' } 
\Big( \begin{array}{c} p' \\ q' \end{array} \Big) 
D_{\alpha} \left[ \pi^{q'} \big( \varphi_{k}(x') , x' \big) \right]
D^{ ( |q'|, p'-q' ) } h \big( \varphi_{k}(x') , x' \big) \\
& \quad+   \sum_{|\xi| \leq |p'|-1}  D_{\alpha} \left[ P_{\xi}^{p'} \big[ \varphi_{k}(x'), x' \big] \, D_{\xi}h \big( \varphi_{k}(x') , x' \big) \right].
\end{aligned}\end{equation*}
\end{lem}

\begin{proof}
We first estimate $\eta$. From \eqref{CTFS250}, $\eta$ is weakly differentiable in $Q_{r}(z)$ with the estimate $D\eta \in L^{\infty}(Q_{r}(z))$, which implies that
\begin{equation*}\label{} 
\int_{ S } \bigg| D_{\alpha}\eta - \sum_{0' \leq q' \leq p' } 
\Big( \begin{array}{c} p' \\ q' \end{array} \Big) \pi^{q'} D^{(|q'|, p'-q')} D_{\alpha}u  + \sum_{k \in K_{-}}   \bh_{\alpha}^{k} \left[ u_{k-1} - u_{k} \right] H_{k}   \bigg|^{2} dx =0.
\end{equation*}
With \eqref{CTF220}, we obtain from Lemma \ref{CTFS400}, Lemma \ref{CTFS500} and Lemma \ref{CTFS600} that
\begin{equation*}\begin{aligned}\label{} 
& \int_{ Q_{r}(z) \setminus S } \bigg| D_{\alpha}\eta - \sum_{0' \leq q' \leq p' } 
\Big( \begin{array}{c} p' \\ q' \end{array} \Big) \pi^{q'} D^{(|q'|, p'-q')} D_{\alpha}u + \sum_{k \in K_{-}}  \bh_{\alpha}^{k} \left[ u_{k-1} - u_{k} \right]  H_{k}   \bigg|^{2} \, dx \\
& \quad \leq cr^{n+2\gamma}  \left[ \| D^{m+1}u \|_{L^{\infty}( Q_{2r}(z) )}^{2} + \sum_{k \in K_{-}}  \| u \|_{C^{m,\gamma}( \overline{ Q_{6}^{k} } )} \right].
\end{aligned}\end{equation*}
So the lemma follows from \eqref{PET150}.
\end{proof}

To obtain the piece-wise regularity, we will use the following lemma.

\begin{lem}\label{CTFS800}
For $ \bh_{\alpha}^{k}$ and $G_{\alpha}$ in Lemma \ref{CTFS700},  we have that
\begin{equation*}\begin{aligned}\label{}
\sum_{k\in K_{-}} \left\|  \bh_{\alpha}^{k} \left[ u_{k-1} - u_{k} \right] \right \|_{C^{\gamma}(Q_{5}')} 
+ \sum_{k \in K} \| G_{\alpha} \|_{C^{\gamma}(Q_{5}^{k})} 
\leq c \sum_{k \in K}  \| u \|_{C^{m,\gamma}( Q_{7}^{k} )} ,
\end{aligned}\end{equation*}
for any $\alpha \in \{1, \cdots, n \}$.
\end{lem}

\begin{proof}
By that $\pi^{q'}  \big( \varphi_{k}(x'),x' \big)  
= [D_{2}\varphi_{k} (x') ]^{q_{2}} \cdots  [D_{n}\varphi_{k} (x')]^{q_{n}}$ $\big( x' \in Q_{5}', \, k \in K \big)$, the lemma holds from \eqref{GS180}, \eqref{PET140}, \eqref{PET150} and the definition of $P_{\xi}^{p'}$ in Lemma \ref{GSS275}.
\end{proof}

\newpage

\section{Higher order derivatives}

In the main equation, we assume that the coefficients are piece-wise continuous, and so the gradient of the weak solution might have big jumps on each $\partial Q_{5}^{k} \cap \partial Q_{5}^{k-1}$ $(k \in K)$. To show that $D^{m+1}u$ is piece-wise H\"{o}lder continuous, we first show that $U^{p'}$ $(|p'|=m)$  defined in \eqref{HOD020} are H\"{o}lder continuous. Then we will prove that $D^{m+1}u$ is piece-wise H\"{o}lder continuous by comparing $D^{m+1}u$ and $U^{p'}$. Here, $U^{p'}$ is defined by using the perturbation results in Lemma \ref{PETS900} and Lemma \ref{CTFS700}.

For the simplicity of the calculation, we abuse the notation $D_{\xi}u = D_{\xi}u_{k}$ in $Q_{7}^{k}$. Recall from \eqref{GS410} that
\begin{equation*}\label{} 
\pi_{1}=-1, 
\quad
\pi' = (\pi_{2}, \cdots, \pi_{n}),
\quad
\pi = (\pi_{1},\pi') = (-1,\pi_{2}, \cdots, \pi_{n})
\quad \text{ in } \quad Q_{7}.
\end{equation*}
With the assumption \eqref{PET170}, for any $p' \geq 0'$ with $|p'| = m$, set $G_{1}^{p'}$ as in Lemma \ref{PETS900} and $G_{2}^{p'},\cdots, G_{n}^{p'}$ as in Lemma \ref{CTFS700}. Then  define $U^{p'} : Q_{5} \to \br^{Nn}$ and $G^{p'} : Q_{5} \to \br^{Nn}$ as follows :
\begin{equation}\label{HOD020}
U^{p'} = \left( \begin{array}{ccc}
U_{1}^{1,p'} & \cdots & U_{n}^{1,p'} \\
\vdots & \ddots &  \vdots \\ 
U_{1}^{N,p'} & \cdots & U_{n}^{N,p'} 
\end{array}\right)
\qquad \text{and} \qquad
G ^{p'}=  \left( \begin{array}{ccc}
G_{1}^{1,p'} & \cdots & G_{n}^{1,p'} \\
\vdots  & \ddots  &\vdots  \\ 
G_{1}^{N,p'} & \cdots & G_{n}^{N,p'} 
\end{array}\right),
\end{equation}
where
\begin{equation*}\label{} 
U_{1}^{i,p'}
=G_{1}^{i,p'} + \sum_{1 \leq \alpha \leq n} \pi_{\alpha} \left[  \sum_{0' \leq q' \leq p' } 
\Big( \begin{array}{c} p' \\ q' \end{array} \Big)
\pi^{q'} D^{(|q'|,p'-q')}  \left[ \sum_{1 \leq j \leq N } \sum_{1 \leq \beta \leq n} A^{\alpha \beta}_{ij}D_{\beta}u^{j} - F_{\alpha}^{i} \right] \right]
\end{equation*}
in $Q_{5}$,
and 
\begin{equation}\label{HOD040} 
U_{\beta}^{i,p'} = G_{\beta}^{i,p'} +  \sum_{0' \leq q' \leq p' } 
\Big( \begin{array}{c} p' \\ q' \end{array} \Big)
\pi^{q'}
\left[ D^{ (|q'|,p'-q') }  D_{\beta}u^{i} + \pi_{\beta} D^{ (|q'|,p'-q') }  D_{1}u^{i} \right]
\end{equation}
in $Q_{5}$ for any $\beta \in \{ 2, \cdots, n \}$ and $i \in \{ 1, \cdots, N \}$. Recall from Lemma \ref{PETS800} and  \eqref{PET220} that
\begin{equation}\label{HOD050}
G_{1}^{i,p'} =  \sum_{ k \in  K_{-}  } \sum_{1 \leq \alpha \leq n} \ \bg_{\alpha}^{p',k} \left[ U_{\alpha}^{i,k-1} - U_{\alpha}^{i,k} \right]  H_{k}  
\quad \text{in} \quad Q_{5}
\end{equation}
and
\begin{equation*}\label{}
U_{\alpha}^{i,k} = \sum_{ 1 \leq j \leq N} \sum_{ 1 \leq \beta \leq n}  A_{\alpha \beta}^{ij,k} D_{\beta}u^{j,k} - F_{\alpha}^{i,k}
\quad \text{in} \quad Q_{5}
\end{equation*}
for any  $ 1 \leq \alpha \leq n, \ 1 \leq i \leq N$ and  $k \in K$ with the linear operator $\ \bg_{\alpha}^{p',k} : C^{m-1,\gamma} \left( \partial Q_{6}^{k} \cap \partial Q_{6}^{k-1} \right) \to C^{\gamma} \left( Q_{5}' \right) $  $(k \in K_{-})$  defined as
\begin{equation*}\begin{aligned}\label{}
\ \bg_{\alpha}^{p',k}[h](x') 
& = \sum_{|\xi| \leq |p'|-1} \left[ \pi_{\alpha}  P_{\xi}^{ p' }  D_{\xi}  h \right] \big( \varphi_{k}(x'),x' \big) \\
& \quad + \sum_{ 0' < q' \leq p' } \left( \begin{array}{c} p' \\ q' \end{array} \right) D^{q'} D_{i}\varphi_{k}(x') \, D^{p'-q'}  \left[ h\big(  \varphi_{k}(x'), x' \big) \right],
\end{aligned}\end{equation*}
for any $x' \in Q_{5}'$, $ 1 \leq \alpha \leq n$ and  $k \in K$.
Also recall from Lemma \ref{CTFS700} that 
\begin{equation}\label{HOD060}
G_{\beta}^{i,p'} = \sum_{k \in K_{-}} \bh_{\beta}^{p',k} \left [ u_{k-1} - u_{k} \right ] H_{k} 
\quad \text{in} \quad Q_{5}
\qquad ( \beta =2, 3, \cdots, n)
\end{equation}
with the linear operator $\bh_{\beta}^{p',k} : C^{m,\gamma} \left( \partial Q_{6}^{k} \cap \partial Q_{6}^{k-1} \right) \to C^{\gamma} \left( Q_{5}' \right) $  $(k \in K_{-})$ defined as
\begin{equation*}\begin{aligned}\label{}
\bh_{\beta}^{p',k}[h](x') &=  \sum_{0' \leq q' \leq p' } 
\Big( \begin{array}{c} p' \\ q' \end{array} \Big) 
D_{\beta} \left[ \pi^{q'} \big( \varphi_{k}(x') , x' \big) \right]
D^{ ( |q'|, p'-q' ) } h \big( \varphi_{k}(x') , x' \big) \\
& \quad+   \sum_{|\xi| \leq |p'|-1}  D_{\beta} \left[ P_{\xi}^{p'} \big[ \varphi_{k}(x'), x' \big] \, D_{\xi}h \big( \varphi_{k}(x') , x' \big) \right],
\end{aligned}\end{equation*}
for any $x' \in Q_{5}'$, $ 2 \leq \beta \leq n$ and  $k \in K$.
Then by Lemma \ref{PETS900} and Lemma \ref{CTFS800},
\begin{equation}\label{HOD065}
\left\| G^{p'} \right\|_{C^{\gamma}( Q_{5}^{k} )}  \leq c \sum_{k \in K} \Big[ \| u \|_{C^{m,\gamma}( Q_{7}^{k} )} 
+  \| F \|_{C^{m,\gamma}( Q_{7}^{k})} \Big] 
\quad \text{ in } \quad Q_{5}.
\end{equation}

To handle $U_{1}^{p'}$, we show Lemma \ref{HODS100}. In view of Lemma \ref{PETS900}, set $\nu^{p'} : Q_{5} \to \br^{Nn}$ as
\begin{equation}\label{HOD070}
\nu^{p'} = \left( \begin{array}{ccc}
\nu_{1}^{1,p'} & \cdots & \nu_{n}^{1,p'} \\
\vdots & \ddots &  \vdots \\ 
\nu_{1}^{N,p'} & \cdots & \nu_{n}^{N,p'} 
\end{array}\right)
\quad \text{ in } \quad Q_{5},
\end{equation}
where
\begin{equation*}\begin{aligned}\label{} 
\nu_{\alpha}^{i,p'} 
& =  \sum_{0' \leq q' \leq p' } 
\Big( \begin{array}{c} p' \\ q' \end{array} \Big)
\pi^{q'}    D^{(|q'|,p'-q')} \left( \sum_{1 \leq j \leq N } \sum_{1 \leq \beta \leq n} A^{\alpha \beta}_{ij}D_{\beta}u^{j} - F_{\alpha}^{i} \right) \\
&\quad  -  \sum_{0' \leq q' \leq p' } 
\Big( \begin{array}{c} p' \\ q' \end{array} \Big)
\pi^{q'}   \sum_{1 \leq j \leq N } \sum_{1 \leq \beta \leq n} A^{\alpha \beta}_{ij}D^{(|q'|,p'-q')} D_{\beta}u^{j},
\end{aligned}\end{equation*}
for any $1 \leq \alpha \leq n$ and $1 \leq i \leq N$. For the simplicity, set $\tilde{u}^{p'} : Q_{5} \to \br^{Nn}$ as
\begin{equation}\label{HOD075}
\tilde{u}_{\alpha}^{j,p'} 
= \sum_{0' \leq q' \leq p' } 
\Big( \begin{array}{c} p' \\ q' \end{array} \Big)
\pi^{q'} 
D^{ (|q'|,p'-q') } D_{\beta}u^{j}  
\quad \text{in} \quad Q_{5},
\end{equation}
for any $1 \leq \alpha \leq  n$ and $1 \leq j \leq N$.  By recalling \eqref{HOD070} and that $U_{1}^{i,p}$ in \eqref{HOD020},
\begin{equation}\label{HOD080} 
U_{1}^{i,p'}
= G_{1}^{i,p'} + \sum_{1 \leq \alpha \leq n } \pi_{\alpha} \left[ \nu_{\alpha}^{i,p'} + \sum_{1 \leq j \leq N } \sum_{1 \leq \beta \leq n} A^{\alpha \beta}_{ij}\tilde{u}_{\beta}^{j,p'}  \right]
\quad \text{in} \quad Q_{5},
\end{equation}
and
\begin{equation}\label{HOD085}
U_{\beta}^{j,p'} = G_{\beta}^{j,p'} +  \tilde{u}_{\beta}^{j,p'} + \pi_{\beta} \tilde{u}_{1}^{j,p'}
\quad \text{in} \quad Q_{5}, 
\end{equation}
for any $2 \leq \beta \leq  n$ and $1 \leq j \leq N$.  From  Lemma \ref{CTFS800}, we find that
\begin{equation}\begin{aligned}\label{HOD090}
\left\| \nu^{p'} \right\|_{C^{\gamma}(Q_{5}^{k})} 
\leq c \sum_{k \in K} \Big[ \| u \|_{C^{m,\gamma}( Q_{7}^{k} )} 
+  \| F \|_{C^{m,\gamma}( Q_{7}^{k})} \Big].
\end{aligned}\end{equation}
From \eqref{ell} and that $\pi_{1}=-1$, we have that
\begin{equation}\label{HOD140} 
\sum_{ 1 \leq i,j \leq N } \sum_{1 \leq \alpha,  \beta \leq n } A^{\alpha \beta}_{ij}\pi_{\alpha} \pi_{\beta} \zeta_{i} \zeta_{j} \geq \lambda \left| \left( \pi_{1}, \cdots, \pi_{n} \right) \right|^{2} |\zeta|^{2} \geq \lambda |\zeta|^{2} 
\quad \text{in} \quad Q_{5}.
\end{equation}

\begin{lem}\label{HODS100}
For $U_{\alpha}^{i,p'}$ and $G_{\alpha}^{i,p'}$ in \eqref{HOD020} $(1 \leq \alpha \leq n, \ 1 \leq i \leq N)$, we have that 
\begin{equation*}\begin{aligned}
& \sum_{1 \leq j \leq N } \sum_{1 \leq \alpha, \beta \leq n } A^{\alpha \beta}_{ij}\pi_{\alpha} \pi_{\beta}    \tilde{u}_{1}^{j,p'} \\
& \quad = G_{1}^{i,p'} - U_{1}^{i,p'} + \sum_{  1 \leq \alpha \leq n  } \pi_{\alpha} \left [ \nu_{\alpha}^{i,p'} +  \sum_{ 1 \leq j \leq N } \sum_{2 \leq \beta  \leq n }  A^{\alpha \beta}_{ij}\left( U_{\beta}^{i,p'} - G_{\beta}^{i,p'} \right) \right]
~ \text{ in } ~ Q_{5},
\end{aligned}\end{equation*}
for any $1 \leq i \leq N$.
\end{lem}

\begin{proof}
For any $2 \leq \beta \leq n$ and $1 \leq j \leq N$, \eqref{HOD085} yields that
\begin{equation*}\label{} 
\tilde{u}_{\beta}^{j,p'}
=  U_{\beta}^{j,p'} - G_{\beta}^{j,p'} - \pi_{\beta}  \tilde{u}_{1}^{j,p'}
\quad \text{in} \quad Q_{5}.
\end{equation*}
So by \eqref{HOD080} and that $\pi_{1} = -1$, 
\begin{equation*}\begin{aligned}\label{}
U_{1}^{i,p'}
& = G_{1}^{i,p'} +  \sum_{  1 \leq \alpha \leq n  } \pi_{\alpha} \sum_{ 1 \leq j \leq N } \sum_{2 \leq \beta  \leq n }  A^{\alpha \beta}_{ij}\left( U_{\beta}^{i,p'} - G_{\beta}^{i,p'} \right) \\
& \quad + \sum_{ 1 \leq \alpha \leq n } \pi_{\alpha} \left[ \nu_{\alpha}^{i,p'} - \sum_{1 \leq j \leq N } \sum_{1 \leq \beta \leq n} \pi_{\beta} A^{\alpha \beta}_{ij} \tilde{u}_{1}^{j,p} \right]
& \quad \text{in} \quad  Q_{5}
\end{aligned}\end{equation*}
for any $1 \leq i \leq N$, which implies that
\begin{equation*}\begin{aligned}
& \sum_{1 \leq j \leq N } \sum_{1 \leq \alpha,  \beta \leq n } A^{\alpha \beta}_{ij}\pi_{\alpha} \pi_{\beta}    \tilde{u}_{1}^{j,p'} \\
& \quad = G_{1}^{i,p'} - U_{1}^{i,p'} + \sum_{  1 \leq \alpha \leq n  } \pi_{\alpha} \left [ \nu_{\alpha}^{i,p'} +  \sum_{ 1 \leq j \leq N } \sum_{2 \leq \beta  \leq n }  A^{\alpha \beta}_{ij}\left( U_{\beta}^{i,p'} - G_{\beta}^{i,p'} \right) \right]
~ \text{ in } ~ Q_{5},
\end{aligned}\end{equation*}
for any $1 \leq i \leq N$. So the lemma follows.
\end{proof}

To use \eqref{HOD040} and Lemma \ref{HODS100}, for fixed $p' \geq 0'$ with $|p'|=m$, set
\begin{equation}\label{HOD270}
\bb^{p'}[z,w] = \left( \bb^{p'} _{1}[z,w], \cdots, \bb^{p'} _{n}[z,w] \right)
\qquad (z,w \in Q_{5}),
\end{equation}
where  
\begin{equation*}\begin{aligned}\label{}
\bb^{j,p'} _{1}[z,w] 
& =  \sum_{0' \leq q' \leq p' } 
\left( \begin{array}{c} p' \\ q' \end{array} \right)
\pi^{q'}(z) D^{ (|q'|,p'-q') }  D_{1}u^{j} (w) 
\end{aligned}\end{equation*}
and
\begin{equation*}\label{}
\bb_{\beta}^{j,p'}[z,w] = \sum_{0' \leq q' \leq p' } 
\left( \begin{array}{c} p' \\ q' \end{array} \right)
\pi^{q'}(z)
\Big[ D^{ (|q'|,p'-q') }  D_{\beta}u^{j}(w) + \pi_{\beta}(z) D^{ (|q'|,p'-q') }  D_{1}u^{j}(w)  \Big],
\end{equation*}
for any $2 \leq \beta \leq n$ and $1 \leq j \leq N$. Then we have that
\begin{equation}\label{HOD280}
\bb_{\beta}^{j,p'}[w,w] 
= \tilde{u}_{\beta}^{j,p'}(w) + [1-\delta_{\beta1}]  \pi_{\beta}(w) \tilde{u}_{1}^{j,p'}(w) 
\quad (1 \leq  \beta \leq n, \ 1 \leq j \leq N ),
\end{equation}
for any $w \in Q_{5}$. So it follows from Lemma \ref{GSS600} that
\begin{equation}\label{HOD285}
\left|  \bb^{p'}[w,w] \right|
\leq c \left| \tilde{u} ^{p'}(w) \right|
\end{equation}
and
\begin{equation}\label{HOD287}
\left| \bb^{p'} [w,w] - \bb^{p'} [z,z] \right|
\leq c \left[ \left| \tilde{u}^{p'} (w) - \tilde{u}^{p'}(z) \right| + |z-w|^{\gamma} \left| \tilde{u}^{p'}(z) \right| \right] 
\end{equation}
for any $w,z \in Q_{5}$.

\begin{lem}\label{HODS290}
For any $z,w \in Q_{5}^{k}$ $(k \in K)$ and $p' \geq 0'$ with $|p'|=m$, we have that
\begin{equation}\label{HOD292}
 \left| \bb^{p'} [z,w] - \bb^{p'} [w,w] \right|
\leq c |z-w|^{\gamma} \left| \tilde{u}^{p'}(w) \right|,
\end{equation}
\begin{equation}\label{HOD294} 
\left| \bb^{p'} [z,z] \right|
\leq c \left[  \left| U^{p'}(z) \right| +  \sum_{k \in K} \Big[ \| u \|_{C^{m,\gamma} (Q_{7}^{k}) }  +  \| F  \|_{C^{m,\gamma} (Q_{7}^{k}) }  \Big] \right],
\end{equation}
and
\begin{equation}\begin{aligned}\label{HOD296}
\left| \bb^{p'} [w,w] - \bb^{p'} [z,z] \right|
& \leq c  \left[  \left| U^{p'}(w) - U^{p'}(z)\right| 
+  |w-z|^{\gamma} \left| U^{p'}(z) \right| \right]  \\
& \quad + c |w-z|^{\gamma} \sum_{k \in K} \left[  \| u \|_{C^{m,\gamma} (Q_{7}^{k}) }  +  \| F \|_{C^{m,\gamma} (Q_{7}^{k}) }  \right].
\end{aligned}\end{equation}
In addition, we have that
\begin{equation}\label{HOD297}
\left| \tilde{u}^{p'} \right| 
\leq c \left[ \left| U^{p'} \right| + \left| G^{p'} \right|  + \left| \nu^{p'} \right|  \right]
\quad \text{ in } \quad Q_{5}.
\end{equation}
\end{lem}

\begin{proof}
Since $\pi \in C^{\gamma}(Q_{5})$, we obtain \eqref{HOD292} from \eqref{HOD270}. By \eqref{HOD040} and Lemma \ref{HODS100},
\begin{equation*}\begin{aligned}
& \sum_{1 \leq i,j \leq N } \sum_{1 \leq \alpha,  \beta \leq n } A^{\alpha \beta}_{ij}\pi_{\alpha} \pi_{\beta}    \tilde{u}_{1}^{i,p'} \tilde{u}_{1}^{j,p'} \\
& \quad = \sum_{ 1 \leq i \leq N }    \tilde{u}_{1}^{i,p'} \left[ G_{1}^{i,p'} - U_{1}^{i,p'} + \sum_{  1 \leq \alpha \leq n  } \pi_{\alpha} \left( \nu_{\alpha}^{i,p'} +  \sum_{ 1 \leq j \leq N } \sum_{2 \leq \beta  \leq n }  A^{\alpha \beta}_{ij}\left( U_{\beta}^{i,p'} - G_{\beta}^{i,p'} \right) \right) \right]
\end{aligned}\end{equation*}
in $Q_{5}$. It follows from \eqref{HOD140}  that 
\begin{equation*}\label{}
\lambda \left|   \tilde{u}_{1}^{p'}  \right|^{2}
\leq \sum_{1 \leq i,j \leq N } \sum_{1 \leq \alpha,  \beta \leq n } A^{\alpha \beta}_{ij}\pi_{\alpha} \pi_{\beta}    \tilde{u}_{1}^{i,p'} \tilde{u}_{1}^{j,p'} 
\leq c \left|   \tilde{u}_{1}^{p'}  \right| \left[   \left| U^{p'} \right| + \left| G^{p'} \right| + \left| \nu^{p'} \right|  \right].
\end{equation*}
So \eqref{HOD297} holds from  \eqref{HOD085}. Also with \eqref{HOD285} and \eqref{HOD297}, \eqref{HOD294} holds from  \eqref{HOD065} and \eqref{HOD090}. To show \eqref{HOD296}, investigate that
\begin{equation*}\begin{aligned}
& \sum_{1 \leq j \leq N } \sum_{1 \leq \alpha,  \beta \leq n } A_{\alpha \beta}^{ij}(w) \pi_{\alpha}(w) \pi_{\beta}(w)  \left[ \tilde{u}_{1}^{j,p'}(w) -  \tilde{u}_{1}^{j,p'}(z) \right] \\
& \quad = \sum_{1 \leq j \leq N } \sum_{1 \leq \alpha,  \beta \leq n } 
\left[ A_{\alpha \beta}^{ij}(w) \pi_{\alpha}(w) \pi_{\beta}(w) \tilde{u}_{1}^{j,p'}(w) -  A_{\alpha \beta}^{ij}(z) \pi_{\alpha}(z) \pi_{\beta}(z) \tilde{u}_{1}^{j,p'}(z) \right]  \\
& \qquad + \sum_{1 \leq j \leq N } \sum_{1 \leq \alpha,  \beta \leq n } 
\left[ A_{\alpha \beta}^{ij}(z) \pi_{\alpha}(z) \pi_{\beta}(z) \tilde{u}_{1}^{j,p'}(z) -  A_{\alpha \beta}^{ij}(w) \pi_{\alpha}(w) \pi_{\beta}(w) \tilde{u}_{1}^{j,p'}(z) \right].
\end{aligned}\end{equation*}
Recall that $A_{\alpha \beta}^{ij}, \pi_{\alpha} \in C^{\gamma} \left( Q_{5}^{k} \right)$ $( k \in K)$. Then we have from Lemma \ref{HODS100} that 
\begin{equation*}\begin{aligned}
& \sum_{1 \leq i,j \leq N } \sum_{1 \leq \alpha,  \beta \leq n } A_{\alpha \beta}^{ij}(w) \pi_{\alpha}(w) \pi_{\beta}(w)  \left[ \tilde{u}_{1}^{i,p'}(w) -  \tilde{u}_{1}^{i,p'}(z) \right]   \left[ \tilde{u}_{1}^{j,p'}(w) -  \tilde{u}_{1}^{j,p'}(z) \right] \\
& \quad \leq c \left| \tilde{u}_{1}^{p'} (w) - \tilde{u}_{1}^{p'}(z) \right| 
 \left[ \big| U^{p'}(w) - U^{p'}(z)\big| + \big| G^{p'}(w) - G^{p'}(z) \big| + \big| \nu^{p'}(w) - \nu^{p'}(z) \big| \right] \\
& \qquad  + c  |w-z|^{\gamma}  \left| \tilde{u}_{1}^{p'} (w) - \tilde{u}_{1}^{p'}(z) \right|
\left| \tilde{u}^{p'}(z) \right|,
\end{aligned}\end{equation*}
for any $ w, z \in Q_{5}^{k}$ $(k \in K)$. It follows from \eqref{HOD140} that
\begin{equation*}\begin{aligned}\label{}
\left| \tilde{u}_{1}^{p'}(w) -  \tilde{u}_{1}^{p'}(z) \right|
& \leq c   \left[ \big| U^{p'}(w) - U^{p'}(z)\big| + \big| G^{p'}(w) - G^{p'}(z) \big| + \big| \nu^{p'}(w) - \nu^{p'}(z) \big| \right]  \\
& \quad + c |w-z|^{\gamma}  \left| \tilde{u}^{p'}(z) \right|,
\end{aligned}\end{equation*}
for any $ w, z \in Q_{5}^{k}$ $(k \in K)$. So  by \eqref{HOD085} and that $\pi_{\alpha} \in C^{\gamma} \left( Q_{5}^{k} \right)$ $( k \in K)$,
\begin{equation*}\begin{aligned}\label{}
\left| \tilde{u}^{p'}(w) -  \tilde{u}^{p'}(z) \right|
& \leq c   \left[ \big| U^{p'}(w) - U^{p'}(z)\big| + \big| G^{p'}(w) - G^{p'}(z) \big| + \big| \nu^{p'}(w) - \nu^{p'}(z) \big| \right]  \\
& \quad + c |w-z|^{\gamma} \left| \tilde{u}^{p'}(z) \right|,
\end{aligned}\end{equation*}
for any $ w, z \in Q_{5}^{k}$ $(k \in K)$.  It follows from \eqref{HOD287} and \eqref{HOD297} that 
\begin{equation*}\begin{aligned}\label{}
& \left| \bb^{p'} [w,w] - \bb^{p'} [z,z] \right|  \\
& \quad \leq c \left[ \big| U^{p'}(w) - U^{p'}(z)\big| + \big| G^{p'}(w) - G^{p'}(z) \big| + \big| \nu^{p'}(w) - \nu^{p'}(z) \big|  \right] \\
& \qquad  + c |z-w|^{\gamma}  \left[ \big| U^{p'}(z)\big| + \big| G^{p'}(z) \big| + \big|  \nu^{p'}(z) \big|  \right],
\end{aligned}\end{equation*}
for any $ w, z \in Q_{5}^{k}$ $(k \in K)$.  So \eqref{HOD296} follows from \eqref{HOD065} and \eqref{HOD090}.
\end{proof}

To derive Lemma \ref{HODS550} and Lemma \ref{HODS900}, we obtain the estimates with respect to $y$-coordinate system as in Lemma \ref{HODS200}.  

\begin{lem}\label{HODS200}
For any $z \in Q_{5}$, let $\Psi_{z} : Q_{5} \to \br^{n}$ be a coordinate transformation with  $y = \Psi_{z}(x) = \big( \Psi_{z}^{1}(x), \Psi_{z}'(x) \big)$ and $\Phi_{z}^{-1} = \Psi_{z}$ such  that
\begin{equation*}\label{} 
\Psi_{z}^{1}(x) = x^{1} - z^{1} - \pi'(z) \cdot (x'-z')
\qquad \text{and} \qquad
\Psi_{z}'(x) = x'-z'.
\end{equation*} 
Then for any $\alpha,\beta \in \{ 2, \cdots, n \}$, we have that 
\begin{equation*}\label{} 
\frac{ \partial x^{1} }{ \partial y^{1} } = 1, 
\qquad
\frac{ \partial x^{1} }{ \partial y^{\beta} } =  \pi_{\beta}(z),
\qquad
\frac{ \partial x^{\alpha} }{ \partial y^{1} } 
= 0,
\qquad
\frac{ \partial x^{\alpha} }{ \partial y^{\beta} } = \delta_{ij},
\end{equation*}
and
\begin{equation*}\label{} 
\frac{ \partial y^{1} }{ \partial x^{1} } = 1, 
\qquad
\frac{ \partial y^{1} }{ \partial x^{\beta} } = -  \pi_{\beta}(z),
\qquad
\frac{ \partial y^{\alpha} }{ \partial x^{1} } 
= 0,
\qquad
\frac{ \partial y^{\alpha} }{ \partial x^{\beta} } = \delta_{ij}.
\end{equation*}
\end{lem}

\begin{proof}
Since $x^{1} = y^{1} + z^{1} + \pi'(z) \cdot y'$ and $x' = y' + z'$, we find that 
\begin{equation*}\label{} 
\frac{ \partial x^{1} }{ \partial y^{1} } = 1, 
\quad
\frac{ \partial x^{1} }{ \partial y^{\beta} } =  \pi_{\beta}(z),
\quad
\frac{ \partial x^{\alpha} }{ \partial y^{1} } 
= 0 
\quad \text{and} \quad
\frac{ \partial x^{\alpha} }{ \partial y^{\beta} } = \delta_{\alpha \beta}
\qquad (\alpha,\beta \in \{ 2, \cdots, n \} ).
\end{equation*}
Since $y^{1} = x^{1} - \pi'(z) \cdot (x'-z')$ and $y' = x' - z'$, we find that
\begin{equation*}\label{} 
\frac{ \partial y^{1} }{ \partial x^{1} } = 1, 
\quad
\frac{ \partial y^{1} }{ \partial x^{\beta} } = -  \pi_{\beta}(z),
\quad
\frac{ \partial y^{\alpha} }{ \partial x^{1} } 
= 0 
\quad \text{and} \quad
\frac{ \partial y^{\alpha} }{ \partial x^{\beta} } = \delta_{\alpha \beta}
\qquad (\alpha,\beta \in \{ 2, \cdots, n \} ).
\end{equation*}
This completes the proof.
\end{proof}

\begin{lem}\label{HODS250}
For $z \in Q_{5}$,  let $y=\Psi_{z}(x)$ and $\Phi_{z}^{-1} = \Psi_{z}$  as in Lemma \ref{HODS200}. Also let
\begin{equation}\label{HOD253}
\hat{A}^{\alpha \beta}_{ij}(y) = \sum_{1 \leq s, t \leq n} A^{st}_{ij} \big( \Phi_{z}(y) \big)   \frac{ \partial y^{\alpha} }{ \partial x^{s}} \frac{ \partial y^{\beta} }{ \partial x^{t} } \big( \Phi_{z}(y) \big)
\quad \text{in} \quad
\Psi_{z}(Q_{5}),
\end{equation}
and
\begin{equation}\label{HOD256}
\hat{F}_{\alpha}(y) = \frac{ \partial y^{\alpha} }{ \partial x^{s} } \big( \Phi_{z}(y) \big)  F_{s} \big( \Phi_{z}(y) \big)
\quad \text{in} \quad
\Psi_{z}(Q_{5}).
\end{equation}
Then for $\hat{u}(y) = u \big( \Phi_{z}(y) \big)$, we have that  
\begin{equation}\label{HOD260} 
D_{y^{\alpha}} \left[ \hat{A}^{\alpha \beta}_{ij}D_{y^{\beta}}\hat{u}^{j} \right]
= D_{y^{\alpha}} \hat{F}_{\alpha}^{i}  
\quad \text{in} \quad
\Psi_{z}(Q_{5}),
\end{equation}
and $\hat{A}^{\alpha \beta}_{ij}, \hat{u}, \hat{F} \in C^{m,\gamma} \left( \Psi_{z} \left( Q_{5}^{k} \right) \right)$ for any $k \in K$ with the estimate that
\begin{equation}\label{HOD263} 
\| \hat{u} \|_{C^{m,\gamma} \left( \Psi_{z}(Q_{5}^{k}) \right)}  +  \big \| \hat{F} \big \|_{C^{m,\gamma} \left( \Psi_{z}(Q_{5}^{k}) \right)}  
\leq c \Big[ \| u \|_{C^{m,\gamma} (Q_{5}^{k}) }  +  \big \| F \big \|_{C^{m,\gamma} (Q_{5}^{k}) } \Big].
\end{equation}
\end{lem}

\begin{proof}
Since $D_{x^{\alpha}} \left[  a_{ij} D_{x^{j}} u - F_{i} \right] = 0$ in $Q_{5}$, by changing variables,
\begin{equation*}
D_{y^{\alpha}} \left[ A^{st}_{ij}  \big( \Phi_{z}(y) \big)  \frac{ \partial y^{\alpha} }{ \partial x^{s}} \frac{ \partial y^{\beta} }{ \partial x^{t} } 
 D_{y^{\beta}} \hat{u} \right] =  D_{y^{\alpha}} \left[  \frac{ \partial y^{\alpha} }{ \partial x^{s} } F_{s}  \right]
\quad \text{in} \quad
\Psi_{z}(Q_{5}).
\end{equation*}
With \eqref{PET130} and Lemma \ref{HODS200}, the lemma holds by \eqref{HOD256} and that $\hat{u}(y) = u \big( \Phi(y) \big)$.
\end{proof}

\begin{lem}\label{HODS300}
For a fixed $z \in Q_{5}$,  let $y=\Psi_{z}(x)$,  $\Phi_{z}^{-1} = \Psi_{z}$ and $\hat{u}(y) = u \left( \Phi_{z}(y) \right)$ as in Lemma \ref{HODS200}. Then for any $w \in Q_{5}$ and $p' \geq 0'$ with $|p'| = m$, we have that
\begin{equation*}\label{} 
D_{y^{\beta}}D_{y'}^{p'}\hat{u} \big( \Psi_{z}(w) \big)
= \bb^{p'} _{\beta}[z,w]
\qquad \qquad ( \beta \in \{ 1, \cdots, n \} ).
\end{equation*}
\end{lem}

\begin{proof}
We prove the lemma for a fixed $p' \geq 0'$ with $|p'| = m$. By Lemma \ref{HODS200},
\begin{equation*}\label{} 
\frac{ \partial x^{1} }{ \partial y^{1} } = 1, 
\quad
\frac{ \partial x^{1} }{ \partial y^{\beta} } =  \pi_{\beta}(z),
\quad
\frac{ \partial x^{\alpha} }{ \partial y^{1} } 
= 0 
\quad \text{and} \qquad
\frac{ \partial x^{\alpha} }{ \partial y^{\beta} } = \delta_{\alpha \beta}
\qquad
(\alpha,\beta \in \{ 2, \cdots, n \}).
\end{equation*}
Fix $w \in Q_{5}$. By the chain rule,
\begin{equation*}\begin{aligned}\label{} 
& D_{y'}^{p'}\hat{u} \big( \Psi_{z}(w) \big) \\
& \ \ = \sum_{0 \leq q' \leq p'} 
\Big( \begin{array}{c} p' \\ q' \end{array} \Big) \,
\bigg( \frac{ \partial x^{1} }{ \partial y^{2} } \bigg)^{ q^{2} } 
\bigg( \frac{ \partial x^{2} }{ \partial y^{2} } \bigg)^{ p^{2} - q^{2} }
\cdots 
\bigg( \frac{ \partial x^{1} }{ \partial y^{n} } \bigg)^{ q^{n} }
\bigg( \frac{ \partial x^{n} }{ \partial y^{n} } \bigg)^{ p^{n} - q^{n} }
D_{x}^{(|q'|,p'-q')}u(w) \\
& \ \ = \sum_{0 \leq q' \leq p'} 
\Big( \begin{array}{c} p' \\ q' \end{array} \Big)
\pi^{q'}(z) D_{x}^{(|q'|,p'-q')}u(w) .
\end{aligned}\end{equation*}
Then we apply the chain rule again to find that  
\begin{equation*}\begin{aligned}\label{} 
D_{y^{1}}D_{y'}^{p'}\hat{u} \big( \Psi_{z}(w) \big)
= \sum_{0 \leq q' \leq p'} 
\Big( \begin{array}{c} p' \\ q' \end{array} \Big)
\pi^{q'}(z) D_{x^{1}}D_{x}^{(|q'|,p'-q')}u(w)
= \bb^{p'} _{1}[z,w],
\end{aligned}\end{equation*}
and
\begin{equation*}\begin{aligned}\label{} 
& D_{y^{\beta}}D_{y'}^{p'}\hat{u} \big( \Psi_{z}(w) \big) \\
& \quad = \sum_{0 \leq q' \leq p'} 
\Big( \begin{array}{c} p' \\ q' \end{array} \Big)
\pi^{q'}(z) \bigg[ \frac{ \partial x^{\beta} }{ \partial y^{\beta} } \cdot D_{x^{\beta}}D_{x}^{(|q'|,p'-q')}u(w) +  \frac{ \partial x^{1} }{ \partial y^{\beta} } \cdot  D_{1}D_{x}^{(|q'|,p'-q')}u(w) \bigg] \\
& \quad = \sum_{0 \leq q' \leq p'} 
\Big( \begin{array}{c} p' \\ q' \end{array} \Big)
\pi^{q'}(z) \Big[ D_{x^{\beta}} D_{x}^{(|q'|,p'-q')}u(w) + \pi_{\beta}(z) D_{x^{1}}D_{x}^{(|q'|,p'-q')}u(w) \Big] \\
& \quad = \bb_{\beta}^{p'}[z,w],
\end{aligned}\end{equation*}
for any $\beta \in \{ 2, \cdots, n \}$. Since $w$ was arbitrary chosen, the lemma follows.
\end{proof}

We will estimate $|D^{m+1}u|$ in Lemma \ref{HODS500}.

\begin{lem}\label{HODS400}
For a fixed $z \in Q_{5}$,  let $y=\Psi_{z}(x)$,  $\Phi_{z}^{-1} = \Psi_{z}$ and $\hat{u}(y) = u \left( \Phi_{z}(y) \right)$ as in Lemma \ref{HODS200}. Then we have that
\begin{equation*}\begin{aligned}\label{}
\left| D_{y}^{m+1}\hat{u}  \big( \Psi_{z}(z) \big)  \right|  
\leq c \left[ \sum_{|p'|=m} \left| \bb^{p'} [z,z] \right| + \sum_{k \in K} \Big[  \| u \|_{C^{m,\gamma} (Q_{5}^{k}) }  +  \| F  \|_{C^{m,\gamma} (Q_{5}^{k}) } \Big] \right].
\end{aligned}\end{equation*}
\end{lem}

\begin{proof}
Fix nonnegative integers $ l_{1}$ and $l_{2} $ with $l_{1} + l_{2} = m-1$. We differentiate  \eqref{HOD260} in Lemma \ref{HODS250} by $D_{y}^{l_{1}}D_{y'}^{l_{2}}$ to find that for any $i \in \{  1, \cdots, N \}$,
\begin{equation*}\begin{aligned}
& \left| \sum_{ 1 \leq j \leq N }  \hat{A}^{11}_{ij} \left( D_{y}^{l_{1}}D_{y'}^{l_{2}}D_{y^{1}y^{1}} \hat{u}^{j} \right) \right| \\
& \quad \leq c \left[ \big|D_{y}^{l_{1}+1}D_{y'}^{l_{2}+1} \hat{u} \big| 
+ \sum_{k \in K} \left[ \| \hat{u} \|_{C^{m,\gamma} \left( \Psi(Q_{5}^{k}) \right)}  +  \big \| \hat{F} \big \|_{C^{m,\gamma} \left( \Psi(Q_{5}^{k}) \right)}  \right] \right]
\end{aligned}\end{equation*}
in $\Psi ( Q_{5} )$, which implies that
\begin{equation*}\begin{aligned}
& \sum_{ 1 \leq i,j \leq N }  \hat{A}^{11}_{ij} \left( D_{y}^{l_{1}}D_{y'}^{l_{2}}D_{y^{1}y^{1}} \hat{u}^{j} \right) \left( D_{y}^{l_{1}}D_{y'}^{l_{2}}D_{y^{1}y^{1}} \hat{u}^{i} \right) \\
& \quad \leq c \left| D_{y}^{l_{1}}D_{y'}^{l_{2}}D_{y^{1}y^{1}} \hat{u}  \right|
\left[ \big|D_{y}^{l_{1}+1}D_{y'}^{l_{2}+1} \hat{u} \big| 
+ \sum_{k \in K} \left[ \| \hat{u} \|_{C^{m,\gamma} \left( \Psi(Q_{5}^{k}) \right)}  +  \big \| \hat{F} \big \|_{C^{m,\gamma} \left( \Psi(Q_{5}^{k}) \right)}  \right] \right]
\end{aligned}\end{equation*}
in $\Psi ( Q_{5} )$. So it follows from \eqref{ell} that 
\begin{equation*}\begin{aligned}
\left| D_{y}^{l_{1}}D_{y'}^{l_{2}}D_{y^{1}y^{1}} \hat{u}  \right|
\leq c \left[ \big|D_{y}^{l_{1}+1}D_{y'}^{l_{2}+1} \hat{u} \big| 
+ \sum_{k \in K} \left[ \| \hat{u} \|_{C^{m,\gamma} \left( \Psi(Q_{5}^{k}) \right)}  +  \big \| \hat{F} \big \|_{C^{m,\gamma} \left( \Psi(Q_{5}^{k}) \right)}  \right] \right]
\end{aligned}\end{equation*}
in $\Psi ( Q_{5} )$, which implies that 
\begin{equation*}\label{} 
\big| D_{y}^{l_{1}+2}D_{y'}^{l_{2}}\hat{u} \big|  
\leq c \left[ \big| D_{y}^{l_{1}+1}D_{y'}^{l_{2}+1}\hat{u} \big| + \sum_{k \in K} \left[ \| \hat{u} \|_{C^{m,\gamma} \left( \Psi(Q_{5}^{k}) \right)}  +  \big \| \hat{F} \big \|_{C^{m,\gamma} \left( \Psi(Q_{5}^{k}) \right)}  \right]   \right]
\end{equation*}
in $\Psi ( Q_{5} )$. Since  nonnegative integers $ l_{1}$ and $l_{2} $ with $l_{1} + l_{2} = m-1$ were arbitrary chosen, we find  from \eqref{HOD263} in Lemma \ref{HODS250} that
\begin{equation*}\label{HOD420} 
\left| D_{y}^{m+1}\hat{u}  \big( \Psi_{z}(z) \big)   \right| 
\leq c \left[ \big| D_{y}D_{y'}^{m}\hat{u}  \big( \Psi_{z}(z) \big)   \big| + \sum_{k \in K} \left[  \| u \|_{C^{m,\gamma} (Q_{5}^{k}) }  +  \| F  \|_{C^{m,\gamma} (Q_{5}^{k}) } \right] \right].
\end{equation*}
So the lemma holds from Lemma \ref{HODS300}.
\end{proof}

\begin{lem}\label{HODS500}
For a fixed $z \in Q_{5}$,  let $y=\Psi_{z}(x)$,  $\Phi_{z}^{-1} = \Psi_{z}$ and $\hat{u}(y) = u \left( \Phi_{z}(y) \right)$ as in Lemma \ref{HODS200}. Then we have that
 \begin{equation*}\begin{aligned}
\big| D^{m+1}u(z) \big| 
\leq c \left[ \sum_{|p'|=m}  \left| \bb^{p'}[z,z] \right|+ \sum_{k \in K} \Big[ \| u \|_{C^{m,\gamma} (Q_{5}^{k}) }  +  \| F  \|_{C^{m,\gamma} (Q_{5}^{k}) }  \Big] \right].
\end{aligned}\end{equation*}
\end{lem}

\begin{proof}
We claim that 
\begin{equation}\begin{aligned}\label{HOD520} 
\big| D_{x}^{m+1}u(z) \big| \leq c \big| D_{y}^{m+1}\hat{u} \big( \Psi(z) \big)  \big|.
\end{aligned}\end{equation}
By the chain rule, for any $p' \geq 0$ with $|p'| \leq m+1$, we have from Lemma \ref{HODS200} that
\begin{equation*}\begin{aligned}\label{} 
D_{x^{1}}^{m+1-|p'|} D_{x'}^{p'}u (z) 
& = \sum_{0 \leq q' \leq p'} 
\Big( \begin{array}{c} p' \\ q' \end{array} \Big) 
\bigg( \frac{ \partial y^{1} }{ \partial x^{2} } \bigg)^{ p^{2} - q^{2}} 
\bigg( \frac{ \partial y^{2} }{ \partial x^{2} } \bigg)^{ q^{2} }
\cdots 
\bigg( \frac{ \partial y^{1} }{ \partial x^{n} } \bigg)^{ p^{n} - q^{n} }
\bigg( \frac{ \partial y^{n} }{ \partial x^{n} } \bigg)^{ q^{n} } \\
& \qquad \qquad \times D_{y^{1}}^{m+1-|p'|}D_{y}^{(|p'|-|q'|,q')} \hat{u} \big( \Psi(z) \big) \\
& = \sum_{0 \leq q' \leq p'}  
\left( \begin{array}{c} p' \\ q' \end{array} \right)
(-1)^{|p'-q'|} \pi^{p'-q'}(z) D_{y}^{(m+1-|q'|,q')}\hat{u} \big( \Psi(z) \big).
\end{aligned}\end{equation*}
It follows that for any $p' \geq 0'$ with $|p'| \leq m+1$,
\begin{equation*}\begin{aligned}\label{} 
\left| D_{x^{1}}^{m+1-|p'|} D_{x'}^{p'}u (z)  \right| \leq c \left| D_{y}^{m+1} \hat{u} \left( \Psi_{z}(z) \right) \right|.
\end{aligned}\end{equation*}
So the claim \eqref{HOD520} holds. The lemma follows from \eqref{HOD520} and Lemma \ref{HODS400}.
\end{proof}

\begin{lem}\label{HODS550}
For any  $z \in Q_{5}$,  we have that
\begin{equation*}\label{} 
\big| D^{m+1}u(z) \big|
\leq c \left[ \sum_{|p'| = m}  \left| U^{p'}(z) \right| +  \sum_{k \in K} \Big[ \| u \|_{C^{m,\gamma} (Q_{7}^{k}) }  +  \| F  \|_{C^{m,\gamma} (Q_{7}^{k}) }  \Big] \right].
\end{equation*}
\end{lem}

\begin{proof}
The lemma follows from \eqref{HOD294} in Lemma \ref{HODS290} and  Lemma \ref{HODS500}. 
\end{proof}

We next estimate $ \left| D^{m+1}u(w) - D^{m+1}u(z) \right|$ for any $w,z \in Q_{5}^{k}$ $(k \in K)$.

\begin{lem}\label{HODS600}
For a fixed $z \in Q_{5}$,  let $y=\Psi_{z}(x)$,  $\Phi_{z}^{-1} = \Psi_{z}$ and $\hat{u}(y) = u \left( \Phi_{z}(y) \right)$ as in Lemma \ref{HODS200}. Then for any $w \in Q_{5}$, we have that
\begin{equation*}\begin{aligned}
& \left| D_{y}D_{y'}^{m}\hat{u} \big( \Psi_{z}(w) \big) - D_{y}D_{y'}^{m}\hat{u} \big( \Psi_{z}(z) \big)  \right|  \\
& \quad \leq  c \sum_{|p'|=m} \left[ \left| \bb^{p'}[w,w] - \bb^{p'}[z,z] \right|  + |w-z|^{\gamma} \left| \bb^{p'}[z,z] \right| \right ] \\
& \qquad + c |w-z|^{\gamma} \sum_{k \in K} \Big[  \| u \|_{C^{m,\gamma} (Q_{5}^{k}) }  +  \| F \|_{C^{m,\gamma} (Q_{5}^{k}) }  \Big].
\end{aligned}\end{equation*}
\end{lem}

\begin{proof}
We have from Lemma \ref{HODS300} that
\begin{equation*}\begin{aligned}\label{} 
D_{y}D_{y'}^{p'}\hat{u} \big( \Psi_{z}(w) \big) - D_{y}D_{y'}^{p'}\hat{u} \big( \Psi_{z}(z) \big) 
= \bb^{p'}[z,w] - \bb^{p'}[z,z],
\end{aligned}\end{equation*}
for any $w \in Q_{5}$ and $p' \geq 0'$ with $|p'|=m$. It follows that
\begin{equation*}
\big| D_{y}D_{y'}^{p'}\hat{u} \big( \Psi_{z}(w) \big) - D_{y}D_{y'}^{p'}\hat{u} \big( \Psi_{z}(z) \big)  \big|
\leq \left| \bb^{p'}[z,w] - \bb^{p'}[w,w] \right| + \left| \bb^{p'}[w,w] - \bb^{p'}[z,z] \right|.
\end{equation*}
for any $w \in Q_{5}$ and $p' \geq 0'$ with $|p'|=m$. We find from \eqref{HOD292} in Lemma \ref{HODS290} that
\begin{equation*}\begin{aligned}\label{}
\left| \bb^{p'}[z,w] - \bb^{p'}[w,w] \right| 
&\leq c |z-w|^{\gamma}|D^{m+1}u(w)| 
\end{aligned}\end{equation*}
for any $w \in Q_{5}$ and $p' \geq 0'$ with $|p'|=m$.  By combining the above two estimates,  the lemma follows from Lemma \ref{HODS500}.
\end{proof}

The proof of the following lemma is almost parallel to that of Lemma \ref{HODS400}.

\begin{lem}\label{HODS700}
For a fixed $z \in Q_{5}$,  let $y=\Psi_{z}(x)$,  $\Phi_{z}^{-1} = \Psi_{z}$ and $\hat{u}(y) = u \left( \Phi_{z}(y) \right)$ as in Lemma \ref{HODS200}. Then for any $w \in Q_{5}^{k}$ $(k \in K)$,we have that
\begin{equation*}\begin{aligned}\label{}
& \Big| D_{y}^{m+1}\hat{u} \big( \Psi_{z}(w) \big) - D_{y}^{m+1}\hat{u} \big( \Psi_{z}(z) \big) \Big| \\
& \quad \leq  c \sum_{|p'|=m} \left[ \left| \bb^{p'}[w,w] - \bb^{p'}[z,z] \right|  + |w-z|^{\gamma} \left| \bb^{p'}[z,z] \right| \right ] \\
& \qquad + c |w-z|^{\gamma} \sum_{k \in K} \Big[  \| u \|_{C^{m,\gamma} (Q_{5}^{k}) }  +  \| F \|_{C^{m,\gamma} (Q_{5}^{k}) }  \Big].
\end{aligned}\end{equation*}
\end{lem}

\begin{proof}
Fix $w,z \in Q_{5}^{k}$ $(k \in K)$, nonnegative integers $ l_{1}$ and $l_{2} $ with $l_{1} + l_{2} = m-1$. We differentiate  \eqref{HOD260} in Lemma \ref{HODS250} by $D_{y}^{l_{1}}D_{y'}^{l_{2}}$ to find that
\begin{equation*}\begin{aligned}
& \left | \sum_{ 1 \leq j \leq N } \left[ \hat{A}^{11}_{ij}\big( \Psi_{z}(w) \big)  D_{y}^{l_{1}}D_{y'}^{l_{2}}D_{y^{1}y^{1}}\hat{u}^{j}\big( \Psi_{z}(w) \big) -  \hat{A}^{11}_{ij}\big( \Psi_{z}(z) \big) D_{y}^{l_{1}}D_{y'}^{l_{2}}D_{y^{1}y^{1}}\hat{u}^{j}\big( \Psi_{z}(z) \big)  \right] \right| \\
& \quad \leq c |w-z|^{\gamma} \left[  \left| D_{y}^{m+1}\hat{u}\big( \Psi_{z}(z) \big) \right| + \left \| \hat{u} \right \|_{C^{m,\gamma} \left( \Psi_{z}(Q_{5}^{k}) \right)}  +  \big\| \hat{F} \big\|_{C^{m,\gamma} \left( \Psi_{z}(Q_{5}^{k}) \right)}  \right] \\
& \qquad + c \left |  D_{y}^{l_{1}+1}D_{y'}^{l_{2}+1}\hat{u}\big( \Psi_{z}(w) \big) - D_{y}^{l_{1}+1}D_{y'}^{l_{2}+1}\hat{u}\big( \Psi_{z}(z) \big) \right |.
\end{aligned}\end{equation*}
By the triangle inequality,
\begin{equation*}\begin{aligned}\label{}
& \left | \sum_{ 1 \leq j \leq N }  \hat{A}^{11}_{ij}\big( \Psi_{z}(w) \big)  \left[ D_{y}^{l_{1}}D_{y'}^{l_{2}}D_{y^{1}y^{1}}\hat{u}^{j}\big( \Psi_{z}(w) \big) -   D_{y}^{l_{1}}D_{y'}^{l_{2}}D_{y^{1}y^{1}}\hat{u}^{j}\big( \Psi_{z}(z) \big)  \right] \right| \\
& \quad \leq \left | \sum_{ 1 \leq j \leq N } \left[ \hat{A}^{11}_{ij}\big( \Psi_{z}(w) \big)  D_{y}^{l_{1}}D_{y'}^{l_{2}}D_{y^{1}y^{1}}\hat{u}^{j}\big( \Psi_{z}(w) \big) -  \hat{A}^{11}_{ij}\big( \Psi_{z}(z) \big) D_{y}^{l_{1}}D_{y'}^{l_{2}}D_{y^{1}y^{1}}\hat{u}^{j}\big( \Psi_{z}(z) \big)  \right] \right| \\
& \qquad + \left| \hat{A}^{11}_{ij}\big( \Psi_{z}(z) \big) - \hat{A}^{11}_{ij}\big( \Psi_{z}(w) \big) \right| \left| D_{y}^{l_{1}}D_{y'}^{l_{2}}D_{y^{1}y^{1}}\hat{u}\big( \Psi_{z}(z) \big) \right| .
\end{aligned}\end{equation*}
Since $\hat{A}^{11}_{ij}\big( \Psi_{z}(w) \big) \geq c \lambda$ and $\hat{a}_{ij},  \hat{F} \in C^{m,\gamma} \left( \Psi_{z}(Q_{5}^{k}) \right)$ for any $k \in K$, it follows that
\begin{equation*}\begin{aligned}\label{}
& \left | \sum_{ 1 \leq j \leq N }  \hat{A}^{11}_{ij}\big( \Psi_{z}(w) \big)  \left[ D_{y}^{l_{1}}D_{y'}^{l_{2}}D_{y^{1}y^{1}}\hat{u}\big( \Psi_{z}(w) \big) -   D_{y}^{l_{1}}D_{y'}^{l_{2}}D_{y^{1}y^{1}}\hat{u}\big( \Psi_{z}(z) \big)  \right] \right| \\
& \quad \leq c |w-z|^{\gamma} \Big[  \big| D_{y}^{m+1}\hat{u}\big( \Psi_{z}(z) \big) \big| + \| \hat{u} \|_{C^{m,\gamma} \left( \Psi_{z}(Q_{5}^{k}) \right)}  +  \big\| \hat{F} \big\|_{C^{m,\gamma} \left( \Psi_{z}(Q_{5}^{k}) \right)}  \Big] \\
& \qquad + c \Big|  D_{y}^{l_{1}+1}D_{y'}^{l_{2}+1}\hat{u}\big( \Psi_{z}(w) \big) - D_{y}^{l_{1}+1}D_{y'}^{l_{2}+1}\hat{u}\big( \Psi_{z}(z) \big) \Big|.
\end{aligned}\end{equation*}
We multiply  $D_{y}^{l_{1}}D_{y'}^{l_{2}}D_{y^{1}y^{1}}\hat{u}^{i}\big( \Psi_{z}(w) \big) -   D_{y}^{l_{1}}D_{y'}^{l_{2}}D_{y^{1}y^{1}}\hat{u}^{i}\big( \Psi_{z}(z) \big)$ on each side and then sum it over $i \in \{1, \cdots , N \}$. By applying \eqref{ell}, we obtain that 
\begin{equation*}\begin{aligned}\label{}
& \Big| D_{y}^{l_{1}+2}D_{y'}^{l_{2}}\hat{u}\big( \Psi_{z}(w) \big) - D_{y}^{l_{1}+2}D_{y'}^{l_{2}}\hat{u}\big( \Psi_{z}(z) \big) \Big| \\
& \quad \leq c |w-z|^{\gamma} \Big[  \big| D_{y}^{m+1}\hat{u}\big( \Psi_{z}(z) \big) \big| + \| \hat{u} \|_{C^{m,\gamma} \left( \Psi_{z}(Q_{5}^{k}) \right)}  +  \big\| \hat{F} \big\|_{C^{m,\gamma} \left( \Psi_{z}(Q_{5}^{k}) \right)}  \Big] \\
& \qquad + c \Big|  D_{y}^{l_{1}+1}D_{y'}^{l_{2}+1}\hat{u}\big( \Psi_{z}(w) \big) - D_{y}^{l_{1}+1}D_{y'}^{l+1}\hat{u}\big( \Psi_{z}(z) \big) \Big|.
\end{aligned}\end{equation*}
Since nonnegative integers $ l_{1}$ and $l_{2} $ with $l_{1} + l_{2} = m-1$ were arbitrary chosen, 
\begin{equation*}\begin{aligned}\label{}
& \Big| D_{y}^{m+1}\hat{u} \big( \Psi_{z}(w) \big) - D_{y}^{m+1}\hat{u} \big( \Psi_{z}(z) \big) \Big| \\
& \quad \leq  c \big| D_{y}D_{y'}^{m}\hat{u} \big( \Psi_{z}(w) \big)  - D_{y}D_{y'}^{m}\hat{u} \big( \Psi_{z}(z) \big)  \big| \\
& \qquad + c |w-z|^{\gamma} \Big[ \big| D_{y}D_{y'}^{m}\hat{u} \big( \Psi_{z}(z) \big)  \big| +  \| u \|_{C^{m,\gamma} (Q_{5}^{k}) }  +  \| F \|_{C^{m,\gamma} (Q_{5}^{k}) }  \Big].
\end{aligned}\end{equation*}
Since $w,z \in Q_{5}^{k}$ $(k \in K)$ were arbitrary chosen, the lemma follows from Lemma \ref{HODS400} and Lemma \ref{HODS600}.
\end{proof}

\begin{lem}\label{HODS800}
For any $z,w \in Q_{5}^{k}$ $(k \in K)$, we have that
\begin{equation*}\begin{aligned}
 \big| D^{m+1}u(w) - D^{m+1}u(z)\big|  
& \leq  c \sum_{|p'|=m} \left[ \left| \bb^{p'}[w,w] - \bb^{p'}[z,z] \right|  + |w-z|^{\gamma} \left| \bb^{p'}[z,z] \right| \right ] \\
& \quad + c |w-z|^{\gamma} \sum_{k \in K} \Big[  \| u \|_{C^{m,\gamma} (Q_{5}^{k}) }  +  \| F \|_{C^{m,\gamma} (Q_{5}^{k}) }  \Big].
\end{aligned}\end{equation*}
\end{lem}

\begin{proof}
For a fixed $z,w \in Q_{5}^{k}$ $(k \in K)$,  let $y=\Psi_{z}(x)$,  $\Phi_{z}^{-1} = \Psi_{z}$ and $\hat{u}(y) = u \left( \Phi_{z}(y) \right)$ as in Lemma \ref{HODS200}. We claim that 
\begin{equation}\label{HOD830} 
\big| D_{x}^{m+1}u(w) - D_{x}^{m+1}u(z)  \big| \leq c \big| D_{y}^{m+1}\hat{u} \big( \Psi_{z}(w) \big)- D_{y}^{m+1}\hat{u} \big( \Psi_{z}(z) \big) \big|.
\end{equation}
By the chain rule, for any $p' \geq 0$ with $|p'| \leq m+1$, we have from Lemma \ref{HODS200} that
\begin{equation*}\begin{aligned}\label{} 
D_{x^{1}}^{m+1-|p'|} D_{x'}^{p'}u (w) 
& = \sum_{0 \leq q' \leq p'} 
\Big( \begin{array}{c} p' \\ q' \end{array} \Big) 
\bigg( \frac{ \partial y^{1} }{ \partial x^{2} } \bigg)^{ p^{2} - q^{2}} 
\bigg( \frac{ \partial y^{2} }{ \partial x^{2} } \bigg)^{ q^{2} }
\cdots 
\bigg( \frac{ \partial y^{1} }{ \partial x^{n} } \bigg)^{ p^{n} - q^{n} }
\bigg( \frac{ \partial y^{n} }{ \partial x^{n} } \bigg)^{ q^{n} } \\
& \qquad \qquad \times D_{y^{1}}^{m+1-|p'|}D_{y}^{(|p'|-|q'|,q')} \hat{u} \big( \Psi(w) \big) \\
& = \sum_{0 \leq q' \leq p'}  
\Big( \begin{array}{c} p' \\ q' \end{array} \Big)
(-1)^{|p'-q'|} \pi^{p'-q'}(z) D_{y}^{(m+1-|q'|,q')}\hat{u} \big( \Psi(w) \big).
\end{aligned}\end{equation*}
It follows that  for any $p' \geq 0$ with $|p'| \leq m+1$,
\begin{equation*}
\Big|  D_{x^{1}}^{m+1-|p'|} D_{x'}^{p'}u (w) - D_{x^{1}}^{m+1-|p'|} D_{x'}^{p'}u (z) \Big|
\leq c \Big| D_{y}^{m+1}\hat{u} \big( \Psi(w) \big) - D_{y}^{m+1}\hat{u} \big( \Psi(z) \big) \Big|.
\end{equation*}
So the claim \eqref{HOD830} follows. Since $z, w \in Q_{5}^{k}$ $(k \in K)$ were arbitrary chosen, the lemma holds from \eqref{HOD830} and Lemma \ref{HODS700}.
\end{proof}

\begin{lem}\label{HODS900}
For any $z,w \in Q_{5}^{k}$ $(k \in K)$, we have that
\begin{equation*}\begin{aligned}
\big| D^{m+1}u(w) - D^{m+1}u(z)\big| 
& \leq c \sum_{|p'| = m} \left[  \big| U^{p'}(w) - U^{p'}(z)\big| 
+  |w-z|^{\gamma} \big| U^{p'}(z) \big| \right]  \\
& \quad + c |w-z|^{\gamma} \sum_{k \in K} \left[  \| u \|_{C^{m,\gamma} (Q_{7}^{k}) }  +  \| F \|_{C^{m,\gamma} (Q_{7}^{k}) }  \right].
\end{aligned}\end{equation*}
\end{lem}

\begin{proof}
With Lemma \ref{HODS800}, the lemma holds from \eqref{HOD294} and \eqref{HOD296} in Lemma \ref{HODS290} 
\end{proof}

We now obtain Proposition \ref{prop_compare}.

\begin{proof}[Proof of Proposition \ref{prop_compare}]
Proposition \ref{prop_compare} holds by Lemma \ref{HODS550} and Lemma \ref{HODS900}.
\end{proof}

\newpage

\section{Proof of the main theorem}

\subsection{\texorpdfstring{$L^{2}\text{-estimate of } D^{m+1}u$}{Uinfty}}

We obtain $L^{2}$-type estimate of $D^{m+1}u$. We remark that we have the term $ r^{\gamma}  \| D^{m+1}u \|_{L^{\infty}(Q_{2r}(z) ) }^{2}$ in the right-hand side of  Lemma \ref{CTFS700} and Lemma \ref{PMTS200}, but this term $ r^{\gamma}  \| D^{m+1}u \|_{L^{\infty}(Q_{2r}(z) ) }^{2}$ will be removed by using the term $r^{\gamma}$ later in Lemma \ref{PMTS1000} and Lemma \ref{PMTS1100}.

To simplify the computation, we set 
\begin{equation}\label{PMT130} 
E = \sum_{ k \in K }  \left[ \| u \|_{C^{m,\gamma} (Q_{7}^{k}) }  +  \| F \|_{C^{m,\gamma} (Q_{7}^{k}) } \right].
\end{equation}
Then we find from \eqref{HOD065} and \eqref{HOD090} that
\begin{equation}\label{PMT140} 
\sum_{k \in K} \sum_{ |p'| = m } 
\left[ \left\| G^{p'} \right\|_{C^{\gamma}(Q_{5}^{k})} + \left \| \nu^{p'} \right \|_{C^{\gamma}(Q_{5}^{k})} \right]
\leq E.
\end{equation}
To compare $D^{m+1}u$ and $U^{p'}$, we will use \eqref{PMT150} and \eqref{PMT160}. To perturb the equation, we use \eqref{PMT170}. The test function will be handled by \eqref{PMT180}.

\sskip

By Lemma \ref{HODS550} and \eqref{PMT130}, 
\begin{equation}\label{PMT150} 
| D^{m+1}u(z) |
\leq c \left[ E + \sum_{|p'| = m}  \left| U^{p'}(z) \right|   \right]
\leq c \left[ E + | D^{m+1}u(z) | \right]
\quad (z \in Q_{5}).
\end{equation}
Also by Lemma \ref{HODS900} and \eqref{PMT130}, for any $ z, w \in Q_{5}^{k}$ with $k \in K$,
\begin{equation}\label{PMT160} 
\left| D^{m+1}u(w) - D^{m+1}u(z) \right| 
\leq c \left[ |w-z|^{\gamma} E +  \sum_{|p'| = m} \big| U^{p'}(w) - U^{p'}(z) \big| \right].
\end{equation}

For the simplicity, for any  $p' \geq 0'$ with $|p'|=m$, set $\tilde{u}^{p'} : Q_{5} \to \br^{Nn} $as
\begin{equation}\label{PMT165}
\tilde{u}^{p'} = \left( \begin{array}{ccc}
\tilde{u}_{1}^{1,p'} & \cdots & \tilde{u}_{n}^{1,p'} \\
\vdots & \ddots &  \vdots \\ 
\tilde{u}_{1}^{N,p'} & \cdots & \tilde{u}_{n}^{N,p'} 
\end{array}\right)
\end{equation}
as in \eqref{HOD075}  so that
\begin{equation*}\label{}
\tilde{u}_{\beta}^{j,p'} = \sum_{0' \leq q' \leq p' } \Big( \begin{array}{c} p' \\ q' \end{array} \Big)\pi^{q'} D^{(|q'|,p'-q')} D_{\beta}u^{j}
\quad \text{in} \quad Q_{5}.
\end{equation*}
Then by Lemma \ref{PETS900}, for any $\eta \in C_{c}^{\infty}(Q_{r}(z))$ with $Q_{r}(z) \subset Q_{5}$, 
\begin{equation}\begin{aligned}\label{PMT170}
& \left| \int_{Q_{r}(z) } \sum_{1 \leq \alpha \leq n} \left[ \sum_{1 \leq j \leq N} \sum_{1 \leq \beta \leq n} A^{\alpha \beta}_{ij,k}   \tilde{u}_{\beta}^{j,k} + \nu_{\alpha}^{i} \right] D_{\alpha} \eta^{i} - G_{1}^{i} D_{1} \eta^{i}  \, dx \right| \\
& \quad \leq c  r^{n+\frac{1}{4}} \left[   \| D^{m+1}u \|_{L^{\infty}(Q_{r}(z))}^{2}   +  \sum_{k \in K}  \left[  \| u \|_{C^{m,\gamma}( Q_{7}^{k})}^{2} +  \| F \|_{C^{m,\gamma}( Q_{7}^{k})}^{2} \right]  \right] \\
& \qquad + c  r^{\frac{1}{4}} \int_{ Q_{r}(z) } \epsilon |D\eta|^{2}  + \epsilon^{-2} |\eta|^{2}  \, dx,
\end{aligned}\end{equation}
In view of  \eqref{HOD040}, \eqref{HOD080} tand \eqref{HOD297}, we have that for any $p' \geq 0'$ with $|p'|=m$, 
\begin{equation*}\label{}
\left| \tilde{u}^{p'} \right| \leq c \left[  \left| U^{p'} \right|  + \left| \nu^{p'} \right| + \left| G^{p'} \right| \right]
\quad \text{and} \quad
\left| U^{p'} \right|  \leq  c \left[ \left| \tilde{u}^{p'} \right| + \left| \nu^{p'} \right| + \left| G^{p'} \right|  \right]
\quad \text{in} \quad Q_{5}.
\end{equation*}
So  it follows from Lemma \ref{HODS550} and \eqref{PMT140} that
\begin{equation}\label{PMT175}
| D^{m+1}u(z) |
\leq c \left[ E + \sum_{|p'| = m}  \big| \tilde{u}^{p'}(z) \big|   \right]
\leq c \left[ E + | D^{m+1}u(z) | \right]
\quad (z \in Q_{5}).
\end{equation}

For $\eta^{p'}  : Q_{r}(z) \to \br^{N}$ in \eqref{CTF650} with $Q_{2r}(z) \subset Q_{5}$,  Lemma \ref{CTFS700} implies that
\begin{equation}\label{PMT180} 
\int_{ Q_{r}(z) } \bigg| D_{\beta}\eta^{p'} - \tilde{u}_{\beta}^{p'} + [1-\delta_{\beta}] G_{\beta}^{p'} \bigg|^{2} \, dx 
\leq cr^{n+2\gamma}  \left[   \| D^{m+1}u \|_{L^{\infty}(Q_{2r}(z))}^{2} + E^{2} \right],
\end{equation}
for any $1 \leq \beta \leq n$. Here, with \eqref{HOD050} and \eqref{HOD060}, we obtain from  \eqref{GS180}, Lemma \ref{PETS900} and Lemma \ref{CTFS800} that
\begin{equation}\label{PMT190}
|\eta|  \leq c \sum_{k \in K} \Big[ \| u \|_{C^{m,\gamma}(Q_{7}^{k})} 
+  \| F \|_{C^{m,\gamma}(Q_{7}^{k})} \Big]
\leq cE 
\quad \text{in} \quad  Q_{r}(z).
\end{equation}

\begin{lem}\label{PMTS200}
For any $\phi \in C_{c}^{\infty}(Q_{r}(z))$ with $0 \leq |\phi| \leq 1$, we have that
\begin{equation*}\label{}
\int_{Q_{r}(z) } \left| D^{m+1}u \right|^{2} \phi^{2}   \, dx  
\leq c \bigg[ r^{n+2\gamma} \| D^{m+1}u \|_{L^{\infty}(Q_{2r}(z))}^{2} +  \int_{ Q_{r}(z) } \big( 1+ |D\phi|\big) ^{2} E^{2} \, dx \bigg]
\end{equation*}
for any $Q_{2r}(z) \subset Q_{5}$.
\end{lem}

\begin{proof}
Fix $p' \geq 0'$ with $|p'|=m$. For the simplicity, let $\tilde{u} = \tilde{u}^{p'}$ (in \eqref{PMT165}), $\eta = \eta^{p'}$ (in \eqref{CTF650}) and omit $\sum_{ 1 \leq i,j \leq N } \sum_{ 1 \leq \alpha, \beta \leq n} $ in the integral. Then by \eqref{PMT140} and \eqref{PMT180},
\begin{equation}\label{PMT230}
\int_{ Q_{r}(z) } \left| D_{\beta}\eta - \tilde{u}_{\beta} \right|^{2} \, dx 
\leq cr^{n}  \left[   r^{2\gamma} \| D^{m+1}u \|_{L^{\infty}(Q_{2r}(z))}^{2} + E^{2} \right].
\end{equation}
For \eqref{PMT170},  we choose $ \phi^{2} \eta$ instead of $\phi$. Since $r \in (0,1]$ and $0 \leq |\phi| \leq 1$, we obtain from Young's inequality and \eqref{PMT190} that for any $\varepsilon \in (0,1]$,
\begin{equation*}\begin{aligned}\label{}
& \left | \int_{Q_{r}(z) } A^{\alpha \beta}_{ij} \tilde{u}_{\beta}^{j} D_{\alpha} \left( [\phi^{i}]^{2} \eta^{i} \right)  \, dx \right| \\
& \quad \leq c \int_{Q_{r}(z) } \left[ \left| G_{1} \right| + \left| \nu \right| \right]   \left| D\left( \phi^{2} \eta \right) \right|   \, dx \\
& \qquad  + c  r^{2\gamma} \int_{Q_{r}(z)} \left[  \| D^{m+1}u \|_{L^{\infty}(Q_{r}(z))}^{2}   + E^{2} + \varepsilon \left| D\left( \phi^{2} \eta \right) \right|^{2} + \varepsilon^{-2} \left| \phi^{2} \eta \right|^{2} \right]  \, dx.
\end{aligned}\end{equation*}
By \eqref{PMT140} and \eqref{PMT190}, $|\nu| + |G_{1}| + |\eta| \leq cE$ in $Q_{r}(z)$. So by Young's inequality,
\begin{equation*}\begin{aligned}\label{}
& \left | \int_{Q_{r}(z) } A^{\alpha \beta}_{ij} \tilde{u}_{\beta}^{j} D_{\alpha} \left( \phi^{2} \eta^{i} \right)  \, dx \right| \\
& \quad \leq c \int_{Q_{r}(z) } \left[ \left| G_{1} \right| + \left| \nu \right| \right]   \left| D\left( \phi^{2} \eta \right) \right|   \, dx \\
& \quad \leq c \int_{Q_{r}(z) } r^{2\gamma} \| D^{m+1}u \|_{L^{\infty}(Q_{2r}(z))}^{2} +  \varepsilon  r^{2\gamma}|D\eta|^{2} \phi^{2}  + c(\varepsilon)\left( 1+ |D\phi|^{2} \right) E^{2}\, dx,
\end{aligned}\end{equation*}
for any $\varepsilon \in (0,1]$. We estimate the left-hand and right-hand side.

We first estimate the right-hand side. We obtain from Young's inequality that
\begin{equation*}\begin{aligned}\label{}
\int_{Q_{r}(z) } |D\eta|^{2} \phi^{2} \, dx 
\leq 2 \int_{ Q_{r}(z) } \sum_{1 \leq \beta \leq n} \left| D_{\beta}\eta - \tilde{u}_{\beta} \right|^{2} \phi^{2} \, dx 
+ 2 \int_{ Q_{r}(z) } \left| \tilde{u}  \right|^{2} \phi^{2} \, dx.
\end{aligned}\end{equation*}
From \eqref{PMT140}, we have that $|G| \leq cE$ in $Q_{r}(z)$. So by \eqref{PMT180},
\begin{equation}\begin{aligned}\label{PMT250}
\int_{Q_{r}(z) } |D\eta|^{2} \phi^{2} \, dx
\leq c \int_{ Q_{r}(z) }   r^{2\gamma}  \| D^{m+1}u \|_{L^{\infty}(Q_{2r}(z))}^{2}   + |\tilde{u}|^{2} + E^{2} \, dx.
\end{aligned}\end{equation}
By \eqref{PMT175}, $|\tilde{u}| \leq c \left[ |D^{m+1}u| + E \right]$ in $Q_{r}(z)$. So the right-hand side is estimated as
\begin{equation*}\begin{aligned}\label{}
& \int_{Q_{r}(z) } r^{2\gamma} \| D^{m+1}u \|_{L^{\infty}(Q_{2r}(z))}^{2} +  \varepsilon  r^{2\gamma} |D\eta|^{2} \phi^{2}  + c(\varepsilon)\left( 1+ |D\phi|^{2} \right) E^{2}\, dx \\
& \quad \leq c \int_{Q_{r}(z) } r^{2\gamma} \| D^{m+1}u \|_{L^{\infty}(Q_{2r}(z))}^{2} +  \varepsilon |\tilde{u}|^{2}  \phi^{2}  + c(\varepsilon)\left( 1+ |D\phi|^{2} \right) E^{2}\, dx 
\end{aligned}\end{equation*}

We now estimate the left-hand side. By that $D \big( \phi^{2}\eta \big) = 2 \phi \eta D\phi  + \phi^{2} D \eta $,
 \begin{equation*}\begin{aligned}\label{}
& \left | \int_{Q_{r}(z) } A^{\alpha \beta}_{ij} \tilde{u}_{\beta}^{j} D_{\alpha} \left( [\phi^{i}]^{2} \eta^{i} \right)  \, dx \right| \\
& \quad \geq \left | \int_{Q_{r}(z) } A^{\alpha \beta}_{ij} \tilde{u}_{\beta}^{j} \phi \left(D_{\alpha}  \eta^{i} \right) \phi  \, dx \right| 
 - \left | \int_{Q_{r}(z) } A^{\alpha \beta}_{ij} \tilde{u}_{\beta}^{j} \left( 2 \phi D_{\alpha} \phi \,  \eta^{i} \right)  \, dx \right|.
\end{aligned}\end{equation*}
By Young's inequality, we obtain that
\begin{equation*}\begin{aligned}\label{}
& \left | \int_{Q_{r}(z) } A^{\alpha \beta}_{ij} \tilde{u}_{\beta}^{j} \phi \left(D_{\alpha}  \eta^{i} \right) \phi  \, dx \right| \\
& \quad \geq   \left | \int_{Q_{r}(z) }  A^{\alpha \beta}_{ij} \tilde{u}_{\beta}^{j} \phi \, \tilde{u}_{\alpha} ^{i} \phi  \, dx \right| 
-  \left | \int_{Q_{r}(z) } A^{\alpha \beta}_{ij} \tilde{u}_{\beta}^{j} \phi \left( D_{\alpha}  \eta^{i} -  \tilde{u}_{\alpha}^{i} \right) \phi  \, dx \right| \\
& \quad \geq \frac{\lambda}{2} \int_{Q_{r}(z)} |\tilde{u}|^{2} \phi^{2} \, dx 
- c \int_{Q_{r}(z)} \left| D_{i} \eta - \tilde{u}_{j} \right| \phi^{2},
\end{aligned}\end{equation*}
and
\begin{equation*}\begin{aligned}
& \left | \int_{Q_{r}(z) } A^{\alpha \beta}_{ij} \tilde{u}_{\beta}^{j} \left( 2 \phi D_{\alpha} \phi \,  \eta^{i} \right)  \, dx \right|
\leq \frac{\lambda}{4} \int_{Q_{r}(z)} |\tilde{u}|^{2} \phi^{2} \, dx
+ c \int_{Q_{r}(z)} |\eta|^{2} |D\phi|^{2} \, dx.
\end{aligned}\end{equation*}
By combining the above three estimate, it follows from \eqref{PMT190} and \eqref{PMT230} that
\begin{equation*}\begin{aligned}\label{}
&\left | \int_{Q_{r}(z) } A^{\alpha \beta}_{ij} \tilde{u}_{\beta}^{j} D_{\alpha} \left( [\phi^{i}]^{2} \eta^{i} \right)  \, dx \right| \\
& \quad \geq \frac{\lambda}{4} \int_{Q_{r}(z)} |\tilde{u}|^{2} \phi^{2} \, dx 
- c \int_{ Q_{r}(z) } r^{2\gamma} \| D^{m+1}u \|_{L^{\infty}(Q_{2r}(z))}^{2}  + \left( 1+|D\phi|^{2} \right) E^{2} \, dx,
\end{aligned}\end{equation*}
and the estimate for the left-hand side is obtained. So by combining the estimates for left-hand side and right-hand side, we obtain that
\begin{equation*}\label{}
\frac{\lambda}{4} \int_{Q_{r}(z)} |\tilde{u}|^{2} \phi^{2} \, dx 
\leq c_{1} \int_{Q_{r}(z) } r^{\gamma} \| D^{m+1}u \|_{L^{\infty}(Q_{2r}(z))}^{2} +  \varepsilon |\tilde{u}|^{2} \phi^{2}  + c(\varepsilon) \left( 1+ |D\phi|^{2} \right) E^{2}\, dx,
\end{equation*}
for some universal constant $c_{1} \geq 1$. Let  $\varepsilon \in (0,1]$ be a small universal constant satisfying that $c_{1} \epsilon \leq \lambda/8$. Then we have that
\begin{equation*}\label{}
\int_{Q_{r}(z)} |\tilde{u}|^{2} \phi^{2}  dx
\leq c \int_{Q_{r}(z) } r^{2\gamma} \| D^{m+1}u \|_{L^{\infty}(Q_{2r}(z))}^{2} 
+ \left( 1+ |D\phi|^{2} \right) E^{2} dx.
\end{equation*}
From \eqref{PMT175}, we have that $ \left| D^{m+1}u \right| \leq c \left[ |\tilde{u}| + E \right]$ in $Q_{r}(z)$. So the lemma follows.
\end{proof}

\subsection{Comparison estimate}

Fix $Q_{2r}(z) \subset Q_{5}$ and $p' \geq 0'$ with $|p'|=m$. Suppose that $v \in W^{1,2}(Q_{r}(z))$ is the weak solution of 
\begin{equation}\label{PMT310}
\left\{\begin{array}{rcll}
D_{\alpha} \left[ A^{\alpha \beta}_{ij}D_{\beta}v^{j} \right] & = & D_{\alpha} \left[  A^{\alpha \beta}_{ij} [1-\delta_{\beta1}] G_{\beta}^{j,p'} + \delta_{\alpha1} G_{1}^{i,p'}  - \nu_{\alpha}^{i,p'}  \right] & \text{in } Q_{r}(z), \\
v & = & \eta^{p'} & \text{on } \partial Q_{r}(z),
\end{array}\right.
\end{equation}

\begin{lem}\label{PMTS300}
There exists a small universal constant  $R_{1} \in (0,1]$ such that  if $r \in (0,R_{1}]$ then for any $p' \geq 0'$ with $|p'|=m$ and $1 \leq \beta \leq n$, we have that 
\begin{equation*}\begin{aligned}\label{}
& \int_{ Q_{r}(z) } \left|  \sum_{0' \leq q' \leq p' } \Big( \begin{array}{c} p' \\ q' \end{array} \Big) \pi^{q'} D^{ (|q'|,p'-q') }  D_{\beta}u  - D_{\beta}v - [1-\delta_{\beta 1}] G_{\beta}^{p'}  \right|^{2}\, dx \\
& \quad \leq  cr^{n+2\gamma}   \left[  \| D^{m+1}u \|_{L^{\infty}(Q_{r}(z))}^{2}  +  E^{2} \right].
\end{aligned}\end{equation*}
\end{lem}
\begin{proof}
Fix $p' \geq 0'$ with $|p'|=m$. Recall $\tilde{u} = \tilde{u}^{p'}$ in \eqref{PMT165}. We find from \eqref{PMT170} that
\begin{equation*}\begin{aligned}\label{}
& \left| \int_{Q_{r}(z) } A^{\alpha \beta}_{ij} \tilde{u}_{\beta}^{j} D_{\alpha} \phi^{i} +  G_{1}^{i,p'}  D_{1} \phi^{i}  + \nu_{\alpha}^{i,p'}  D_{\alpha}\phi^{i} \, dx \right| \\
& \quad \leq c  r^{2\gamma}  \int_{ Q_{r}(z) }  \| D^{m+1}u \|_{L^{\infty}(Q_{r}(z))}^{2} +  E^{2} +  |D\phi|^{2}   \, dx,
\end{aligned}\end{equation*}
for any $\phi \in C_{c}^{\infty}(Q_{r}(z))$. From \eqref{PMT310}, we find that
\begin{equation*}\begin{aligned}\label{}
\int_{ Q_{r}(z) }  \left[ A^{\alpha \beta}_{ij} \left( D_{\beta}v^{j} - [1-\delta_{\beta 1}] G_{\beta}^{j,p'} \right) + \delta_{\alpha 1} G_{\alpha}^{p'} + \nu_{\alpha}^{i,p'}  \right] D_{\alpha} \phi ^{i} \, dx = 0,
\end{aligned}\end{equation*}
for any $\phi \in C_{c}^{\infty}(Q_{r}(z))$. So it follows from \eqref{PMT130} and \eqref{PMT150} that
\begin{equation*}\begin{aligned}
& \left| \int_{ Q_{r}(z) } 
A^{\alpha \beta}_{ij} \left[ \tilde{u}_{\beta}^{j} - D_{\beta}v^{j} + [1-\delta_{\beta 1}]G_{\beta}^{j,p'} \right]  D_{\alpha} \phi^{i} \, dx \right|  \\
& \quad \leq c  r^{2\gamma}  \int_{ Q_{r}(z) }  \| D^{m+1}u \|_{L^{\infty}(Q_{r}(z))}^{2} +  E^{2} +  |D\phi|^{2}   \, dx,
\end{aligned}\end{equation*}
for any $\phi \in C_{c}^{\infty}(Q_{r}(z))$. From  \eqref{PMT180},  we obtain that
\begin{equation*}\begin{aligned}\label{} 
\int_{ Q_{r}(z) } \left| D_{\alpha}\eta^{i} - \tilde{u}_{\alpha}^{i} + [1-\delta_{\alpha 1}] G_{\alpha}^{i} \right|^{2} \, dx 
\leq cr^{n+2\gamma}  \left[   \sum_{|p'|=m} \big\| U^{p'} \big \|_{L^{\infty}( Q_{2r}(z) )}^{2} + E^{2} \right].
\end{aligned}\end{equation*}
Take $\phi = \eta - v$. Then by using Young's inequality and the elliptic condition,
\begin{equation*}\begin{aligned}
& \int_{ Q_{r}(z) } \left| \tilde{u}_{\beta} - D_{\beta}v -[1-\delta_{\beta 1}] G_{\beta}^{p'}  \right|^{2}\, dx  \\
& \quad \leq  cr^{n+2\gamma}   \left[   \| D^{m+1}u \|_{L^{\infty}(Q_{2r}(z))} + E^{2} \right] 
 + cr^{2\gamma} \int_{ Q_{r}(z) } \left| \tilde{u}_{\beta} -  [1-\delta_{\beta 1}] G_{\beta}^{p'} - D_{\beta}v   \right|^{2}\, dx.
\end{aligned}\end{equation*}
So with \eqref{PMT165},  for the lemma holds for sufficiently small $R_{1}>0$.
\end{proof}

\subsection{Excess decay estimate}

We obtain the excess decay estimate of $U^{p'}$ in Lemma \ref{PMTS700}. Then by Campanato type embedding, we show H\"{o}lder continuity of $U^{p'}$. In the proof of Lemma \ref{PMTS700}, we use the following variation of the technical lemma in \cite[Lemma 3.4]{HQLF1}.

\begin{lem}\label{PMTS400}
Let $\phi(t)$ be a nonnegative  function on $[0,R]$. If
\begin{equation*}
\phi(\rho) \leq  A\bigg( \frac{\rho}{r} \bigg)^{\alpha} \phi(r) + B  r^{\beta}  \left( \frac{r}{\rho} \right)^{n}
\end{equation*}
holds for any $0 < \rho \leq r \leq R$ with $A,B,\alpha,\beta$ nonnegative constants and $\beta < \alpha$. Then for any $\gamma \in (\beta,\alpha)$, there exist  a positive constant $c$ depending on $n,A,\alpha,\beta,\gamma$ such that 
\begin{equation*}
\phi(\rho) \leq  c \bigg[ \bigg( \frac{\rho}{r} \bigg)^{\gamma} \phi(r) + B  \rho^{\beta} \bigg],
\end{equation*}
for all $0<\rho \leq r \leq R$.
\end{lem}

We have the following excess decay estimate for the reference equations from \cite{KYSP1}.

\begin{lem}\label{PMTS500}
For $Q_{2r}(z) \subset Q_{5}$, suppose that $\tilde{F} \in C^{\gamma} \big( Q_{r}^{k}(z), \br^{n} \big)$ for any $k \in K$. Then for the weak solution $v$ of
\begin{equation*}\label{}
D_{\alpha} \left[ A^{\alpha \beta}_{ij} D_{\beta}v^{j} \right] = D_{\alpha} \tilde{F}_{\alpha}^{i}
\text{ in } Q_{r}(z),
\end{equation*}
and  $V : Q_{r}(z) \to \br^{Nn}$defined as
\begin{equation*}\label{} 
V= \left( \sum_{1 \leq i \leq N} \sum_{1 \leq \alpha \leq n} \pi_{\alpha} \left[ \sum_{ 1 \leq j \leq N} \sum_{1 \leq \beta \leq n}  A_{ij}^{\alpha \beta} D_{\beta} v^{i} - \tilde{F}_{\alpha}^{i} \right], D_{x'}v^{i} + \pi' \, D_{1} v^{i} \right),
\end{equation*}
we have that for any $0 < \rho \leq r$,
\begin{equation*}\begin{aligned}
& \frac{1}{\rho^{2\gamma}}  \mint_{ Q_{\rho}(z) } \left| V  - (V)_{ Q_{\rho}(z) } \right|^{2} \, dx \\
& \quad \leq  c \left( \frac{1}{ r^{ 2\gamma }}
\mint_{ Q_{r}(z) } \left| V - (V)_{ Q_{r}(z) } \right|^{2} \, dx + \mint_{ Q_{r}(z) } |V|^{2} \, dx 
+  \sup_{ k \in K }  \big\| \tilde{F} \big\|_{C^{\gamma}(Q_{r}^{k}(z))}^{2}  \right) .
\end{aligned}\end{equation*}
\end{lem}

\begin{proof}
Take $R=3$ in Theorem 1.7 in \cite{KYSP1}, then the lemma follows.
\end{proof}

\begin{lem}\label{PMTS700}
For any $p' \geq 0'$ with $|p'|=m$, if  $Q_{2r}(z) \subset Q_{5}$ and $0 < \rho \leq 2r$ then 
\begin{equation*}\begin{aligned}
& \frac{1}{\rho^{2\gamma}}  \mint_{ Q_{\rho}(z) } \left| U^{p'} - \big( U^{p'}  \big)_{Q_{\rho}(z)} \right|^{2}  \, dx \\
& \quad \leq  c \left[
\frac{1}{ r^{ 2\gamma } } \mint_{ Q_{2r}(z) } \left| U^{p'}  - \big( U^{p'} \big)_{Q_{2r}(z)} \right|^{2} \, dx +  \| D^{m+1}u \|_{L^{\infty}(Q_{2r}(z))}^{2} + E^{2}  \right].
\end{aligned}\end{equation*}
\end{lem}

\begin{proof}
We only need to prove the case that $0 < r \leq R_{1}$ where $R_{1} \in (0,1]$ is a universal constant chosen in Lemma \ref{PMTS300}. Otherwise, we have that $0 < \rho   \leq R_{1} < r (\leq 7)$ or $R_{1} < \rho \leq r ( \leq 7)$ in which case the lemma can be obtained by using the case that $0 < r \leq R_{1}$  and that $R_{1} \in (0,1]$ is a universal constant.

We assume that $0 < \rho \leq r$, otherwise we have that $r < \rho \leq 2r$ and one can easily prove the lemma. Fix  $p' \geq 0'$ with $|p'|=m$ and $0 < \rho \leq r$. Let $v^{p'} \in W^{1,2}(Q_{2r}(z))$ be the weak solution of 
\begin{equation*}\label{}
\left\{\begin{array}{rcll}
D_{\alpha} \left[ A^{\alpha \beta}_{ij}D_{\beta}v^{j} \right] & = & D_{\alpha} \left[  -A^{\alpha \beta}_{ij} [1-\delta_{\alpha 1}] G_{\beta}^{j,p'} + \delta_{\alpha1} G_{1}^{i,p'}  - \nu_{\alpha}^{i,p'}  \right] & \text{ in } Q_{r}(z), \\
v & = & \eta^{p'} & \text{ on } \partial Q_{r}(z),
\end{array}\right.
\end{equation*}
where $\eta$ is defined in \eqref{CTF650}. Here,  we have from \eqref{PMT140} that
\begin{equation}\label{PMT275}
\tilde{F}_{\alpha}^{i} : =  - A^{\alpha \beta}_{ij} [1-\delta_{\beta 1}] G_{\beta}^{j,p'} + \delta_{\alpha1} G_{1}^{i,p'}  - \nu_{\alpha}^{i,p'}  
\in C^{\gamma} \left( Q_{r}^{k}(z), \br^{n} \right)
\end{equation}
with the estimate
\begin{equation}\label{PMT730}
\big\| \tilde{F}_{\alpha}^{i} \big\|_{C^{\gamma}(Q_{5}^{k})}  
= \left\|  - A^{\alpha \beta}_{ij} [1-\delta_{\beta 1}] G_{\beta}^{j,p'} + \delta_{\alpha1} G_{1}^{i,p'}  - \nu_{\alpha}^{i,p'}   \right\|_{C^{\gamma}(Q_{5}^{k})}
\leq E.
\end{equation}
for any $k\in K$ and $i = 1,2,\cdots,n$. So by Lemma \ref{PMTS500} and \eqref{PMT130},
\begin{equation*}\begin{aligned}
& \frac{1}{\rho^{2\gamma}} \mint_{ Q_{\rho}(z) } \left| V^{p'}  - (V^{p'})_{ Q_{\rho}(z) } \right|^{2} \, dx \\
& \quad \leq  c \left[ \frac{1}{ r^{ 2\gamma }}
\mint_{ Q_{r}(z) } \left| V^{p'} - (V^{p'})_{ Q_{r}(z) } \right|^{2} \, dx + \mint_{ Q_{r}(z) } \left| V^{p'} \right|^{2} \, dx 
+  \sup_{ k \in K }  \big\| \tilde{F} \big\|_{C^{\gamma}(Q_{r}^{k}(z))}^{2}  \right] .
\end{aligned}\end{equation*}
where
\begin{equation*}\label{} 
V^{i,p'}= \left( \sum_{1 \leq \alpha \leq n} \pi_{\alpha} \left[ \sum_{ 1 \leq j \leq N} \sum_{1 \leq \beta \leq n}  A_{ij}^{\alpha \beta} D_{\beta} v^{j} - \tilde{F}_{\alpha}^{i} \right], D_{x'}v^{i} + \pi' \, D_{1} v^{i} \right),
\end{equation*}
in $Q_{r}(z)$. We compare $U^{p'}$ and $V^{p'}$ as follows. By \eqref{HOD080} and \eqref{HOD085},
\begin{equation*}\label{} 
U_{1}^{i,p'}
= G_{1}^{i,p'} + \sum_{1 \leq \alpha \leq n } \pi_{\alpha} \left[ \nu_{\alpha}^{i,p'} + \sum_{1 \leq j \leq N } \sum_{1 \leq \beta \leq n} A^{\alpha \beta}_{ij}\tilde{u}_{\beta}^{j,p'}  \right]
\quad \text{in} \quad Q_{5},
\end{equation*}
and
\begin{equation*}\label{}
U_{\beta}^{j,p'} = G_{\beta}^{j,p'} +  \tilde{u}_{\beta}^{j,p'} + \pi_{\beta} \tilde{u}_{1}^{j,p'}
\quad \text{in} \quad Q_{5}, 
\end{equation*}
for any $2 \leq \beta \leq  n$ and $1 \leq j \leq N$.  Since $\pi_{1}=-1$, we obtain from \eqref{PMT275} that
\begin{equation*}\begin{aligned}\label{}
U_{1}^{i,p'} - V_{1}^{i,p'}
& = G_{1}^{i,p'} + \sum_{1 \leq \alpha  \leq n} \pi_{\alpha} \left[  \nu_{\alpha}^{i,p'} +  \sum_{1 \leq j \leq N } \sum_{1 \leq \beta \leq n} A^{\alpha \beta}_{ij} \left[ \tilde{u}_{\beta}^{j,p'} - D_{\beta}v^{j} \right]+ \tilde{F}_{\alpha}^{i}  \right]  \\
& =  \sum_{1 \leq \alpha \leq n} \pi_{\alpha} \left[   \sum_{1 \leq j \leq N } \sum_{1 \leq \beta \leq n} A^{\alpha \beta}_{ij} \left[ \tilde{u}_{\beta}^{j,p'} - D_{\beta}v^{j}  - [1-\delta_{\beta1}] G_{\beta}^{j,p'} \right] \right] 
\end{aligned}\end{equation*}
and
\begin{equation*}\begin{aligned}\label{}
U_{\beta}^{i,p'} - V_{\beta}^{i,p'}
& =   G_{\beta}^{i,p'} +  \tilde{u}_{\beta}^{i,p'} + \pi_{\beta} \tilde{u}_{1}^{i,p'}  - D_{\beta}v^{i} - \pi_{\beta} D_{1}v^{i} \\
& = \tilde{u}_{\beta}^{i,p'} - D_{\beta}v^{i,p'}  - [1-\delta_{\beta 1}]  G_{\beta}^{i,p'}  + \pi_{\beta} \left( \tilde{u}_{1}^{i,p'} - D_{1}v^{i,p'} \right)
\end{aligned}\end{equation*}
for any $2 \leq \beta \leq n$ and $1 \leq i \leq N$.  Since $0< r \leq R_{1}$, we have from Lemma \ref{PMTS300} that 
\begin{equation*}\label{}
\int_{ Q_{r}(z) } \left|  D_{\alpha}v  - \tilde{u}_{\alpha}^{p'} + [1-\delta_{\alpha 1}] G_{\alpha}^{p'}  \right|^{2} \, dx 
 \leq  cr^{n+2\gamma}   \left[   \| D^{m+1}u \|_{L^{\infty}(Q_{2r}(z))}^{2} +  E^{2} \right],
\end{equation*}
for any $1 \leq \alpha \leq n$. So we obtain that 
\begin{equation*}\label{PMT780} 
\mint_{Q_{r}(z)} \left| U^{p'} - V^{p'} \right|^{2} \, dx 
\leq  cr^{n+2\gamma}  \left[   \| D^{m+1}u \|_{L^{\infty}(Q_{2r}(z))}^{2} +  E^{2} \right].
\end{equation*}
Thus
\begin{equation*}\begin{aligned}
& \mint_{ Q_{\rho}(z) } \left| U^{p'}  - (U^{p'})_{ Q_{\rho}(z) } \right|^{2} \, dx \\
& \quad \leq  c \left( \frac{\rho}{r} \right)^{2\gamma}
\mint_{ Q_{r}(z) } \left| U^{p'} - (U^{p'})_{ Q_{r}(z) } \right|^{2}  dx \\
& \qquad + cr^{2\gamma} \left( \frac{r}{\rho} \right)^{n} \left( \mint_{ Q_{r}(z) } \left| U^{p'} \right|^{2} dx +  \| D^{m+1}u \|_{L^{\infty}(Q_{2r}(z))}^{2} +  E^{2}  \right).
\end{aligned}\end{equation*}
By \eqref{PMT150} and that $\mint_{ Q_{r}(z) } \left| U^{p'} - (U^{p'})_{ Q_{r}(z) } \right|^{2}  dx \leq c \mint_{ Q_{2r}(z) } \left| U^{p'} - (U^{p'})_{ Q_{2r}(z) } \right|^{2}  dx $, 
\begin{equation*}\begin{aligned}
\mint_{ Q_{\rho}(z) } \left| U^{p'}  - (U^{p'})_{ Q_{\rho}(z) } \right|^{2} \, dx 
& \leq  c \left( \frac{\rho}{r} \right)^{2\gamma}
\mint_{ Q_{2r}(z) } \left| U^{p'} - (U^{p'})_{ Q_{2r}(z) } \right|^{2}  dx \\
& \quad + cr^{2\gamma} \left( \frac{r}{\rho} \right)^{n} \left[ \| D^{m+1}u \|_{L^{\infty}(Q_{2r}(z))}^{2} +  E^{2}  \right].
\end{aligned}\end{equation*}
Since $0 < \rho \leq r$ was arbitrary chosen, the lemma follows from Lemma \ref{PMTS400}.
\end{proof}

So with Campanato type embedding, see for instance \cite[Theorem 3.1]{HQLF1} or \cite[Theorem 2.9]{GE1}, we obtain H\"{o}lder continuity of $U$.

\begin{prop}\label{PMTS800}
Suppose that $h \in L^{2}(Q_{2r}(z))$ satisfies 
\begin{equation*}
\int_{Q_{\rho}(y)} | h - (h)_{Q_{\rho}(y)}|^{2} \, dx 
\leq M^{2} \rho^{n+2\gamma} 
\qquad 
\big( y \in Q_{r}(z), \ \rho \in (0,r] \big)
\end{equation*}
for some $\gamma \in (0,1)$. Then  we have that $[h]_{C^{\gamma}(Q_{r}(z))}\leq c M$.
\end{prop}

\sskip

The proof of Lemma \ref{PMTS850} is an obvious application of Proposition \ref{PMTS800}. In Lemma \ref{PMTS850}, the term $cr^{2\gamma}  \sum_{|q'|=m} \big \| U^{q'} \big  \|_{L^{\infty}( Q_{2r}(z) )}^{2} $ exists but we can handle it by making $cr^{2\gamma}$ sufficiently small.(See Lemma \ref{PMTS1000})

\begin{lem}\label{PMTS850}
For any $p' \geq 0'$ with $|p'|=m$ and $Q_{2r}(z) \subset Q_{5}$, we have that
\begin{equation*}\label{}
\big\| D^{m+1}u \big\|_{L^{\infty}(Q_{r}(z))}^{2}
\leq c \bigg[ \mint_{Q_{2r}(z)} \big| D^{m+1}u \big|^{2} \, dx   + r^{2\gamma} \bigg( \| D^{m+1}u \|_{L^{\infty}(Q_{2r}(z))}^{2} + E^{2} \bigg) \bigg]
\end{equation*}
and 
\begin{equation*}\label{}
\big[ U^{p'} \big]_{C^{\gamma}(Q_{r}(z))}^{2} 
\leq c \left[  \frac{1}{r^{2\gamma} }
\mint_{Q_{2r}(z)} \left| U^{p'}  - \big( U^{p'} \big)_{Q_{2r}(z)} \right|^{2} dx 
+  \| D^{m+1}u \|_{L^{\infty}(Q_{2r}(z))}^{2} + E^{2} \right].
\end{equation*}
\end{lem}

\begin{proof}
Fix $p' \geq 0'$ with $|p'|=m$. Choose  $y \in Q_{r}(z)$ and $\rho \in (0,r]$. By Lemma \ref{PMTS700},
\begin{equation*}\begin{aligned}
& \frac{1}{\rho^{2\gamma}}  \mint_{ Q_{\rho}(y) } \left| U^{p'} - \big( U^{p'}  \big)_{Q_{\rho}(y)} \right|^{2}  \, dx \\
& \quad \leq  c \left[
\frac{1}{ r^{ 2\gamma } } \mint_{ Q_{r}(y) } \left| U^{p'}  - \big( U^{p'} \big)_{Q_{r}(y)} \right|^{2} \, dx +  \| D^{m+1}u \|_{L^{\infty}(Q_{r}(y))}^{2} + E^{2}  \right].
\end{aligned}\end{equation*}
Since $y \in Q_{r}(z)$, we have that 
\begin{equation*}\label{}
\mint_{Q_{r}(y)} \left| U^{p'} - \big( U^{p'} \big)_{Q_{r}(y)} \right|^{2} \, dx   \leq c \mint_{Q_{2r}(z)} \left| U^{p'}  - \big( U^{p'}  \big)_{Q_{2r}(z)} \right|^{2} \, dx.
\end{equation*}
Thus
\begin{equation*}\begin{aligned}
& \frac{1}{\rho^{2\gamma}}  \mint_{ Q_{\rho}(y) } \left| U^{p'} - \big( U^{p'}  \big)_{Q_{\rho}(y)} \right|^{2}  \, dx \\
& \quad \leq  c \left[
\frac{1}{ r^{ 2\gamma } } \mint_{ Q_{2r}(z) } \left| U^{p'}  - \big( U^{p'} \big)_{Q_{2r}(z)} \right|^{2} \, dx +  \| D^{m+1}u \|_{L^{\infty}(Q_{2r}(z))}^{2} + E^{2}  \right].
\end{aligned}\end{equation*}
Since $y \in Q_{r}(z)$ and $\rho \in (0,r]$ was arbitrary chosen, by taking 
\begin{equation*}\label{}
M = c \left[  \frac{1}{r^{2\gamma} }
\mint_{Q_{2r}(z)} \left| U^{p'}  - \big( U^{p'} \big)_{Q_{2r}(z)} \right|^{2} \, dx 
+  \| D^{m+1}u \|_{L^{\infty}(Q_{2r}(z))}^{2} + E^{2} \right]
\end{equation*}
in Proposition \ref{PMTS800}, we obtain that 
\begin{equation}\begin{aligned}\label{PMT900}
& \big[ U^{p'} \big]_{C^{\gamma}(Q_{r}(z))}^{2}  \\
& \quad \leq c \left[  \frac{1}{r^{2\gamma} }
\mint_{Q_{2r}(z)} \left| U^{p'}  - \big( U^{p'} \big)_{Q_{2r}(z)} \right|^{2} \, dx 
+  \| D^{m+1}u \|_{L^{\infty}(Q_{2r}(z))}^{2} + E^{2} \right]
\end{aligned}\end{equation}
Then from the following inequality
\begin{equation*}\label{}
\left\| U^{p'} \right\|_{L^{\infty}(Q_{r}(z))} 
\leq \left| \big( U^{p'} \big)_{Q_{r}(z)} \right| + c r^{\gamma} \left[ U^{p'} \right]_{C^{\gamma}(Q_{r}(z))},
\end{equation*}
we get that
\begin{equation}\label{PMT950}
\big\| U^{p'} \big\|_{L^{\infty}(Q_{r}(z))}^{2}
\leq c \bigg[ \mint_{Q_{2r}(z)} \big| U^{p'} \big|^{2} \, dx   + r^{2\gamma} \bigg( \| D^{m+1}u \|_{L^{\infty}(Q_{2r}(z))}^{2} + E^{2} \bigg) \bigg].
\end{equation}
Since $p' \geq 0'$ with $|p'|=m$ was arbitrary chosen, with \eqref{PMT150}, the lemma holds from \eqref{PMT900} and \eqref{PMT950}.
\end{proof}

\begin{lem}\label{PMTS1000}
For any $Q_{2r}(z) \subset Q_{5}$, we have that
\begin{equation*}
\left\| D^{m+1}u   \right\|_{L^{\infty}(Q_{r}(z))}^{2}
\leq c \left[ \mint_{Q_{2r}(z)} \big| D^{m+1}u  \big|^{2} \, dx  +  E^{2}  \right].
\end{equation*}
\end{lem}

\begin{proof}
Fix $r \leq \rho < \tau \leq 2r$ and $p' \geq 0'$ with $|p'|=m$. Then by Lemma \ref{PMTS850},
\begin{equation*}\begin{aligned}\label{}
& \left\| D^{m+1}u \right\|_{L^{\infty} \left( Q_{\frac{\tau-\rho}{2}}(y) \right)}^{2} \\
& \quad \leq c \bigg[ \mint_{Q_{\tau - \rho}(y)} \big| D^{m+1}u  \big|^{2} \, dx 
 + (\tau - \rho)^{\gamma} \left[  \left\| D^{m+1}u   \right\|_{L^{\infty}( Q_{\tau - \rho}(y) )}^{2} +  E^{2} \right] \bigg],
\end{aligned}\end{equation*}
for any $y  \in Q_{\rho}(z)$. For any $y  \in Q_{\rho}(z)$, we have that $Q_{\tau-\rho}(y) \subset Q_{\tau}(z) \subset Q_{2r}(z)$. Thus
\begin{equation*}\begin{aligned}\label{}
& \left\| D^{m+1}u \right\|_{L^{\infty} ( Q_{\rho}(z) ) }^{2} \\
& \quad \leq  c \bigg[ r^{\gamma} \left(  \left\| D^{m+1}u   \right\|_{L^{\infty}( Q_{\tau}(z) )}^{2} +  E^{2} \right)  +  \frac{1}{(\tau-\rho)^{n}} \int_{Q_{2r}(z) } \left| D^{m+1}u \right|^{2} \, dx  
\bigg].
\end{aligned}\end{equation*}
Then there exists a small universal constant $\bar{r} \in (0,3]$ such that if $r \in (0,\bar{r}]$ then
\begin{equation*}\label{}
\left\| D^{m+1}u \right\|_{L^{\infty} ( Q_{\rho}(z) ) }^{2}
\leq \frac{1}{2} \left\| D^{m+1}u \right\|_{L^{\infty} ( Q_{\tau}(z) ) }^{2}
   + c \left[  \frac{ 1 }{(\tau-\rho)^{n}}  \int_{Q_{2r}(z) } \big| D^{m+1}u \big|^{2} \, dx  +  E^{2} \right].
\end{equation*}
Since $r \leq \rho \leq \tau \leq 2r$ was arbitrary chosen the lemma follows by using a technical argument  \cite[Lemma 4.3]{HQLF1} or \cite[Lemma 4.1]{GE1}, we obtain that
\begin{equation*}
\left\| D^{m+1}u   \right\|_{L^{\infty}(Q_{r}(z))}^{2}
\leq c \left[ \mint_{Q_{2r}(z)} \big| D^{m+1}u  \big|^{2} \, dx  +  E^{2}  \right],
\end{equation*}
for any $Q_{2r}(z) \subset Q_{5}$ with $r \in (0,\bar{r}]$. 

Also  there exists a universal constant $l \in \bn$ such that  if $\bar{r} \leq r \leq 7$ then $Q_{r}(z)$ can be covered with $l$ cubes $\left\{ Q_{\bar{r}}(z_{1}), \cdots, Q_{\bar{r}}(z_{l}) \right\}$ satisfying that $Q_{2\bar{r}}(z_{1}), \cdots, Q_{2\bar{r}}(z_{l}) \subset Q_{2r}(z)$. So the lemma even holds when $r \geq \bar{r}$.
\end{proof}

By using covering argument in \cite[Corollary 6.1]{GE1}, we obtain the following result.

\begin{lem}\label{PMTS1100}
For any $Q_{2r}(z) \subset Q_{5}$, we have that
\begin{equation*}\begin{aligned}\label{} 
\left\| D^{m+1}u \right\|_{L^{\infty}( Q_{ 2r } (z) )}^{2} \leq cr^{-2} E^{2}.
\end{aligned}\end{equation*}
\end{lem}

\begin{proof}
We claim that 
\begin{equation}\label{PMT1120}
\left\| D^{m+1}u \right\|_{L^{\infty}( Q_{ r } (z) )}^{2} \leq cr^{-2} E^{2} 
\qquad (Q_{2r}(z) \subset Q_{5}).
\end{equation}
If we show \eqref{PMT1120} then the lemma holds by a covering argument. 

Fix $r \leq \rho < \tau \leq 2r$ and $ y \in Q_{\rho}$. Let $\phi \in C_{c}^{\infty}(Q_{ \frac{ \tau - \rho}{2} }(y))$ be cut-off function with 
\begin{equation*}\label{}
0 \leq \phi \leq  1,
\qquad 
\phi=1 \text{ in } Q_{ \frac{ \tau - \rho}{4} }(y)
\qquad \text{and} \qquad
|D\phi| \leq c(\tau-\rho)^{-1}.
\end{equation*}
Since $Q_{ \tau - \rho }(y) \subset Q_{5}$, from  Lemma \ref{PMTS200} and Lemma \ref{PMTS1000},  we have that
\begin{equation*}\begin{aligned}\label{} 
\left\| D^{m+1}u \right\|_{L^{\infty}( Q_{\frac{\tau - \rho}{2}}(y) )}^{2} 
& \leq c \left[ \mint_{Q_{\frac{ \tau -\rho}{2}}(y) } \left| D^{m+1}u \right|^{2} \, dx + E^{2} \right] \\
& \leq  c \bigg[   (\tau-\rho)^{\gamma} \left\| D^{m+1}u \right\|_{L^{\infty}( Q_{ \tau - \rho} (y) )}^{2} + (\tau-\rho)^{-2} E^{2}  \bigg].
\end{aligned}\end{equation*}
Since $y \in Q_{\rho}(z)$ was arbitrary chosen and $Q_{\tau - \rho}(y) \subset Q_{\tau}(z)$, we find that
\begin{equation*}\begin{aligned}\label{} 
\left\| D^{m+1}u \right\|_{L^{\infty}( Q_{\rho}(z) )}^{2} 
& \leq  c \bigg[   (\tau-\rho)^{\gamma} \left\| D^{m+1}u \right\|_{L^{\infty}( Q_{ \tau} (z) )}^{2} + (\tau-\rho)^{-2} E^{2}  \bigg] .
\end{aligned}\end{equation*}
For a sufficiently small universal constant $\bar{r} \in (0,1]$, if $ r \in (0,\bar{r}]$ then
\begin{equation}\begin{aligned}\label{} 
\left\| D^{m+1}u \right\|_{L^{\infty}( Q_{\rho}(z) )}^{2} 
& \leq  \frac{1}{2} \cdot  \left\| D^{m+1}u \right\|_{L^{\infty}( Q_{ \tau} (z) )}^{2} + c (\tau-\rho)^{-2} E^{2},
\end{aligned}\end{equation}
because of that $r \leq \rho < \tau \leq 2r$. Since $r \leq \rho < \tau \leq 2r$ was arbitrary chosen, by using a technical argument  \cite[Lemma 4.3]{HQLF1} or \cite[Lemma 4.1]{GE1},
\begin{equation}\begin{aligned}\label{PMT1170} 
\left\| D^{m+1}u \right\|_{L^{\infty}( Q_{ r } (z) )}^{2} 
\leq cr^{-2} E^{2}
\qquad 
\left( Q_{2r}(z) \subset Q_{5} \text{ with } r \in (0,\bar{r}] \right).
\end{aligned}\end{equation}
So the claim \eqref{PMT1120} holds when $ r \in (0,\bar{r}]$.

Since $\bar{r} \in (0,1]$ is a universal constant the claim \ref{PMT1120} when $r > \bar{r}$ can be proved by a covering argument with \eqref{PMT1170}  because $\bar{r} \in (0,1]$ is a universal constant.
\end{proof}

We are now ready to prove our main theroems.

\begin{proof}[Proof of Theorem \ref{main theorem of U}]
By using Proposition \ref{prop_compare}, we obtain from Lemma \ref{PMTS850}, Lemma \ref{PMTS1000} and Lemma \ref{PMTS1100} that
\begin{equation*}\begin{aligned}\label{} 
\big\| U^{p'} \big\|_{L^{\infty}(Q_{2r}(z))}^{2} \leq cr^{-2} E^{2},
\end{aligned}\end{equation*}
\begin{equation*}\label{}
\big\| U^{p'} \big\|_{L^{\infty}(Q_{r}(z))}^{2}
\leq c \left[ \mint_{Q_{2r}(z)} \big| U^{p'} \big|^{2} \, dx   + E^{2} \right]
\end{equation*}
and
\begin{equation*}\label{}
\big[ U^{p'} \big]_{C^{\gamma}(Q_{r}(z))}^{2} 
\leq c \left[  \frac{1}{r^{2\gamma} }
\mint_{Q_{2r}(z)} \left| U^{p'}  - \big( U^{p'} \big)_{Q_{2r}(z)} \right|^{2} dx 
+  \sum_{|q'|=m}  \big\| U^{q'} \big \|_{L^{\infty}(Q_{2r}(z))}^{2} + E^{2} \right].
\end{equation*}
for any $p' \geq 0'$ with $|p'|=m$ and $Q_{2r}(z) \subset Q_{5}$. So Theorem \ref{main theorem of U} holds. 
\end{proof}

\begin{proof}[Proof of Theorem \ref{main theorem}]
We find from Theorem \ref{main theorem of U} and Proposition \ref{prop_compare} that
\begin{equation*}\begin{aligned}\label{} 
\frac{1}{|Q_{r}|} \int_{Q_{ r }^{k}(z) }  \left| D^{m+1}u \right|^{2}  \, dx 
\leq cr^{-2} E^{2},
\end{aligned}\end{equation*}
\begin{equation*}\label{}
\left\| D^{m+1}u \right\|_{L^{\infty}(Q_{r}^{k}(z))}^{2}
\leq c \left[ \sum_{l \in K} \frac{1}{|Q_{2r}|} \int_{Q_{2r}^{l}(z)} \big| D^{m+1}u \big|^{2} \, dx   + E^{2} \right]
\end{equation*}
and 
\begin{equation*}\label{}
\big[ D^{m+1}u \big]_{C^{\gamma}(Q_{r}^{k}(z))}^{2} 
\leq \frac{c}{r^{2\gamma}}  \left[  \sum_{l \in K} \frac{1}{|Q_{2r}|}  \int_{Q_{2r}^{l}(z)} \big| D^{m+1}u \big|^{2} \, dx  
+ E^{2}  \right],
\end{equation*}
for any $Q_{2r}(z) \subset Q_{5}$ and $k \in K$. So Theorem \ref{main theorem} holds.
\end{proof}

\subsection*{Acknowledgement}
Y. Kim was supported by the National Research Foundation of Korea (NRF) grant funded by the Korea Government NRF-2020R1C1C1A01013363.

\bibliographystyle{amsplain}

\begin{thebibliography}{20}





\bibitem{KYSP1} Y. Kim, P. Shin, Gradient estimate for linear elliptic systems from composite materials,submitted.

\bibitem{KYSP2} Y. Kim, P. Shin, A geometric result for composite materials with $C^{1,\gamma}$-boundaries, preprint.

\bibitem{BEVM1} E. Bonnetier, M. Vogelius, An elliptic regularity result for a composite medium with``touching'' fibers of circular cross-section. SIAM J. Math. Anal. 31 (2000), no. 3, 651--677. 

\bibitem{CMKDVU1} M. Chipot, D. Kinderlehrer, G. Vergara-Caffarelli, Smoothness of linear laminates, Arch. Rational Mech. Anal. 96 (1986), no. 1, 81--96.

\bibitem{HQLF1} Q. Han, F. Lin, \textit{Elliptic partial differential equations}, Second edition, Courant Lecture Notes in Mathematics,  Courant Institute of Mathematical Sciences, New York, American Mathematical Society, Providence, RI, 2011.

\bibitem{GE1} E. Giusti, \textit{Direct Methods in the Calculus of Variations}, World Scientific, 2003.

\bibitem{JYKY1}  Y. Jang, Y. Kim Global gradient estimates for parabolic systems from composite materials, Calc Varc Partial Differ Equ. (2018) 57(2):57:63.

\bibitem{LYNL1}  Y.Y. Li, L. Nirenberg, Estimates for elliptic systems from composite material. Dedicated to the memory of Jürgen K. Moser, Comm. Pure Appl. Math. 56 (2003), no. 7, 892--925.

\bibitem{LYVM1}  Y.Y. Li, M. Vogelius, Gradient estimates for solutions to divergence form elliptic equations with discontinuous coefficients, Arch. Ration. Mech. Anal. 153 (2000), no. 2, 91--151. 

\bibitem{DHZH1} H. Dong, H. Zhang, On an elliptic equation arising from composite materials, Arch.
Ration. Mech. Anal. 222 (2016), no. 1, 47–89.



\end{thebibliography}

\end{document}